\documentclass[12pt]{article}%
\usepackage[colorlinks=true,citecolor=blue,bookmarks=true]{hyperref}
\usepackage{amsmath}
\usepackage{graphicx}
\usepackage{amsfonts}
\usepackage{amssymb}%
\providecommand{\U}[1]{\protect\rule{.1in}{.1in}}
\newtheorem{theorem}{Theorem} [section]

\newtheorem{corollary}{Corollary}[section]

\newtheorem{definition}{Definition} [section]

\newtheorem{lemma}{Lemma}[section]

\newtheorem{proposition}{Proposition} [section]
\newtheorem{remark}{Remark} [section]

\newenvironment{proof}[1][Proof]{\textbf{#1.} }{\ \rule{1em}{1em}}
\numberwithin{equation}{section}

\begin{document}

\title{Stability of rotating gaseous stars}
\author{Zhiwu Lin\\School of Mathematics\\Georgia Institute of Technology\\Atlanta, GA 30332, USA\\Yucong Wang\\School of Mathematical Sciences\\Xiamen University\\Xiamen, 361005, China. \\School of Mathematical Science\\Peking University\\Beijing, 100871, China}
\date{}
\maketitle

\begin{abstract}
We consider stability of rotating gaseous stars modeled by the Euler-Poisson
system with general equation of states. When the angular velocity of the star
is Rayleigh stable, we proved a sharp stability criterion for axi-symmetric
perturbations. We also obtained estimates for the number of unstable modes and
exponential trichotomy for the linearized Euler-Poisson system. By using this
stability criterion, we proved that for a family of slowly rotating stars
parameterized by the center density with fixed angular velocity, the turning
point principle is not true. That is, unlike the case of non-rotating stars,
the change of stability of the rotating stars does not occur at extrema points
of the total mass. By contrast, we proved that the turning point principle is
true for the family of slowly rotating stars with fixed angular momentum
distribution. When the angular velocity is Rayleigh unstable, we proved linear
instability of rotating stars. Moreover, we gave a complete description of the
spectra and sharp growth estimates for the linearized Euler-Poisson equation.

\end{abstract}

\section{Introduction}

Consider a self-gravitating gaseous star modeled by the Euler-Poisson system
of compressible fluids
\begin{equation}%
\begin{cases}
\rho_{t}+\nabla\cdot\left(  \rho v\right)  =0,\\
\rho\left(  v_{t}+v\cdot\nabla v\right)  +\nabla p=-\rho\nabla V,\\
\Delta V=4\pi\rho,\ \lim_{\left\vert x\right\vert \rightarrow\infty}V\left(
t,x\right)  =0,
\end{cases}
\label{EP}%
\end{equation}
where $x\in\mathbb{R}^{3},\ t>0,\ \rho\left(  x,t\right)  \geq0$ is the
density, $v\left(  x,t\right)  \in\mathbb{R}^{3}$ is the velocity, $p=P(\rho)$
is the pressure, and $V$ is the self-consistent gravitational potential.
Assume $P(\rho)$ satisfies:
\begin{equation}
P(s)=C^{1}(0,\infty),\quad P^{\prime}>0, \label{P1}%
\end{equation}
and there exists $\gamma_{0}\in(\frac{6}{5},2)$ such that
\begin{equation}
\lim_{s\rightarrow0+}s^{1-\gamma_{0}}P^{\prime}(s)=K>0. \label{P2}%
\end{equation}
The assumption (\ref{P2}) implies that the pressure $P(\rho)\approx
K\rho^{\gamma_{0}}$ for $\rho$ near 0. We note that $\gamma_{0}=\frac{5}{3}$
for realistic stars.

The Euler-Poisson system (\ref{EP}) has many steady solutions. The simplest
one is the spherically symmetric non-rotating star with $\left(  \rho
_{0},v_{0}\right)  =\left(  \rho_{0}\left(  \left\vert x\right\vert \right)
,0\right)  $. We refer to \cite{LZ2019} and references therein for the
existence and stability of non-rotating stars. A turning point principle (TPP)
was shown in \cite{LZ2019} that the stability of the non-rotating stars is
entirely determined by the mass-radius curve parameterized by the center
density. In particular, the stability of a non-rotating star can only change
at extrema (i.e. local maximum or minimum points) of the total mass.

We consider axi-symmetric rotating stars of the form
\[
\left(  \rho_{0},\vec{v}_{0}\right)  =\left(  \rho_{0}\left(  r,z\right)
,r\omega_{0}\left(  r\right)  \mathbf{e}_{\theta}\right)  ,
\]
where $\left(  r,\theta,z\right)  $ are the cylindrical coordinates,
$\omega_{0}\left(  r\right)  $ is the angular velocity and $\left(
\mathbf{e}_{r},\mathbf{e}_{\theta},\mathbf{e}_{z}\right)  $ denote unit
vectors along $r,\theta,z$ directions. We note that for barotropic equation of
states $P=P\left(  \rho\right)  $, it was known as Poincar\'{e}-Wavre theorem
(\cite[Section 4.3]{TJ1978}) that the angular velocity must be independent of
$z$. The existence and stability of rotating stars is a classical problem in
astrophysics. For homogeneous (i.e. constant density) rotating stars, it had
been extensively investigated since the work of Maclaurin in 1740s, by many
people including Dirichlet, Jacobi, Riemann, Poincar\'{e} and Chandrasekhar
etc. We refer to the books \cite{Chandra69,JWS2013} for history and results on
this topic. The compressible rotating stars are much less understood. From
1920s, Lichtenstein initiated a mathematical study of compressible rotating
stars, which was summarized in his monograph (\cite{Lich1933}). In particular,
he showed the existence of slowly rotating stars near non-rotating stars by
implicit function theorem. See also \cite{H1994,JJ2017,JJMT2019,JJ2019,SW2017}
for related results. The existence of rotating stars can also be established
by variational methods
(\cite{AGBR1971,CF1980,CL1994,FT1980,FT1981,LYY1991,LS2004,LS2009}), or global
bifurcation theory (\cite{AG1991,SW2019,SW2019R}). Compared with the existence
theory, there has been relatively few rigorous works on the stability of
rotating stars. In this paper, we consider the stability of rotating stars
under axi-symmetric perturbations. There are two natural questions to address:
1) Does TPP still hold for a family of rotating stars? 2) How does the
rotation affect the stability (instability) of rotating stars?

The answers to these two questions have been disputed in the astrophysical
literature. Bisnovaty-Kogan and Blinnikov \cite{BB1974} suggested that for a
family of rotating stars with fixed angular momentum distribution per unit
mass and parameterized by the center density $\mu$, TPP is true (i.e.
stability changes at the extrema of the total mass). They used heuristic
arguments (so called static method) as in the non-rotating case. Such
arguments suppose that at the transition point of stability, there must exist
a zero frequency mode which can only be obtained by infinitesimally
transforming equilibrium configurations near the given one, without changing
the total mass $M\left(  \mu\right)  $. Hence, the transition point is a
critical point of the total mass (i.e. $M^{\prime}\left(  \mu\right)  =0$). It
is reasonable to study the family of rotating stars with fixed angular
momentum distribution, which is invariant under Euler-Poisson dynamics. In
\cite{BB1974}, they also considered a family of rigidly rotating stars (i.e.
$\omega_{0}$ is constant) for a special equation of state similar to white
dwarf stars. By embedding each rigidly rotating star into a family with the
same angular momentum distribution and with some numerical help, it was found
that the transition of stability is not the extrema of mass. In
\cite{Stahler83}, for a family of rotating stars with fixed rotational
parameter (i.e. the ratio of rotational energy to gravitational energy),
similar arguments as in \cite{BB1974} were used to indicate that TPP is true
for this family and their numerical results suggested that instability occurs
beyond the first mass extrema. However, up to date there is no rigorous proof
or disproof of TPP for different families of rotating stars.

The issue that whether rotation can have a stabilizing effect on rotating
stars has long been in debate. For a long time, it was believed that rotation
is stabilizing for any angular velocity profile. This conviction was based on
conclusions drawn from perturbation analysis near neutral modes of
non-rotating stars, which was done by Ledoux \cite{Ledoux1945} for rigidly
rotating stars and by Lebovitz \cite{LNR1970} for general angular velocities.
However, the later works of Sidorov \cite{Sidorov1982,Sidorov1981} and
K\"{a}hler \cite{Kahler1994} showed that rotating could be destabilizing.
Hazlehurst \cite{Hazelhurst1994} argued that the advocates of destabilization
of rotation had used an argument that is open to criticism and disagreed that
rotation could be destabilizing.

In this paper, we answer above two questions in a rigorous way. To state our
results more precisely, we introduce some notations. Let $(\rho_{0}\left(
r,z\right)  ,\vec{v_{0}}=r\omega_{0}\left(  r\right)  \mathbf{e}_{\theta})$ be
an axi-symmetric rotating star solution of \eqref{EP}. The support of
$\rho_{0}$ is denoted by $\Omega$, which is an axi-symmetric bounded domain.
The rotating star solutions satisfy
\begin{equation}
\vec{v}_{0}\cdot\nabla\vec{v}_{0}+\nabla\Phi^{\prime}(\rho_{0})+\nabla
V=0\text{ in }\Omega,\label{ift1}%
\end{equation}%
\begin{equation}
V=-|x|^{-1}\ast\rho_{0}\text{ in }\mathbb{R}^{3},\label{ift3}%
\end{equation}
Equivalently,
\begin{equation}
\Phi^{\prime}(\rho_{0})-|x|^{-1}\ast\rho_{0}-\int_{0}^{r}\omega_{0}%
^{2}(s)s\ ds+c_{0}=0\text{ in }\Omega,\label{eqn-steady-state}%
\end{equation}
where $c_{0}>0$ is a constant.

Let $R_{0}$ be the maximum of $r$ such that $(r,z)\in\Omega$. We assume
$\omega_{0}\in C^{1}[0,R_{0}]$, $\partial\Omega$ is $C^{2}\ $with positive
curvature\ near $\left(  R_{0},0\right)  $, and for any $(r,z)$ near
$\partial\Omega$
\begin{equation}
\rho_{0}(r,z)\thickapprox\text{dist}((r,z),\partial\Omega)^{\frac
{1}{\mathbb{\gamma}_{0}-1}},\label{rhonearboundary}%
\end{equation}
which are satisfied for slowly rotating stars near non-rotating stars as
constructed in (\cite{H1994,JJ2017,JJ2019,SW2017}). Let $X=L_{\Phi
^{\prime\prime}(\rho_{0})}^{2}\times L_{\rho_{0}}^{2}$ and $Y=\left(
L_{\rho_{0}}^{2}\right)  ^{2}$, where $L_{\Phi^{\prime\prime}(\rho_{0})}^{2}$
and $L_{\rho_{0}}^{2}$ are axi-symmetric weighted spaces in $\Omega$ with
weights $\Phi^{\prime\prime}(\rho_{0})$ and $\rho_{0}$. The enthalpy
$\Phi(\rho)>0$ is defined by
\[
\Phi(0)=\Phi^{\prime}(0)=0,\quad\Phi^{\prime}(\rho)=\int_{0}^{\rho}%
\frac{P^{\prime}(s)}{s}ds.
\]
Denote $\mathbf{X}:=X\times Y$. Define the Rayleigh discriminant
$\Upsilon(r)=\frac{\partial_{r}(\omega_{0}^{2}r^{4})}{r^{3}}$.

For Rayleigh stable angular velocity $\omega_{0}$ satisfying $\Upsilon(r)>0$
for $r\in\lbrack0,R_{0}]$, the linearization of the axi-symmetric
Euler-Poisson equations at $(\rho_{0},\vec{v_{0}})$ can be written in a
Hamiltonian form
\begin{equation}
\frac{d}{dt}%
\begin{pmatrix}
u_{1}\\
u_{2}%
\end{pmatrix}
=\mathbf{J}\mathbf{L}%
\begin{pmatrix}
u_{1}\\
u_{2}%
\end{pmatrix}
,\label{hamiltonian-RS}%
\end{equation}
where $u_{1}=(\rho,v_{\theta})$ and $u_{2}=(v_{r},v_{z})$, and $\rho,\left(
v_{r},v_{\theta},v_{z}\right)  \ $are perturbations of density and $\left(
r,\theta,z\right)  $-components of velocity respectively. The operators
\begin{equation}
\mathbf{J}:=%
\begin{pmatrix}
0, & B\\
-B^{\prime}, & 0
\end{pmatrix}
:\mathbf{X}^{\ast}\rightarrow\mathbf{X},\quad\mathbf{L}:=%
\begin{pmatrix}
\mathbb{L}, & 0\\
0, & A
\end{pmatrix}
:\mathbf{X}\rightarrow\mathbf{X}^{\ast},\label{defn-JL}%
\end{equation}
are off-diagonal anti-self-dual and diagonal self-dual operators respectively,
where
\begin{equation}
\mathbb{L}=%
\begin{pmatrix}
L & 0\\
0 & A_{1}%
\end{pmatrix}
:X\rightarrow X^{\ast},\label{defn-cal-L}%
\end{equation}
with
\begin{equation}
L=\Phi^{\prime\prime}(\rho_{0})-4\pi(-\Delta)^{-1},\label{defn-L}%
\end{equation}%
\begin{equation}
B=(B_{1},B_{2})^{T},\ B_{1}=-\nabla\cdot,\ B_{2}=-\frac{\partial_{r}%
(\omega_{0}r^{2})}{r\rho_{0}}\mathbf{e}_{r},\label{defn-B}%
\end{equation}
$A=\rho_{0}$, and $A_{1}=\frac{4\omega_{0}^{2}r^{3}\rho_{0}}{\partial
_{r}(\omega_{0}^{2}r^{4})}=\frac{4\omega_{0}^{2}\rho_{0}}{\Upsilon(r)}$. More
precise definition and properties of these operators can be found in Section
\ref{HamiltonianEPassumption}.

Our main result for the Rayleigh stable case is the following.

\begin{theorem}
\label{Th: rayleigh stable}Assume $\omega_{0}\in C^{1}[0,R_{0}]$,
$\Upsilon(r)>0$, (\ref{rhonearboundary}), $\partial\Omega$ is $C^{2}$ and has
positive curvature near $(R_{0},0)$ . Then the operator $\mathbf{JL}$ defined
by (\ref{defn-JL}) generates a $C^{0}$ group $e^{t\mathbf{JL}}$ of bounded
linear operators on $\mathbf{X}=X\times Y$ and there exists a decomposition%
\[
\mathbf{X}=E^{u}\oplus E^{c}\oplus E^{s},\quad
\]
of closed subspaces $E^{u,s,c}$ satisfying the following properties:

i) $E^{c},E^{u},E^{s}$ are invariant under $e^{t\mathbf{JL}}$.

ii) $E^{u}\left(  E^{s}\right)  $ only consists of eigenvectors corresponding
to positive (negative) eigenvalues of $\mathbf{JL}$ and
\[
\dim E^{u}=\dim E^{s}=n^{-}\left(  \mathbb{L}|_{\overline{R\left(  B\right)
}}\right)  =n^{-}\left(  \mathcal{K}|_{R\left(  B_{1}\right)  }\right)  ,
\]
where $\langle\mathcal{K}\cdot,\cdot\rangle$ is a bounded bilinear quadratic
form on $L_{\Phi^{\prime\prime}(\rho_{0})}^{2}\ $defined by
\begin{equation}
\langle\mathcal{K}\delta\rho,\delta\rho\rangle=\langle L\delta\rho,\delta
\rho\rangle+2\pi\int_{0}^{R_{0}}\Upsilon(r)\frac{\left(  \int_{0}^{r}%
s\int_{-\infty}^{+\infty}\delta\rho(s,z)dzds\right)  ^{2}}{r\int_{-\infty
}^{+\infty}\rho_{0}({r},z)dz}dr, \label{defn-cal-k}%
\end{equation}
for any $\delta\rho\in L_{\Phi^{\prime\prime}(\rho_{0})}^{2}$ and
$n^{-}\left(  \mathcal{K}|_{R\left(  B_{1}\right)  }\right)  $ denotes the
number of negative modes of $\left\langle \mathcal{K}\cdot,\cdot\right\rangle
$ restricted to the subspace
\begin{equation}
R\left(  B_{1}\right)  =\left\{  \delta\rho\in L_{\Phi^{\prime\prime}(\rho
_{0})}^{2}\ |\ \int\delta\rho dx=0\right\}  . \label{defn-R-B1}%
\end{equation}

iii) The exponential trichotomy is true in the space $\mathbf{X}$ in the sense
of (\ref{estimate-stable-unstable}) and (\ref{estimate-center}).
\end{theorem}

\begin{corollary}
\label{cor-stability criterion}Under the assumptions of Theorem
\ref{Th: rayleigh stable}, the rotating star solution $(\rho_{0},\vec{v}_{0})$
is spectrally stable to axi-symmetric perturbations if and only if
\[
\langle\mathcal{K}\delta\rho,\delta\rho\rangle\geq0,
\]
for all $\delta\rho\in L_{\Phi^{\prime\prime}(\rho_{0})}^{2}$ with
$\int_{\mathbb{R}^{3}}\delta\rho dx=0$.
\end{corollary}

Theorem \ref{Th: rayleigh stable} gives not only a sharp stability criteria
for rotating stars with Rayleigh stable angular velocity, but also more
detailed information on the spectra of the linearized Euler-Poisson operator
and exponential trichotomy estimates for the linearized Euler-Poisson system.
These will be useful for the future study of nonlinear dynamics near unstable
rotating stars, particularly, the construction of invariant (stable, unstable
and center) manifolds for the nonlinear Euler-Poisson system.

The sharp stability criterion in Corollary \ref{cor-stability criterion} is
used to study the stability of two families of slowly rotating stars. For the
first family of slowly rotating stars with fixed Rayleigh stable angular
velocity and parameterized by the center density, we show that TPP is not true
and the transition of stability does not occur at the first mass extrema. More
precisely, for fixed $\kappa\omega_{0}\left(  r\right)  \in C^{1,\beta}$, for
some $\beta\in(0,1)$, satisfying $\Upsilon(r)>0$ and $\kappa$ small enough, by
implicit function theorem as in \cite{H1994,JJMT2019,SW2017}, there exists a
family of slowly rotating stars $\left(  \rho_{\mu,\kappa},\kappa r\omega
_{0}\left(  r\right)  \mathbf{e}_{\theta}\right)  $ parameterized by the
center density $\mu$. We show that the transition of stability for this family
is not at the first extrema of the total mass $M_{\mu,\kappa}$. In particular,
when $\gamma_{0}>\frac{4}{3}$, the slowly rotating stars are stable for small
center density and remain stable slightly beyond the first mass maximum. This
is consistent with the numerical evidence in \cite{BB1974} (Figure 10, p. 400)
for the example of rigidly rotating stars and an equation of state with
$\gamma_{0}=\frac{5}{3}$. It shows that Rayleigh stable rotation is indeed
stabilizing for rotating stars. By contrast, for the second family of slowly
rotating stars with fixed monotone increasing angular momentum distribution
(equivalently Rayleigh stable angular velocity), we show that TPP is indeed
true. More precisely, for fixed $j\left(  p,q\right)  \in C^{1,\beta}\left(
\mathbb{R}^{+}\times\mathbb{R}^{+}\right)  $ satisfying $\partial_{p}%
(j^{2}\left(  p,q\right)  )>0$, $j(0,q)=\partial_{p}j(0,q)=0$ and
$\varepsilon$ sufficiently small, there exists a family of slowly rotating
stars $\left(  \rho_{\mu,\varepsilon},\frac{\varepsilon}{r}j\left(
m_{\rho_{\mu,\varepsilon}},M_{\mu,\varepsilon}\right)  \mathbf{e}_{\theta
}\right)  $ parameterized by the center density $\mu$, where
\[
m_{\rho_{\mu,\varepsilon}}(r)=\int_{0}^{r}s\int_{-\infty}^{\infty}\rho
_{\mu,\varepsilon}(s,z)dsdz\
\]
is the mass distribution in the cylinder, and $M_{\mu,\varepsilon}$ is the
total mass. We show that the transition of stability for this family of
rotating stars exactly occurs at the first extrema of the total mass
$M_{\mu,\varepsilon}$. This not only confirms the claim in \cite{BB1974} based
on heuristic arguments when $j\left(  m,M\right)  =\frac{1}{M}j(\frac{m}{M})$,
but also can apply to other examples studied in the literature,
including$\ j\left(  m,M\right)  =j\left(  m\right)  $ (see
\cite{AGBR1971,JJMT2019,LS2004,LTSJ2008}) and $j\left(  m,M\right)
=j(\frac{m}{M})$ (see \cite{OM1968}).

The issue of TPP is also not so clear for relativistic rotating stars. For
relativistic stars, TPP was shown for the secular stability of a family of
rigidly rotating stars (\cite{FI1988}), while numerical results in
\cite{TRY2011} indicated that the transition of dynamic instability does not
occur at the mass maximum (i.e. TPP is not true) for such a family. Our
approach for the Newtonian case might be useful for studying the relativistic case.

For the Rayleigh stable case, the stability of rotating stars is studied by
using the separable Hamiltonian framework as in the non-rotating stars
(\cite{LZ2019}). However, there are fundamental differences between these two
cases. For the non-rotating stars, the stability condition is reduced to find
$n^{-}\left(  L|_{R\left(  B_{1}\right)  }\right)  $, that is, the number of
negative modes of $\left\langle L\cdot,\cdot\right\rangle $ restricted to
$R\left(  B_{1}\right)  $, where $L$ and $R\left(  B_{1}\right)  $ are defined
in (\ref{defn-L}) and (\ref{defn-R-B1}) respectively. We note that the
dynamically accessible space $R\left(  B_{1}\right)  $ (for density
perturbation) is one co-dimensional with only the mass constraint. For the
rotating stars, by using the separable Hamiltonian formulation
(\ref{hamiltonian-RS}), the stability is reduced to find $n^{-}\left(
\mathbb{L}|_{\overline{R\left(  B\right)  }}\right)  $, where $\mathbb{L},B$
are defined in (\ref{defn-cal-L}) and (\ref{defn-B}) respectively. Here, the
dynamically accessible space $\overline{R\left(  B\right)  }$ (for density and
$\theta$-component of velocity) is infinite co-dimensional, which corresponds
to perturbations preserving infinitely many generalized total angular momentum
(\ref{defn-general-momentum}) in the first order. It is hard to compute the
negative modes of $\left\langle \mathbb{L\cdot},\mathbb{\cdot}\right\rangle $
with such infinitely many constraints. A key point in our proof is to find a
reduced functional $\mathcal{K}$ defined in (\ref{defn-cal-k}) for density
perturbation such that $n^{-}\left(  \mathbb{L}|_{\overline{R\left(  B\right)
}}\right)  =n^{-}\left(  \mathcal{K}|_{R\left(  B_{1}\right)  }\right)  $,
where $R\left(  B_{1}\right)  $ denotes the density perturbations preserving
the mass as in the non-rotating case. Therefore, the computation of negative
modes of $\mathbb{L}|_{\overline{R\left(  B\right)  }}$ with infinitely many
constraints is reduced to study $\mathcal{K}|_{R\left(  B_{1}\right)  }$ with
only one mass constraint. This reduced stability criterion in terms of
$\mathcal{K}|_{R\left(  B_{1}\right)  }$ is crucial to prove or disprove TPP
for different families of rotating stars.

Next we consider rotating stars with Rayleigh unstable angular velocity
$\omega_{0}\left(  r\right)  $. That is, there exists a point $r_{0}%
\in(0,R_{0})$ such that $\Upsilon(r_{0})=\frac{\partial_{r}(\omega_{0}%
^{2}r^{4})}{r^{3}}\big|_{r=r_{0}}<0$. In this case, we cannot write the
linearized Euler-Poisson system as a separable linear Hamiltonian PDEs since
$A_{1}=\frac{4\omega_{0}^{2}r^{3}\rho_{0}}{\partial_{r}(\omega_{0}^{2}r^{4})}$
is not defined at $r_{0}$. Instead, we use the following second order system
for $u_{2}=(v_{r},v_{z})$%
\begin{equation}
\partial_{tt}u_{2}=-(\mathbb{L}_{1}+\mathbb{L}_{2})u_{2}:=-\mathbb{\tilde{L}%
}u_{2},\label{2nd order-rayleigh unstable}%
\end{equation}
where $\mathbb{\tilde{L}=L}_{1}+\mathbb{L}_{2},$
\[
\mathbb{L}_{1}u_{2}=\nabla\lbrack\Phi^{\prime\prime}(\rho_{0})(\nabla
\cdot(\rho_{0}u_{2}))-4\pi(-\Delta)^{-1}(\nabla\cdot(\rho_{0}u_{2})],
\]%
\[
\mathbb{L}_{2}u_{2}=%
\begin{pmatrix}
\Upsilon(r)v_{r}\\
0
\end{pmatrix}
,
\]
are self-adjoint operators on $Y$. The following properties of the spectra of
$\mathbb{\tilde{L}}$ are obtained in Proposition \ref{prop-spectra-l-tilde}:
i) $\sigma_{ess}(\tilde{\mathbb{L}})=range(\Upsilon(r))=[-a,b],\ $where
$a>0,b\geq0$; ii) There are finitely many negative eigenvalues and infinitely
many positive eigenvalues outside the interval $[-a,b]$. In particular, the
infimum of $\sigma(\tilde{\mathbb{L}})$ is negative, which might correspond to
either discrete or continuous spectrum.

Define the space
\[
Z=\left\{  u_{2}\in Y\ |\ \nabla\cdot(\rho_{0}u_{2})\in L_{\Phi^{\prime\prime
}(\rho_{0})}^{2}\right\}  ,
\]
with the norm
\begin{equation}
\left\Vert u_{2}\right\Vert _{Z}=\left\Vert u_{2}\right\Vert _{Y}+\left\Vert
\nabla\cdot(\rho_{0}u_{2})\right\Vert _{L_{\Phi^{\prime\prime}(\rho_{0})}^{2}%
}. \label{norm-Z}%
\end{equation}

\begin{theorem}
\label{th: rayleigh unstable}

Assume $\omega_{0}\in C^{1}[0,R_{0}]$, (\ref{rhonearboundary}) and $\inf
_{r\in\lbrack0,R_{0}]}\Upsilon(r)<0$. Let $\eta_{0}\leq-a$ be the minimum of
$\lambda\in\sigma(\tilde{\mathbb{L}})$. Then we have:

i) Equation (\ref{2nd order-rayleigh unstable}) defines a $C^{0}$ group
$T(t)$, $t\in\mathbf{R}$, on $Z\times Y$. There exists $C>0$ such that for any
$\left(  u_{2}\left(  0\right)  ,u_{2t}\left(  0\right)  \right)  \in Z\times
Y$,
\begin{equation}
\left\Vert u_{2}\left(  t\right)  \right\Vert _{Z}+\left\Vert u_{2t}\left(
t\right)  \right\Vert _{Y}\leq Ce^{\sqrt{-\eta_{0}}t}\left(  \left\Vert
u_{2}\left(  0\right)  \right\Vert _{Z}+\left\Vert u_{2t}\left(  0\right)
\right\Vert _{Y}\right)  ,\ \forall t>0. \label{estimate-upper-growth}%
\end{equation}
The flow $T(t)$ conserves the total energy
\begin{equation}
E(u_{2},u_{2t})=\left\Vert u_{2t}\right\Vert _{Y}^{2}+\langle\mathbb{\tilde
{L}}u_{2},u_{2}\rangle. \label{energy}%
\end{equation}

ii) For any $\varepsilon>0$, there exists initial data $u_{2}^{\varepsilon
}\left(  0\right)  \in Z,u_{2t}^{\varepsilon}\left(  0\right)  =0$ such that
\begin{equation}
\left\Vert u_{2}^{\varepsilon}\left(  t\right)  \right\Vert _{Y}\gtrsim
e^{\sqrt{-\eta_{0}+\varepsilon}t}\left\Vert u_{2}^{\varepsilon}\left(
0\right)  \right\Vert _{Z},\ \forall t>0. \label{estimate-lower-growth}%
\end{equation}

\end{theorem}

The above theorem shows that rotating stars with Rayleigh unstable angular
velocity are always linearly unstable. The maximal growth rate is obtained
either by a discrete eigenvalue beyond the range of $\Upsilon(r)$ or by
unstable continuous spectrum due to Rayleigh instability (i.e. negative
$\Upsilon(r)$). In \cite{LNR1970}, it was shown that for slowly rotating stars
with any angular velocity profile, discrete unstable modes cannot be perturbed
from neutral modes of non-rotating stars. However, the unstable continuous
spectrum was not considered there.

We briefly mention some recent mathematical works on the stability of rotating
gaseous stars. The conditional Lyapunov stability of some rotating star
constructed by variational methods had been obtained by Luo and Smoller
\cite{LS2004,LTSJ2008,LS2009,LTSJ2011} under Rayleigh stability assumption,
also called S\"{o}lberg stability criterion in their works.

The paper is organized as follows. In Section 2, we study rotating stars with
Rayleigh stable angular velocity and prove the sharp stability criterion. In
Section 3, we use the stability criterion to prove/disprove TPP for two
families of slowly rotating stars. In Section 4, we prove linear instability
of rotating stars with Rayleigh unstable angular velocity.

Throughout this paper, for $a,b>0\ $we use $a\lesssim b$ to denote the
estimate $a\leq Cb$ for some constant $C$ independent of $a,b,$,
$a\thickapprox b$ to denote the estimate $C_{1}a\leq b\leq C_{2}b$ for some
constants $C_{1},C_{2}>0$, and $a\sim b$ to denote $|a-b|<\epsilon$ for some
$\epsilon>0$ small enough.

\section{\label{HandM}Stability criterion for Rayleigh Stable case}

In this section, we consider rotating stars with Rayleigh stable angular
velocity profiles. The linearized Euler-Poisson system is studied by using a
framework of separable Hamiltonian systems in \cite{LZ2019}. First, we give a
summary of the abstract theory in \cite{LZ2019}.

\subsection{\label{section-abstract}Separable Linear Hamiltonian PDEs}

Consider a linear Hamiltonian PDEs of the separable form%
\begin{equation}
\partial_{t}\left(
\begin{array}
[c]{c}%
u\\
v
\end{array}
\right)  =\left(
\begin{array}
[c]{cc}%
0 & B\\
-B^{\prime} & 0
\end{array}
\right)  \left(
\begin{array}
[c]{cc}%
L & 0\\
0 & A
\end{array}
\right)  \left(
\begin{array}
[c]{c}%
u\\
v
\end{array}
\right)  =\mathbf{JL}\left(
\begin{array}
[c]{c}%
u\\
v
\end{array}
\right)  , \label{hamiltonian-separated}%
\end{equation}
where $u\in X,\ v\in Y$ and $X,Y$ are real Hilbert spaces. We briefly describe
the results in \cite{LZ2019} about general separable Hamiltonian PDEs
(\ref{hamiltonian-separated}). The triple $\left(  L,A,B\right)  $ is assumed
to satisfy assumptions:

\begin{enumerate}
\item[(\textbf{G1})] The operator $B:Y^{\ast}\supset D(B)\rightarrow X$ and
its dual operator $B^{\prime}:X^{\ast}\supset D(B^{\prime})\rightarrow Y\ $are
densely defined and closed (and thus $B^{\prime\prime}=B$).

\item[(\textbf{G2})] The operator $A:Y\rightarrow Y^{\ast}$ is bounded and
self-dual (i.e. $A^{\prime}=A$ and thus $\left\langle Au,v\right\rangle $ is a
bounded symmetric bilinear form on $Y$). Moreover, there exist $\delta>0$ such
that
\[
\langle Au,u\rangle\geq\delta\left\Vert u\right\Vert _{Y}^{2},\;\forall u\in
Y.
\]

\item[(\textbf{G3})] The operator $L:X\rightarrow X^{\ast}$ is bounded and
self-dual (i.e. $L^{\prime}=L$ \textit{etc.}) and there exists a decomposition
of $X$ into the direct sum of three closed subspaces
\[
X=X_{-}\oplus\ker L\oplus X_{+},\ \dim\ker L<\infty,\ \ n^{-}(L)\triangleq\dim
X_{-}<\infty,
\]
satisfying

\begin{enumerate}
\item[(\textbf{G3.a})] $\left\langle Lu,u\right\rangle <0$ for all $u\in
X_{-}\backslash\{0\}$;

\item[(\textbf{G3.b})] there exists $\delta>0$ such that
\[
\left\langle Lu,u\right\rangle \geq\delta\left\Vert u\right\Vert ^{2}\ ,\text{
for any }u\in X_{+}.
\]

\end{enumerate}
\end{enumerate}

We note that the assumptions $\dim\ker L<\infty$ and $A>0\ $can be relaxed
(see \cite{LZ2019}). But these simplified assumptions are enough for the
applications to Euler-Poisson system studied in this section under the
Rayleigh stability assumption (i.e. $\Upsilon(r)>0$ for all $r\in
\lbrack0,R_{0}]$). If the Rayleigh unstable assumption holds (i.e.
$\Upsilon(r_{0})<0$ for some $r_{0}\in\lbrack0,R_{0}]$), then $n^{-}%
(L)=\infty$ and we will discuss this in Section \ref{Rayleigh-instability}.

\begin{theorem}
\label{T:abstract} \cite{LZ2019}Assume (\textbf{G1-3}) for
(\ref{hamiltonian-separated}). The operator $\mathbf{JL}$ generates a $C^{0}$
group $e^{t\mathbf{JL}}$ of bounded linear operators on $\mathbf{X}=X\times Y$
and there exists a decomposition%
\[
\mathbf{X}=E^{u}\oplus E^{c}\oplus E^{s},\quad
\]
of closed subspaces $E^{u,s,c}$ with the following properties:

i) $E^{c},E^{u},E^{s}$ are invariant under $e^{t\mathbf{JL}}$.

ii) $E^{u}\left(  E^{s}\right)  $ only consists of eigenvectors corresponding
to negative (positive) eigenvalues of $\mathbf{JL}$ and
\[
\dim E^{u}=\dim E^{s}=n^{-}\left(  L|_{\overline{R\left(  B\right)  }}\right)
,
\]
where $n^{-}\left(  L|_{\overline{R\left(  B\right)  }}\right)  $ denotes the
number of negative modes of
$\left\langle L\cdot,\cdot\right\rangle |_{\overline{R\left(  B\right)  }}$.
If $n^{-}\left(  L|_{\overline{R\left(  B\right)  }}\right)  >0$,
then there exists $M>0$ such that
\begin{equation}
\left\vert e^{t\mathbf{JL}}|_{E^{s}}\right\vert \leq Me^{-\lambda_{u}%
t},\;t\geq0;\quad\left\vert e^{t\mathbf{JL}}|_{E^{u}}\right\vert \leq
Me^{\lambda_{u}t},\;t\leq0, \label{estimate-stable-unstable}%
\end{equation}
where $\lambda_{u}=\min\{\lambda\mid\lambda\in\sigma(\mathbf{JL}|_{E^{u}%
})\}>0$.

iii) The quadratic form $\left\langle \mathbf{L}\cdot,\cdot\right\rangle
$ vanishes on $E^{u,s}$, i.e. $\langle\mathbf{L}\mathbf{u},\mathbf{u}%
\rangle=0$ for all $\mathbf{u}\in E^{u,s}$, but is non-degenerate on
$E^{u}\oplus E^{s}$, and
\[
E^{c}=\left\{  \mathbf{u}\in\mathbf{X}\mid\left\langle \mathbf{\mathbf{L}%
u,v}\right\rangle =0,\ \forall\ \mathbf{v}\in E^{s}\oplus E^{u}\right\}  .
\]
There exists $M>0$ such that
\begin{equation}
|e^{t\mathbf{J}\mathbf{L}}|_{E^{c}}|\leq M(1+t^{2}),\text{ for all }%
t\in\mathbb{R}. \label{estimate-center}%
\end{equation}

iv) Suppose $\left\langle L\cdot,\cdot\right\rangle $ is non-degenerate on
$\overline{R\left(  B\right)  }$, then $|e^{t\mathbf{JL}}|_{E^{c}}|\leq M$ for
some $M>0$. Namely, there is Lyapunov stability on the center space $E^{c}$.
\end{theorem}

\begin{remark}
Above theorem shows that the solutions of (\ref{Lep}) are spectrally stable if
and only if $L|_{\overline{R\left(  B\right)  }}\geq0$. Moreover,
$n^{-}\left(  L|_{\overline{R\left(  B\right)  }}\right)  $ equals to the
number of unstable modes. The exponential trichotomy estimates
(\ref{estimate-stable-unstable})-(\ref{estimate-center}) are important in the
study of nonlinear dynamics near an unstable steady state, such as the proof
of nonlinear instability or the construction of invariant (stable, unstable
and center) manifolds. The exponential trichotomy can be lifted to more
regular spaces if the spaces $E^{u,s}$ have higher regularity. We refer to
Theorem 2.2 in \cite{lin-zeng-hamiltonian} for more precise statements.
\end{remark}

\subsection{Hamiltonian formulation of linearized EP system}

\label{HamiltonianEPassumption}

Consider an axi-symmetric rotating star solution $(\rho_{0}\left(  r,z\right)
,\vec{v_{0}}=v_{0}\mathbf{e}_{\theta}=r\omega_{0}\left(  r\right)
\mathbf{e}_{\theta})$. The support of density $\rho_{0}$ is denoted by
$\Omega$, which is an axi-symmetric bounded domain. Let $R_{0}$ be support
radius in $r$, that is, the maximum of $r$ such that $(r,z)\in\Omega$. We
choose the coordinate system such that $(R_{0},0)\in\partial\Omega$. We make
the following assumptions: \newline i) $\omega_{0}\in C^{1}[0,R_{0}]$
satisfies the Rayleigh stability condition (i.e. $\Upsilon(r)>0$ for
$r\in\lbrack0,R_{0}]$);\newline ii) $\partial\Omega$ is $C^{2}$ near
$(R_{0},0)\ $and has positive curvature (equivalently $\Omega$ is locally
convex) at $(R_{0},0)$;\newline iii) $\rho_{0}$ satisfies
(\ref{rhonearboundary}). \newline The following lemma will be used later.

\begin{lemma}
\label{lemma-rho-int}Under Assumptions ii) and iii) above, for $\varepsilon>0$
small enough we have
\[
\int_{-\infty}^{+\infty}\rho_{0}^{\lambda}(r,z)dz\thickapprox(R_{0}%
-r)^{\frac{\lambda}{\mathbb{\gamma}_{0}-1}+\frac{1}{2}},
\]
for any $\lambda>0$ and $r\in\left(  R_{0}-\varepsilon,R_{0}\right)  $.
\end{lemma}

\begin{proof}
By (\ref{rhonearboundary}),
\[
\int_{-\infty}^{+\infty}\rho_{0}^{\lambda}(r,z)dz\thickapprox\int_{\left(
r,z\right)  \in\Omega}\text{dist}((r,z),\partial\Omega)^{\frac{\lambda
}{\mathbb{\gamma}_{0}-1}}dz.
\]
First, we consider the case when $\Omega$ is the ball $\left\{  r^{2}%
+z^{2}<R_{0}^{2}\right\}  $. Then for $r$ close to $R_{0}$
\begin{align}
\int_{\left(  r,z\right)  \in\Omega}\text{dist}((r,z),\partial\Omega
)^{\frac{\lambda}{\mathbb{\gamma}_{0}-1}}dz &  =2\int_{0}^{\sqrt{R_{0}%
^{2}-r^{2}}}\left(  R_{0}-\sqrt{r^{2}+z^{2}}\right)  ^{\frac{\lambda
}{\mathbb{\gamma}_{0}-1}}dz\label{estimate-sphere}\\
&  \thickapprox\int_{0}^{\sqrt{R_{0}^{2}-r^{2}}}\left(  R_{0}^{2}-r^{2}%
-z^{2}\right)  ^{\frac{\lambda}{\mathbb{\gamma}_{0}-1}}dz\nonumber\\
&  =\left(  R_{0}^{2}-r^{2}\right)  ^{\frac{\lambda}{\mathbb{\gamma}_{0}%
-1}+\frac{1}{2}}\int_{0}^{1}\left(  1-u^{2}\right)  ^{\frac{\lambda
}{\mathbb{\gamma}_{0}-1}}du\nonumber\\
&  \thickapprox(R_{0}-r)^{\frac{\lambda}{\mathbb{\gamma}_{0}-1}+\frac{1}{2}%
}.\nonumber
\end{align}
For general $\Omega$, let $\frac{1}{r_{0}}>0$ be the curvature of
$\partial\Omega$ at $\left(  R_{0},0\right)  $ and
\[
\Gamma=\left\{  \left(  r,z\right)  \ |\ \left(  r-R_{0}+r_{0}\right)
^{2}+z^{2}=r_{0}^{2}\right\}  ,
\]
be the osculating circle at $\left(  R_{0},0\right)  $. Then near $\left(
R_{0},0\right)  $, $\partial\Omega\ $is approximated by $\Gamma$ to the 2nd
order. For any $r\in\left(  R_{0}-\varepsilon,R_{0}\right)  $, let $\left(
r,-z_{1}\left(  r\right)  \right)  ,\ \left(  r,z_{2}\left(  r\right)
\right)  \ $be the intersection of $\partial\Omega$ with the vertical line
$r^{\prime}=r$, where $z_{1}\left(  r\right)  ,z_{2}\left(  r\right)  >0$.
Then for $\varepsilon$ small enough, we have
\[
z_{1}\left(  r\right)  ,z_{2}\left(  r\right)  =\sqrt{r_{0}^{2}-\left(
r-R_{0}+r_{0}\right)  ^{2}}+o\left(  \sqrt{r_{0}^{2}-\left(  r-R_{0}%
+r_{0}\right)  ^{2}}\right)  .
\]
And for $\left(  r,z\right)  \in\Omega$ with $r\in\left(  R_{0}-\varepsilon
,R_{0}\right)  $,%
\begin{align*}
\text{dist}((r,z),\partial\Omega) &  =\text{dist}((r,z),\Gamma)+o\left(
\text{dist}((r,z),\Gamma)\right)  \\
&  =\left(  r_{0}-\sqrt{\left(  r-R_{0}+r_{0}\right)  ^{2}+z^{2}}\right)
+o\left(  \left(  r_{0}-\sqrt{\left(  r-R_{0}+r_{0}\right)  ^{2}+z^{2}%
}\right)  \right)  .
\end{align*}
Then similar to (\ref{estimate-sphere}), we have
\[
\int_{-\infty}^{+\infty}\rho_{0}^{\lambda}(r,z)dz\thickapprox\left(  r_{0}%
^{2}-\left(  r-R_{0}+r_{0}\right)  ^{2}\right)  ^{\frac{\lambda}%
{\mathbb{\gamma}_{0}-1}+\frac{1}{2}}\thickapprox(R_{0}-r)^{\frac{\lambda
}{\mathbb{\gamma}_{0}-1}+\frac{1}{2}}.
\]

\end{proof}

Let $X_{1}:=L_{\Phi^{\prime\prime}(\rho_{0})}^{2}$, $X_{2}=L_{\rho_{0}}^{2}$,
$X=X_{1}\times X_{2}$, $Y=\left(  L_{\rho_{0}}^{2}\right)  ^{2}$ and
$\mathbf{X}:=X\times Y$. The linearized Euler-Poisson system for axi-symmetric
perturbations around the rotating star solution $(\rho_{0}\left(  r,z\right)
,\omega_{0}\left(  r\right)  r\mathbf{e}_{\theta})$ is
\begin{equation}%
\begin{cases}
\partial_{t}v_{r} & =2\omega_{0}\left(  r\right)  v_{\theta}-\partial_{r}%
(\Phi^{\prime\prime}(\rho_{0})\rho+V(\rho)),\\
\partial_{t}v_{z} & =-\partial_{z}(\Phi^{\prime\prime}(\rho_{0})\rho
+V(\rho)),\\
\partial_{t}v_{\theta} & =-\frac{1}{r}\partial_{r}(\omega_{0}r^{2})v_{r},\\
\partial_{t}\rho & =-\nabla\cdot(\rho_{0}v)=-\nabla\cdot(\rho_{0}%
(v_{r},0,v_{z})),
\end{cases}
\label{linearized-EP}%
\end{equation}
with $\Delta V=4\pi\rho$. Here, $(\rho,\vec{v}=(v_{r},v_{\theta},v_{z}%
))\in\mathbf{X}$ are perturbations of density and velocity.

Define the operators
\[
L:=\Phi^{\prime\prime}(\rho_{0})-4\pi(-\Delta)^{-1}: X_{1}\rightarrow
(X_{1})^{\ast}, \quad A=\rho_{0}:Y\rightarrow Y^{\ast},
\]

\[
A_{1}:=\frac{4\omega_{0}^{2}r^{3}\rho_{0}}{\partial_{r}(\omega_{0}^{2}r^{4}%
)}=\frac{4\omega_{0}^{2}\rho_{0}}{\Upsilon(r)}:X_{2} \rightarrow(X_{2})^{\ast
},
\]
and
\begin{equation}
B=%
\begin{pmatrix}
B_{1}\\
B_{2}%
\end{pmatrix}
:D(B)\subset Y^{\ast}\rightarrow X,\quad B^{\prime}=\left(  B_{1}^{\prime
},B_{2}^{\prime}\right)  :X^{\ast}\supset D(B^{\prime})\rightarrow Y,
\label{defn-b}%
\end{equation}
where
\begin{equation}
B_{1}%
\begin{pmatrix}
v_{r}\\
v_{z}%
\end{pmatrix}
=-\text{$\nabla\cdot$}(v_{r},0,v_{z}),\quad B_{1}^{\prime}\rho=%
\begin{pmatrix}
\partial_{r}\rho\\
\partial_{z}\rho
\end{pmatrix}
, \label{defn-B1}%
\end{equation}
and
\begin{equation}
B_{2}%
\begin{pmatrix}
v_{r}\\
v_{z}%
\end{pmatrix}
=-\frac{\partial_{r}(\omega_{0}r^{2})}{r\rho_{0}}v_{r},\quad(B_{2})^{\prime
}v_{\theta}=%
\begin{pmatrix}
-\frac{\partial_{r}(\omega_{0}r^{2})}{r\rho_{0}}v_{\theta}\\
0
\end{pmatrix}
. \label{defn-B2}%
\end{equation}
Then the linearized Euler-Poisson system (\ref{linearized-EP}) can be written
in a separable Hamiltonian form
\begin{equation}
\frac{d}{dt}%
\begin{pmatrix}
u_{1}\\
u_{2}%
\end{pmatrix}
=\mathbf{J}\mathbf{L}%
\begin{pmatrix}
u_{1}\\
u_{2}%
\end{pmatrix}
, \label{Lep}%
\end{equation}
where $u_{1}=(\rho,v_{\theta})$ and $u_{2}=(v_{r},v_{z})$. The operators
\[
\mathbf{J}:=%
\begin{pmatrix}
0, & B\\
-B^{\prime}, & 0
\end{pmatrix}
:\mathbf{X}^{\ast}\rightarrow\mathbf{X},\quad\mathbf{L}:=%
\begin{pmatrix}
\mathbb{L}, & 0\\
0, & A
\end{pmatrix}
:\mathbf{X}\rightarrow\mathbf{X}^{\ast},
\]
are off-diagonal anti-self-dual and diagonal self-dual respectively, where
\[
\mathbb{L}=%
\begin{pmatrix}
L, & 0\\
0, & A_{1}%
\end{pmatrix}
:X\rightarrow X^{\ast}.
\]

First, we check that $\left(  \mathbb{L},A,B\right)  $ in (\ref{Lep}) satisfy
the assumptions (\textbf{G1})-(\textbf{G3}) for the abstract theory in Section
\ref{section-abstract}. The assumptions (\textbf{G1}) and (\textbf{G2}) can be
shown by the same arguments in the proof of Lemma 3.5 in \cite{LZ2019} and
that $B_{2}$ is bounded. The Rayleigh stability condition $\Upsilon
(r)>0\ $implies that the operator $A_{1}$ is bounded, positive and self-dual.
By the same proof of Lemma 3.6 in \cite{LZ2019}, we have the following lemma.

\begin{lemma}
\label{lemma-decom-L}There exists a direct sum decomposition $L_{\Phi
^{\prime\prime}(\rho_{0})}^{2}=X_{-}\oplus\ker L\oplus X_{+}$ and $\delta
_{0}>0\ $such that:

i) $\dim\left(  X_{-}\right)  ,\dim\ker L<\infty;\ $

ii) $L|_{X_{-}}<0,\ L|_{X_{+}}\geq\delta_{0}$ and $X_{-}\perp X_{+}$ in the
inner product of $L_{\Phi^{\prime\prime}(\rho_{0})}^{2}$.
\end{lemma}

The assumption (\textbf{G3}) readily follows from above lemma. Therefore, we
can apply Theorem \ref{T:abstract} to the linearized Euler-Poisson system
(\ref{Lep}). This proves the conclusions in Theorem \ref{Th: rayleigh stable}
except for the formula $n^{-}\left(  \mathbb{L}|_{\overline{R\left(  B\right)
}}\right)  =n^{-}\left(  \mathcal{K}|_{R\left(  B_{1}\right)  }\right)  $,
which will be shown later. Here, $\overline{R\left(  B\right)  }$ is the
closure of $R(B)$ in $X$, and the operators $B,B_{1}$ are defined in
(\ref{defn-b})-(\ref{defn-B2}).

\begin{remark}
In some literature \cite{LS2004,LTSJ2008,LS2009,LTSJ2011}, the Rayleigh
stability condition is $\Upsilon(r)\geq0$ for all $r\in\lbrack0,R_{0}]$. Here,
we used the stability condition $\Upsilon(r)>0$ for all $r\in\lbrack0,R_{0}]$
as in the astrophysical literature such as \cite{BB1974,TJL2000}. If
$\Upsilon(r)\geq0$ for all $r\in\lbrack0,R_{0}]$ and $\Upsilon(r)=0$ only at
some isolated points, let $\Lambda\left(  r,z\right)  =\frac{4\omega_{0}%
^{2}\rho_{0}}{\Upsilon(r)}$ and the operator $A_{1}:L_{\Lambda}^{2}%
\rightarrow(L_{\Lambda}^{2})^{\ast}$\ is bounded and positive. The linearized
Euler-Poisson system can still be studied in the framework of separable
Hamiltonian systems and similar results as in Theorem
\ref{Th: rayleigh stable} can be obtained.
\end{remark}

\subsection{Dynamically accessible perturbations}

By Theorem \ref{Th: rayleigh stable}, the solutions of (\ref{Lep}) are
spectrally stable (i.e. nonexistence of exponentially growing solution) if and
only if $\mathbb{L}|_{\overline{R(B)}}\geq0$. More precisely, we have

\begin{corollary}
\label{snprpp} Assume $\omega_{0}\in C^{1}[0,R_{0}]$, (\ref{rhonearboundary}),
and $\inf_{r\in\lbrack0,R_{0}]}\Upsilon(r)>0$. The rotating star solution
$\left(  \rho_{0}\left(  r,z\right)  ,\vec{v_{0}}=r\omega_{0}\left(  r\right)
\mathbf{e}_{\theta}\right)  \ $of Euler-Poisson system is spectrally stable if
and only if
\begin{equation}
\langle L\delta\rho,\delta\rho\rangle+\langle A_{1}\delta v_{\theta},\delta
v_{\theta}\rangle\geq0\text{ for all }\left(  \delta\rho,\delta v_{\theta
}\right)  \in\overline{R(B)}.\label{stability-criterion}%
\end{equation}

\end{corollary}

In this section, we discuss the physical meaning of above stability criterion
by using the variational structure of the rotating stars.

For any solution $(\rho,v)$ of the axi-symmetric Euler-Poisson system
(\ref{EP}), define the angular momentum $j=v_{\theta}r\ $and the generalized
total angular momentum
\begin{equation}
A_{g}(\rho,v_{\theta})=\int_{\mathbb{R}^{3}}\rho g(v_{\theta}%
r)dx,\ \label{defn-general-momentum}%
\end{equation}
for any function $g\in C^{1}\left(  \mathbb{R}\right)  $.$\ $

\begin{lemma}
For any $g\in C^{1}(\mathbb{R})$, the functional $A_{g}(\rho,v_{\theta})$ is
conserved for the Euler-Poisson system (\ref{EP}).
\end{lemma}

\begin{proof}
First, we note that the angular momentum $j$ is an invariant of the particle
trajectory under the axi-symmetric force field $-\nabla V-\nabla\Phi^{\prime
}(\rho)$. Let $\varphi\left(  x,t\right)  $ be the flow map of the velocity
field $v$ with initial position $x$, and $J\left(  x,t\right)  $ be the
Jacobian of $\varphi$. Then $\rho\left(  \varphi\left(  x,t\right)  ,t\right)
J\left(  x,t\right)  =\rho\left(  x,0\right)  $ and
\begin{align*}
A_{g}(\rho,v_{\theta})\left(  0\right)   &  =\int_{\mathbb{R}^{3}}\rho\left(
x,0\right)  g(j\left(  x\right)  )dx\\
&  =\int_{\mathbb{R}^{3}}\rho\left(  \varphi\left(  x,t\right)  ,t\right)
J\left(  x,t\right)  \ g(j\left(  \varphi\left(  x,t\right)  \right)  )dx\\
&  =\int_{\mathbb{R}^{3}}\rho\left(  y,t\right)  g(j\left(  y\right)
)dx=A_{g}(\rho,v_{\theta})\left(  t\right)  .
\end{align*}

\end{proof}

The steady state $(\rho_{0},\omega_{0}r\mathbf{e}_{\theta})$ has the following
variational structure. By the steady state equation (\ref{eqn-steady-state}),
we have%
\begin{equation}
\frac{1}{2}\omega_{0}^{2}r^{2}+\Phi^{\prime}(\rho_{0})-|x|^{-1}\ast\rho
_{0}+g_{0}\left(  \omega_{0}r^{2}\right)  +c_{0}=0\ \text{in }\Omega,
\label{eqn-steady-variational}%
\end{equation}
where $c_{0}>0$ is the constant in (\ref{eqn-steady-state}) and $g_{0}\in
C^{1}\left(  \mathbb{R}\right)  $ satisfies the equation
\begin{equation}
g_{0}^{\prime}\left(  \omega_{0}\left(  r\right)  r^{2}\right)  =-\omega
_{0}\left(  r\right)  ,\ \ \ \forall\ r\in\left[  0,R_{0}\right]  .
\label{eqn-g-0}%
\end{equation}
The existence of $g_{0}$ satisfying (\ref{eqn-g-0}) is ensured by the Rayleigh
stable condition $\Upsilon(r)>0$ which implies that $\omega_{0}\left(
r\right)  r^{2}$ is monotone to $r$. The equations (\ref{eqn-steady-state})
and (\ref{eqn-steady-variational}) are equivalent since
\[
g_{0}\left(  \omega_{0}\left(  r\right)  r^{2}\right)  =-\frac{1}{2}\omega
_{0}^{2}r^{2}-\int_{0}^{r}\omega_{0}^{2}(s)s\ ds
\]
due to (\ref{eqn-g-0}) and integration by parts. Denote the the total energy
by
\[
H(\rho,v)=\int_{\mathbb{R}^{3}}\frac{1}{2}\rho v^{2}+\Phi(\rho)-\frac{1}{8\pi
}|\nabla V|^{2}dx,\ \ \Delta V=4\pi\rho,\
\]
which is conserved for the Euler-Poisson system (\ref{EP}). Define the
energy-Casimir functional
\[
H_{c}(\rho,v)=H(\rho,v)+c_{0}\int_{\mathbb{R}^{3}}\rho\ dx+\int_{\mathbb{R}%
^{3}}\rho g_{0}(v_{\theta}r)\ dx,
\]
where $c_{0}\ $and $g_{0}$ are as in (\ref{eqn-steady-variational}). Then
$(\rho_{0},\omega_{0}r\mathbf{e}_{\theta})$ is a critical point of $H_{c}%
(\rho,v)$, since
\begin{align*}
\langle DH_{c}(\rho_{0},\omega_{0}r\mathbf{e}_{\theta}),(\delta\rho,\delta
v)\rangle=  &  \int_{\mathbb{R}^{3}}\left[  \frac{1}{2}\omega_{0}^{2}%
r^{2}+\Phi^{\prime}(\rho_{0})+V(\rho_{0})+c_{0}+g_{0}(\omega_{0}r^{2})\right]
\delta\rho\ dx\\
&  +\int_{\mathbb{R}^{3}}[\rho_{0}\omega_{0}r+\rho_{0}g_{0}^{\prime}%
(\omega_{0}r^{2})r]\delta v_{\theta}\ dx=0
\end{align*}
by equations (\ref{eqn-steady-variational}) and (\ref{eqn-g-0}). By direct
computations,
\begin{align}
&  \langle D^{2}H_{c}(\rho,v)[\rho_{0},\omega_{0}r\mathbf{e}_{\theta}%
](\delta\rho,\delta v),(\delta\rho,\delta v)\rangle\label{2nd variation Hc}\\
&  =\int_{\mathbb{R}^{3}}(\Phi^{\prime\prime}(\rho_{0})\left(  \delta
\rho\right)  ^{2}-4\pi(-\Delta^{-1}\delta\rho)\delta\rho+\rho_{0}\left(
\delta v_{r}\right)  ^{2}+\rho_{0}\left(  \delta v_{z}\right)  ^{2}%
dx\nonumber\\
&  \quad+\int_{\mathbb{R}^{3}}\rho_{0}(1+g_{0}^{\prime\prime}(\omega_{0}%
r^{2})r^{2})\left(  \delta v_{\theta}\right)  ^{2}dx\nonumber\\
&  =\langle L\delta\rho,\delta\rho\rangle+\langle A_{1}\delta v_{\theta
},\delta v_{\theta}\rangle+\left\langle A\left(  \delta v_{r},\delta
v_{z}\right)  ,\left(  \delta v_{r},\delta v_{z}\right)  \right\rangle
,\nonumber
\end{align}
where we used the identity%
\[
1+g_{0}^{\prime\prime}(\omega_{0}r^{2})r^{2}=1-\frac{\omega_{0}^{\prime}r^{2}%
}{\frac{d}{dr}\left(  \omega_{0}r^{2}\right)  }=\frac{4\omega_{0}^{2}r^{3}%
}{\frac{d}{dr}\left(  \omega_{0}^{2}r^{4}\right)  }=\frac{4\omega_{0}^{2}%
}{\Upsilon(r)}.
\]

The functional (\ref{2nd variation Hc}) is a conserved quantity of the
linearized Euler-Poisson system (\ref{Lep}) due to the Hamiltonian structure.
We note that the number of negative directions of (\ref{2nd variation Hc}) is
given by $n^{-}\left(  \mathbb{L}\right)  $.

We now turn to the spaces of $\delta\rho$ and $\left(  \delta\rho,\delta
v_{\theta}\right)  $.

\begin{lemma}
\label{lemma-R-B1}It holds that
\[
R\left(  B_{1}\right)  =\overline{R\left(  B_{1}\right)  }=\left\{  \delta
\rho\in L_{\Phi^{\prime\prime}(\rho_{0})}^{2}\ \bigg|\ \int_{\mathbb{R}^{3}%
}\delta\rho dx=0\right\}  .
\]

\end{lemma}

\begin{proof}
Since $\ker B_{1}^{\prime}=\ker\nabla$ is spanned by constant functions, we
have
\[
\overline{R\left(  B_{1}\right)  }=\left(  \ker B_{1}^{\prime}\right)
^{\perp}=\left\{  \delta\rho\in L_{\Phi^{\prime\prime}(\rho_{0})}%
^{2}\ \bigg|\ \int_{\mathbb{R}^{3}}\delta\rho dx=0\right\}  .
\]
It remains to show $R\left(  B_{1}\right)  =\overline{R\left(  B_{1}\right)
}$ which is equivalent to $R\left(  B_{1}A\right)  =\overline{R\left(
B_{1}A\right)  }$. By Lemma 3.15 in \cite{LZ2019}, we have the orthogonal
decomposition
\[
L_{\rho_{0}}^{2}=\ker\left(  B_{1}A\right)  \oplus W,
\]
where $W=\left\{  w=\nabla p\in L_{\rho_{0}}^{2}\right\}  $. For any
$\delta\rho\in R\left(  B_{1}A\right)  $, by the proof of Lemma 3.15 in
\cite{LZ2019}, there exists a unique gradient field $\nabla p\in L_{\rho_{0}%
}^{2}$ such that
\[
B_{1}A\nabla p=\nabla\cdot\left(  \rho_{0}\nabla p\right)  =\delta\rho.
\]
By Proposition 12 in \cite{JM2020}, we have
\begin{equation}
\left\Vert \nabla p\right\Vert _{L_{\rho_{0}}^{2}}\lesssim\left\Vert
\nabla\cdot\left(  \rho_{0}\nabla p\right)  \right\Vert _{L_{\Phi
^{\prime\prime}(\rho_{0})}^{2}}=\left\Vert \delta\rho\right\Vert
_{L_{\Phi^{\prime\prime}(\rho_{0})}^{2}}. \label{estimate-gradient}%
\end{equation}
For any $u\in D\left(  B_{1}A\right)  $, let $v\in W$ be the projection of $u$
to $W$. Then above estimate (\ref{estimate-gradient}) implies that
\[
dist\left(  u,\ker\left(  B_{1}A\right)  \right)  =\inf_{z\in\ker\left(
B_{1}A\right)  }\left\Vert u-z\right\Vert _{L_{\rho_{0}}^{2}}=\left\Vert
v\right\Vert _{L_{\rho_{0}}^{2}}\lesssim\left\Vert B_{1}Au\right\Vert
_{L_{\Phi^{\prime\prime}(\rho_{0})}^{2}}.
\]
By Theorem 5.2 in \cite[P. 231]{Katobook1995}, this implies that $R\left(
B_{1}\right)  =\overline{R\left(  B_{1}\right)  }.$
\end{proof}

\begin{definition}
\label{defper} The perturbation $\left(  \delta\rho,\delta v_{\theta}\right)
\in X$ is called dynamically accessible if $\left(  \delta\rho,\delta
v_{\theta}\right)  \in\overline{R(B)}$.
\end{definition}

In the next lemma, we give two equivalent characterizations of the dynamically
accessible perturbations.

\begin{lemma}
\label{lemma-dyna-acce} For $\left(  \delta\rho,\delta v_{\theta}\right)  \in
X$, the following statements are equivalent.

(i) $\left(  \delta\rho,\delta v_{\theta}\right)  \in\overline{R(B)};$

(ii)
\begin{equation}
\int_{\mathbb{R}^{3}} g(\omega_{0}r^{2})\delta\rho\ dx+\!\int_{\mathbb{R}^{3}%
}\rho_{0}rg^{\prime}(\omega_{0}r^{2})\delta v_{\theta}\ dx=0,\ \forall g\in
C^{1}\left(  \mathbb{R}\right)  ; \label{dyna-accessible-inte}%
\end{equation}

(iii) $\int_{\mathbb{R}^{3}}\delta\rho\ dx=0$ and
\begin{equation}
\int_{-\infty}^{+\infty}\delta v_{\theta}\rho_{0}\left(  r,z\right)
dz=\frac{\partial_{r}\left(  \omega_{0}r^{2}\right)  }{r^{2}}\int_{0}^{r}%
s\int_{-\infty}^{+\infty}\delta\rho(s,z)dzds.
\label{dyna-accessible-projection}%
\end{equation}

\end{lemma}

\begin{proof}
First, we show (i) and (ii) are equivalent. We have $\overline{R(B)}=\left(
\ker B^{\prime}\right)  ^{\perp}$, where the dual operator $B^{\prime}:$
$X^{\ast}\rightarrow Y\ $is defined in (\ref{defn-B1})-(\ref{defn-B2}). Let
$\left(  \rho,v_{\theta}\right)  $ be a $C^{1}$ function in $\ker B^{\prime}$,
then
\[
B^{\prime}%
\begin{pmatrix}
\rho\\
v_{\theta}%
\end{pmatrix}
=%
\begin{pmatrix}
\partial_{r}\rho-\frac{\partial_{r}(\omega_{0}r^{2})}{r\rho_{0}}v_{\theta}\\
\partial_{z}\rho
\end{pmatrix}
=%
\begin{pmatrix}
0\\
0
\end{pmatrix}
.
\]
Since $\partial_{z}\rho=0$ and $\omega_{0}r^{2}$ is monotone to $r$ by the
Rayleigh stability condition, we can write $\rho=g\left(  \omega_{0}%
r^{2}\right)  $ for some function $g\in C^{1}$. Then $\partial_{r}\rho
-\frac{\partial_{r}(\omega_{0}r^{2})}{r\rho_{0}}v_{\theta}=0$ implies that
$v_{\theta}=\rho_{0}rg^{\prime}(\omega_{0}r^{2})$. Thus $\ker B^{\prime}$ is
the closure of the set
\[
\left\{  \left(  g\left(  \omega_{0}r^{2}\right)  ,\rho_{0}rg^{\prime}%
(\omega_{0}r^{2})\right)  ,\ g\in C^{1}\left(  \mathbb{R}\right)  \right\}  ,
\]
in $X^{\ast}$. Therefore, $\left(  \delta\rho,\delta v_{\theta}\right)
\in\overline{R(B)}=\left(  \ker B^{\prime}\right)  ^{\perp}$ if and only if
(\ref{dyna-accessible-inte}) is satisfied.

Next, we show (ii) and (iii) are equivalent. If (ii) is satisfied, by choosing
$g=1$ we get $\int\delta\rho\ dx=0$. Then by (\ref{dyna-accessible-inte}) and
integration by parts, we have
\[
\int_{0}^{R_{0}}\left[  r^{2}\int_{-\infty}^{+\infty}\delta v_{\theta}\rho
_{0}(r,z)dz-\partial_{r}(\omega_{0}r^{2})\left(  \int_{0}^{r}s\int_{-\infty
}^{+\infty}\delta\rho(s,z)dzds\right)  \right]  g^{\prime}(\omega_{0}%
r^{2})dr=0.
\]
which implies (\ref{dyna-accessible-projection}) since $g\in C^{1}\left(
\mathbb{R}\right)  $ is arbitrary. On the other hand, by reversing the above
computation, (ii) follows from (iii).
\end{proof}

The statement (ii) above implies that for any $\left(  \delta\rho,\delta
v_{\theta}\right)  \in\overline{R(B)}$, we have
\[
\langle DA_{g}(\rho_{0},\omega_{0}r),(\delta\rho,\delta v_{\theta})\rangle=0,
\]
where the generalized angular momentum $A_{g}$ is defined in
(\ref{defn-general-momentum}). That is, a dynamically accessible perturbation
$(\delta\rho,\delta v_{\theta})$ must lie on the tangent space of the
functional $A_{g}$ at the equilibrium $(\rho_{0},\omega_{0}r\ \mathbf{e}%
_{\theta})$. Since $g$ is arbitrary, this implies infinite many constraints
for dynamically accessible perturbations. The stability criterion
(\ref{stability-criterion}) implies that that rotating stars are stable if and
only if they are local minimizers of energy-Casimir functional $H(\rho,v)$
under the constraints of fixed generalized angular momentum $A_{g}$ for all
$g$. This contrasts significantly with the case of non-rotating stars. It was
shown in (\cite{LZ2019}) that non-rotating stars are stable if and only if
they are local minimizers of the energy-Casimir functional under the only
constraint of fixed total mass. The stability criterion
(\ref{stability-criterion}) for rotating stars involves infinitely many
constraints and is much more difficult to check. In the next section, we give
an equivalent stability criterion in terms of a reduced functional
(\ref{defn-cal-k}) under only the mass constraint.

\begin{remark}
For non-rotating stars, the dynamically accessible perturbations are given by
$R\left(  B_{1}\right)  =\overline{R\left(  B_{1}\right)  }$ which is the
perturbations preserving the mass (see Lemma \ref{lemma-R-B1}). For rotating
stars, the dynamically accessible space $\overline{R(B)}$ is different from
$R\left(  B\right)  $.
\end{remark}

\subsection{Reduced functional and the equivalent Stability Criterion}

In this section, we prove the formula $n^{-}\left(  \mathbb{L}|_{\overline
{R\left(  B\right)  }}\right)  =n^{-}\left(  \mathcal{K}|_{R\left(
B_{1}\right)  }\right)  $ and complete the proof of Theorem
\ref{Th: rayleigh stable}.

\begin{lemma}
\label{udeltarhoinL2} For any $\delta\rho\in R\left(  B_{1}\right)  $, define
\begin{equation}
u_{\theta}^{\delta\rho}=\frac{\partial_{r}(\omega_{0}r^{2})}{r^{2}}\frac
{\int_{0}^{r}s\int_{-\infty}^{+\infty}\delta\rho(s,z)dzds}{\int_{-\infty
}^{+\infty}{\rho_{0}}(r,z)dz}\text{. } \label{defn-u-theta-rho}%
\end{equation}
Then $\left(  \delta\rho,u_{\theta}^{\delta\rho}\right)  \in\overline{R\left(
B\right)  }$ and $\left\Vert u_{\theta}^{\delta\rho}\right\Vert _{L_{\rho_{0}%
}^{2}}\ \lesssim\left\Vert \delta\rho\right\Vert _{L_{\Phi^{\prime\prime}%
(\rho_{0})}^{2}}.$
\end{lemma}

\begin{proof}
We have
\begin{align*}
&  \left\Vert u_{\theta}^{\delta\rho}\right\Vert _{L_{\rho_{0}}^{2}}%
^{2}\ \lesssim\int_{\mathbb{R}^{3}}{\rho_{0}}\left(  \frac{\int_{0}^{r}%
s\int_{-\infty}^{+\infty}\delta\rho(s,z)dzds}{r\int_{-\infty}^{+\infty}%
{\rho_{0}}(r,z)dz}\right)  ^{2}dx=2\pi\int_{0}^{R_{0}}\frac{\left(  \int
_{0}^{r}s\int_{-\infty}^{+\infty}\delta\rho(s,z)dzds\right)  ^{2}}%
{r\int_{-\infty}^{+\infty}{\rho_{0}}(r,z)dz}dr\\
&  =2\pi\int_{0}^{R_{0}-\varepsilon}\frac{\left(  \int_{0}^{r}s\int_{-\infty
}^{+\infty}\delta\rho(s,z)dzds\right)  ^{2}}{r\int_{-\infty}^{+\infty}%
{\rho_{0}}(r,z)dz}dr+2\pi\int_{R_{0}-\varepsilon}^{R_{0}}\frac{\left(
\int_{0}^{r}s\int_{-\infty}^{+\infty}\delta\rho(s,z)dzds\right)  ^{2}}%
{r\int_{-\infty}^{+\infty}{\rho_{0}}(r,z)dz}dr\\
&  =I+II,
\end{align*}
where $\varepsilon>0$ is chosen such that Lemma \ref{lemma-rho-int} holds.
Since the function $h_{1}\left(  r\right)  =\int_{-\infty}^{+\infty}{\rho_{0}%
}(r,z)dz$ has a positive lower bound in $\left[  0,R_{0}-\varepsilon\right]  $
and $h_{2}\left(  r\right)  =\int_{-\infty}^{+\infty}\frac{1}{\Phi
^{\prime\prime}(\rho_{0}(r,z))}dz$ is bounded, by Hardy's inequality (see
Lemma 3.21 in \cite{LZ2019}) we have
\begin{align*}
I &  \lesssim\int_{0}^{R_{0}-\varepsilon}r^{-2}\left(  \int_{0}^{r}%
s\int_{-\infty}^{+\infty}\delta\rho(s,z)dzds\right)  ^{2}dr\\
&  \lesssim\int_{0}^{R_{0}-\varepsilon}r^{2}\left(  \int_{-\infty}^{+\infty
}\delta\rho(r,z)dz\right)  ^{2}dr\\
&  \lesssim\int_{0}^{R_{0}-\varepsilon}r^{2}\left(  \int_{-\infty}^{+\infty
}\Phi^{\prime\prime}(\rho_{0})\left(  \delta\rho\right)  ^{2}(r,z)dz\right)
\left(  \int_{-\infty}^{+\infty}\frac{1}{\Phi^{\prime\prime}(\rho_{0}%
(r,z))}dz\right)  dr\\
&  \lesssim\int_{0}^{R_{0}-\varepsilon}r\int_{-\infty}^{+\infty}\Phi
^{\prime\prime}(\rho_{0})\left(  \delta\rho\right)  ^{2}(r,z)dzdr\lesssim
\Vert\delta\rho\Vert_{L_{\Phi^{\prime\prime}(\rho_{0})}^{2}}^{2}.
\end{align*}
By Hardy's inequality and Lemma \ref{lemma-rho-int}, we have
\begin{align*}
II &  =2\pi\int_{R_{0}-\varepsilon}^{R_{0}}\frac{\left(  \int_{0}^{r}%
s\int_{-\infty}^{+\infty}\delta\rho(s,z)dzds\right)  ^{2}}{r\int_{-\infty
}^{+\infty}{\rho_{0}}(r,z)dz}dr\\
&  \lesssim\int_{R_{0}-\varepsilon}^{R_{0}}\frac{\left(  \int_{0}^{r}%
s\int_{-\infty}^{+\infty}\delta\rho(s,z)dzds\right)  ^{2}}{(R_{0}-r)^{\frac
{1}{\mathbb{\gamma}_{0}-1}+\frac{1}{2}}}dr\\
&  \lesssim\int_{R_{0}-\varepsilon}^{R_{0}}\left(  \int_{-\infty}^{+\infty
}\delta\rho(r,z)dz\right)  ^{2}(R_{0}-r)^{-\frac{1}{\gamma_{0}-1}+\frac{3}{2}%
}dr\\
&  \lesssim\int_{R_{0}-\varepsilon}^{R_{0}}\left(  \int_{-\infty}^{+\infty
}\Phi^{\prime\prime}(\rho_{0})(\delta\rho)^{2}dz\right)  \left(  \int
_{-\infty}^{+\infty}\frac{1}{\Phi^{\prime\prime}(\rho_{0})}dz\right)
(R_{0}-r)^{-\frac{1}{\mathbb{\gamma}_{0}-1}+\frac{3}{2}}dr\\
&  \lesssim\int_{R_{0}-\varepsilon}^{R_{0}}\left(  \int_{-\infty}^{+\infty
}\Phi^{\prime\prime}(\rho_{0})(\delta\rho)^{2}dz\right)  (R_{0}-r)dr\\
&  \lesssim\Vert\delta\rho\Vert_{L_{\Phi^{\prime\prime}(\rho_{0})}^{2}}^{2},
\end{align*}
where we used the estimate
\[
\int_{-\infty}^{+\infty}\frac{1}{\Phi^{\prime\prime}(\rho_{0})}dz\thickapprox
\int_{-\infty}^{+\infty}\rho_{0}^{2-\gamma_{0}}dz\thickapprox(R_{0}%
-r)^{\frac{2-\mathbb{\gamma}_{0}}{\mathbb{\gamma}_{0}-1}+\frac{1}{2}},
\]
since $\Phi^{\prime\prime}\left(  s\right)  \approx s^{\mathbb{\gamma}_{0}-2}$
for $s$ small. This proves $\left\Vert u_{\theta}^{\delta\rho}\right\Vert
_{L_{\rho_{0}}^{2}}\ \lesssim\left\Vert \delta\rho\right\Vert _{L_{\Phi
^{\prime\prime}(\rho_{0})}^{2}}$.

The statement $\left(  \delta\rho,u_{\theta}^{\delta\rho}\right)  \in
\overline{R\left(  B\right)  }$ follows from Lemma \ref{lemma-dyna-acce} since
$\int_{\mathbb{R}^{3}}\delta\rho\ dx=0$ for $\delta\rho\in R\left(
B_{1}\right)  $ and $u_{\theta}^{\delta\rho}$ obviously satisfies
(\ref{dyna-accessible-projection}).
\end{proof}

With the help of lemma \ref{udeltarhoinL2} we can finish the proof of Theorem
\ref{Th: rayleigh stable}.

\begin{proof}
[Proof of Theorem \ref{Th: rayleigh stable}]We only need to show $n^{-}\left(
\mathbb{L}|_{\overline{R\left(  B\right)  }}\right)  =n^{-}\left(
\mathcal{K}|_{R\left(  B_{1}\right)  }\right)  $. First, we have%
\begin{equation}
\left\langle \mathbb{L}%
\begin{pmatrix}
\delta\rho\\
\delta v_{\theta}%
\end{pmatrix}
,%
\begin{pmatrix}
\delta\rho\\
\delta v_{\theta}%
\end{pmatrix}
\right\rangle \geq\langle\mathcal{K}\delta\rho,\delta\rho\rangle
,\ \ \forall\left(  \delta\rho,\delta v_{\theta}\right)  \in\overline
{R(B)},\label{inequality-L-K}%
\end{equation}
since
\begin{align*}
\langle A_{1}\delta v_{\theta},\delta v_{\theta}\rangle &  =\int
_{\mathbb{R}^{3}}\frac{4\omega_{0}^{2}}{\Upsilon(r)}\rho_{0}\left(  \delta
v_{\theta}\right)  ^{2}\ dx=2\pi\int_{0}^{R_{0}}\frac{4\omega_{0}^{2}%
r}{\Upsilon(r)}\int_{-\infty}^{+\infty}\rho_{0}\left(  \delta v_{\theta
}\right)  ^{2}dz\ dr\\
&  =2\pi\int_{0}^{R_{0}}\frac{4\omega_{0}^{2}r}{\Upsilon(r)}\int_{-\infty
}^{+\infty}\rho_{0}\left(  u_{\theta}^{\delta\rho}\right)  ^{2}dz\ dr+2\pi
\int_{0}^{R_{0}}\frac{4\omega_{0}^{2}r}{\Upsilon(r)}\int_{-\infty}^{+\infty
}\rho_{0}\left(  \delta v_{\theta}-u_{\theta}^{\delta\rho}\right)
^{2}dz\ dr\\
&  \geq2\pi\int_{0}^{R_{0}}\frac{4\omega_{0}^{2}r}{\Upsilon(r)}\int_{-\infty
}^{+\infty}\rho_{0}\left(  u_{\theta}^{\delta\rho}\right)  ^{2}dz\ dr\\
&  =2\pi\int_{0}^{R_{0}}\Upsilon(r)\frac{\left(  \int_{0}^{r}s\int_{-\infty
}^{+\infty}\delta\rho(s,z)dzds\right)  ^{2}}{r\int_{-\infty}^{+\infty}\rho
_{0}({r},z)dz}dr.
\end{align*}
In the above, we used the observation that
\[
\int_{-\infty}^{+\infty}\rho_{0}\left(  \delta v_{\theta}-u_{\theta}%
^{\delta\rho}\right)  dz\ =\int_{-\infty}^{+\infty}\rho_{0}\delta v_{\theta
}dz-\ u_{\theta}^{\delta\rho}\left(  r\right)  \int_{-\infty}^{+\infty}%
\rho_{0}dz=0,
\]
since
\[
\int_{-\infty}^{+\infty}\rho_{0}\delta v_{\theta}dz=u_{\theta}^{\delta\rho
}\left(  r\right)  \int_{-\infty}^{+\infty}\rho_{0}dz=\frac{\partial
_{r}\left(  \omega_{0}r^{2}\right)  }{r^{2}}\int_{0}^{r}s\int_{-\infty
}^{+\infty}\delta\rho(s,z)dzds
\]
due to (\ref{dyna-accessible-projection}) and (\ref{defn-u-theta-rho}). Since
$\delta\rho\in R\left(  B_{1}\right)  $, it follows from (\ref{inequality-L-K}%
) that $n^{-}\left(  \mathcal{K}|_{R\left(  B_{1}\right)  }\right)  \geq
n^{-}\left(  \mathbb{L}|_{\overline{R\left(  B\right)  }}\right)  $. On the
other hand, we also have $n^{-}\left(  \mathcal{K}|_{R\left(  B_{1}\right)
}\right)  \leq n^{-}\left(  \mathbb{L}|_{\overline{R\left(  B\right)  }%
}\right)  ,$ since
\[
\langle\mathcal{K}\delta\rho,\delta\rho\rangle=\left\langle \mathbb{L}%
\begin{pmatrix}
\delta\rho\\
u_{\theta}^{\delta\rho}%
\end{pmatrix}
,%
\begin{pmatrix}
\delta\rho\\
u_{\theta}^{\delta\rho}%
\end{pmatrix}
\right\rangle .
\]
Thus $n^{-}\left(  \mathcal{K}|_{R\left(  B_{1}\right)  }\right)
=n^{-}\left(  \mathbb{L}|_{\overline{R\left(  B\right)  }}\right)  $. This
finishes the proof of Theorem \ref{Th: rayleigh stable}.
\end{proof}

\section{\label{2Examples}TPP for slowly rotating stars}

In this section, we use the stability criterion in Theorem
\ref{Th: rayleigh stable} to study two families of slowly rotating stars
parameterized by the center density.

\subsection{The case of fixed angular velocity}

In this subsection, we consider a family of slowly rotating stars with fixed
angular velocity.

Under the assumptions (\ref{P1})-(\ref{P2}), for some $\mu_{\max}>0$, there
exists a family of nonrotating stars with radially symmetric density
$\rho_{\mu}(|x|)$ parametrized by the center density $\mu\in(0,\mu_{\max})$.
We refer to \cite{LZ2019} and references therein for such results.
Let$\ R_{\mu}$ be the support radius of $\rho_{\mu}$ and $B_{\mu}=B(0,R_{\mu
})$ be the support of $\rho_{\mu}$. The radial density $\rho_{\mu}$ satisfies
\[
\Delta(\Phi^{\prime}(\rho_{\mu}))+4\pi\rho_{\mu}=0,\ \ \text{in }B_{\mu},
\]
with $\rho_{\mu}(0)=\mu$. For the general equations of state satisfying
(\ref{P1})-(\ref{P2}) with $\gamma_{0}\geq4/3$ , it was shown in \cite{HU2003}
that $\mu_{\max}=+\infty$.

Let $\omega(r)\in C^{1,\beta}[0,\infty)$ be fixed for some $\beta\in(0,1)$. We
construct a family of rotating stars for Euler-Poisson system with the
following form
\[%
\begin{cases}
\rho_{0}=\rho_{\mu,\kappa}(r,z)=\rho_{\mu}(g_{\zeta_{\mu,\kappa}}%
^{-1}((r,z))),\\
\vec{v}_{0}=\kappa r\omega_{0}\left(  r\right)  \mathbf{e}_{\theta},
\end{cases}
\]
where the dilating function is
\[
g_{\zeta_{\mu,\kappa}}=x\left(  1+\frac{\zeta_{\mu,\kappa}}{|x|^{2}}\right)  ,
\]
and $\zeta_{\mu,\kappa}(x):B_{\mu}\rightarrow\mathbb{R}$ is axi-symmetric and
even in $z$.

The existence of rotating stars $\left(  \rho_{\mu,\kappa},\kappa r\omega
_{0}\left(  r\right)  \mathbf{e}_{\theta}\right)  \ $is reduced to the
following equations for $\rho_{\mu,\kappa}$:
\begin{equation}
-\kappa^{2}\int_{0}^{r}\omega^{2}(s)sds+\Phi^{\prime}(\rho_{\mu,\kappa
})+V_{\mu,\kappa}+c_{\mu,\kappa}=0\text{ in }\Omega_{\mu,\kappa},\label{iftw1}%
\end{equation}%
\[
V_{\mu,\kappa}=-|x|^{-1}\ast\rho_{\mu,\kappa}\text{ in }\mathbb{R}^{3},
\]
where $c_{\mu,\kappa}$ is a constant and $\Omega_{\mu,\kappa}=g_{\zeta
_{\mu,\kappa}}(B_{\mu})$ is the support of the density $\rho_{\mu,\kappa}\ $of
the rotating star solution.

By similar arguments as in \cite{H1994, SW2017, JJ2019}, we can get the
following existence theorem.

\begin{theorem}
\label{IFTlocal} Let $\mu\in\lbrack\mu_{0},\mu_{1}]\subset(0,\mu_{\max})$,
$P(\rho)$ satisfy \eqref{P1}-\eqref{P2}, and $\omega(r)\in C^{1,\beta
}[0,\infty)$. Then there exist $\tilde{\kappa}>0$ and solutions $\rho
_{\mu,\kappa}$ of \eqref{iftw1} for all $|\kappa|<\tilde{\kappa}$, satisfying
the following properties:\newline1) $\rho_{\mu,\kappa}\in C_{c}^{1,\alpha
}(\mathbb{R}^{3})$, where $\alpha=\min(\frac{2-\mathbb{\gamma}_{0}%
}{\mathbb{\gamma}_{0}-1},1)$. \newline2) $\rho_{\mu,\kappa}$ is axi-symmetric
and even in $z$. \newline3) $\rho_{\mu,\kappa}(0)=\mu$. \newline4) $\rho
_{\mu,\kappa}\geq0$ has compact support $g_{\zeta_{\mu,\kappa}}(B_{\mu})$.
\newline5) For all $\mu\in\lbrack\mu_{0},\mu_{1}]$, the mapping $\kappa
\rightarrow\rho_{\mu,\kappa}$ is continuous from $(-\tilde{\kappa}%
,\tilde{\kappa})$ into $C_{c}^{1}(\mathbb{R}^{3})$.

When $\kappa=0$, $\rho_{\mu,0}=$ $\rho_{\mu}(|x|)$ is the nonrotating star
solution with $\rho_{\mu}(0)=\mu$.

\end{theorem}

Now we use Theorem \ref{Th: rayleigh stable} to study the stability of above
rotating star solutions $(\rho_{\mu,\kappa},\kappa\omega(r)r\mathbf{e}%
_{\theta})$, for $\mu\in\lbrack\mu_{0},\mu_{1}]$, $\kappa$ small enough, and
$\omega\in C^{1,\beta}[0,\infty)$ satisfying the Rayleigh condition
$\Upsilon(r):=\frac{\partial_{r}(\omega^{2}r^{4})}{r^{3}}>0$. First, we check
the assumptions in Theorem \ref{Th: rayleigh stable}. Let $R_{\mu,\kappa}$ be
the support radius in $r$ for $\Omega_{\mu,\kappa}=g_{\zeta_{\mu,\kappa}%
}(B_{\mu})$. Since $g_{\zeta_{\mu,\kappa}}\in C^{2}(B_{\mu})$ dependents
continuously on $\kappa$, it is easy to check the assumptions on $\Omega
_{\mu,\kappa}$ for $\kappa$ small enough. That is, $\partial\Omega_{\mu
,\kappa}$ is $C^{2}$ and has positive curvature near $(R_{\mu,\kappa},0)$.
Next, we check the assumption \eqref{rhonearboundary}. For nonrotating stars,
it is known (\cite{CS1939,JJ2014,LS1997, LZ2019}) that
\[
\rho_{\mu}(r,z)\approx((R_{\mu}-\sqrt{r^{2}+z^{2}})^{\frac{1}{\mathbb{\gamma
}_{0}-1}})\text{ for }\sqrt{r^{2}+z^{2}}\sim R_{\mu}.
\]
For $\kappa$ small enough, by the definition of the dilating function
$g_{\zeta_{\mu,\kappa}}$, we have
\begin{align*}
\rho_{\mu,\kappa}(r,z) &  =\rho_{\mu}(g_{\zeta_{\mu,\kappa}}^{-1}(r,z))\\
&  \approx((R_{\mu}-|g_{\zeta_{\mu,\kappa}}^{-1}(r,z)|)^{\frac{1}%
{\mathbb{\gamma}_{0}-1}})\\
&  \approx\text{dist}((r,z),\partial g_{\zeta_{\mu,\kappa}}(B_{\mu}%
))^{\frac{1}{\mathbb{\gamma}_{0}-1}},
\end{align*}
for $\left(  r,z\right)  $ near $(R_{\mu,\kappa},0)=g_{\zeta_{\mu,\kappa}%
}(R_{\mu},0)$.

Below, for rotating stars $\left(  \rho_{\mu,\kappa},r\omega_{0}\left(
r\right)  \mathbf{e}_{\theta}\right)  $ we use $X_{\mu,\kappa}$, $X_{1}%
^{\mu,\kappa}$, $Y_{\mu,\kappa}$, $L_{\mu,\kappa}$, $A_{1}^{\mu,\kappa}$,
$B_{1}^{\mu,\kappa}$, $B_{2}^{\mu,\kappa}$, $K_{\mu,\kappa}$, etc., to denote
the corresponding spaces $X$, $X_{1}$, $Y$, and operators $L$, $A_{1}$,
$B_{1}$, $B_{2}$, $\mathcal{K}$ etc. defined in Section 2.

By Theorem \ref{Th: rayleigh stable}, the rotating star $(\rho_{\mu,\kappa
},\kappa\omega(r)r\mathbf{e}_{\theta})$ is spectrally stable if and only if
\begin{equation}
\langle K_{\mu,\kappa}\delta\rho,\delta\rho\rangle=\langle L_{\mu,\kappa
}\delta\rho,\delta\rho\rangle+2\kappa^{2}\pi\int_{0}^{R_{\mu,\kappa}}%
\Upsilon(r)\frac{\left(  \int_{0}^{r}s\int_{-\infty}^{+\infty}\delta
\rho(s,z)dzds\right)  ^{2}}{r\int_{-\infty}^{+\infty}\rho_{\mu,\kappa}%
({r},z)dz}dr\geq0, \label{NSCSkappa}%
\end{equation}
for all
\[
\delta\rho\in R(B_{1}^{\mu,\kappa})=\left\{  \delta\rho\in X_{1}^{\mu,\kappa
}|\int_{\mathbb{R}^{3}}\delta\rho dx=0\right\}  .
\]
Moreover, the number of unstable modes equals $n^{-}\left(  K_{\mu,\kappa
}|_{R(B_{1}^{\mu,\kappa})}\right)  $. The following is an easy corollary of
the stability criterion.

\begin{corollary}
\label{C2}(Sufficient condition for instability) \newline Let $I\subset
\lbrack\mu_{0},\mu_{1}]$ be an interval such that the non-rotating star
$(\rho_{\mu},0)$ is unstable for any $\mu\in I$. Then for any $\omega\in
C^{1,\beta}[0,\infty)$ satisfies $\Upsilon(r)>0$, there exists $\kappa_{0}>0$
such that the rotating star $(\rho_{\mu,\kappa},\kappa\omega(r)r\mathbf{e}%
_{\theta})$ is unstable for any $0<\kappa<\kappa_{0}$ and $\mu\in I$.
\end{corollary}

\begin{proof}
The instability of $(\rho_{\mu},0)$ implies that $n^{-}(L_{\mu,0}%
|_{R(B_{1}^{\mu,0})})>0$ for $\mu\in I$. Thus there exists some $\epsilon>0$
(independent of $\mu$) and $\delta\rho_{\mu,0}=\delta\rho_{\mu,0}(|x|)\in
R(B_{1}^{\mu,0})\ $such that $\langle L_{\mu,0}\delta\rho_{\mu,0},\delta
\rho_{\mu,0}\rangle=-2\epsilon<0$ for $\mu\in I$. Let
\[
\delta\rho_{\mu,\kappa}(r,z)=\delta\rho_{\mu,0}(g_{\zeta_{\mu,\kappa}%
}(r,z))-\frac{\int_{B_{\mu}}\delta\rho_{\mu,0}(|x|)\det Dg_{\zeta_{\mu,\kappa
}}(x)dx}{M_{\mu,\kappa}}\rho_{\mu,\kappa}(r,z),
\]
then $\delta\rho_{\mu,\kappa}(r,z)\in R(B_{1}^{\mu,\kappa})$. Noticing that
\[
\lim_{\kappa\rightarrow0}\int_{B_{\mu}}\delta\rho_{\mu,0}(|x|)\det
Dg_{\zeta_{\mu,\kappa}}(x)dx=\int_{B_{\mu}}\delta\rho_{\mu,0}(|x|)dx=0,
\]
we have
\[
\lim_{\kappa\rightarrow0}\langle L_{\mu,\kappa}\delta\rho_{\mu,\kappa}%
,\delta\rho_{\mu,\kappa}\rangle=\langle L_{\mu,0}\delta\rho_{\mu,0},\delta
\rho_{\mu,0}\rangle=-2\epsilon<0.
\]
Thus, there exists $\kappa_{0}>0$ such that when $0<\kappa<\kappa_{0}$
\begin{align*}
&  \langle K_{\mu,\kappa}\delta\rho_{\mu,\kappa},\delta\rho_{\mu,\kappa
}\rangle\\
&  =\langle L_{\mu,\kappa}\delta\rho_{\mu,\kappa},\delta\rho_{\mu,\kappa
}\rangle+2\kappa^{2}\pi\int_{0}^{R_{\mu,\kappa}}\Upsilon(r)\frac{\left(
\int_{0}^{r}s\int_{-\infty}^{+\infty}\delta\rho_{\mu,\kappa}(s,z)dzds\right)
^{2}}{r\int_{-\infty}^{+\infty}\rho_{\mu,\kappa}({r},z)dz}dr<-\epsilon<0.
\end{align*}
The linear instability of $(\rho_{\mu,\kappa},\kappa\omega(r)r\mathbf{e}%
_{\theta})$ follows.
\end{proof}

Let $\tilde{\mu}$ be the first critical point of the mass-radius ratio
$\frac{M_{\mu}}{R_{\mu}}$ for the nonrotating stars and set $\tilde{\mu
}=+\infty$ if $\frac{M_{\mu}}{R_{\mu}}$ has no critical point. Consider the
rotating stars $(\rho_{\mu,\kappa},\kappa\omega(r)r\mathbf{e}_{\theta})$ for
$\mu\in$ $[\mu_{0},\mu_{1}]\subset\left(  0,\tilde{\mu}\right)  $ and $\kappa$
small. We have the following sufficient condition for stability.

\begin{theorem}
\label{Th: stability-fixed velocity}(Sufficient condition for stability)
\newline Suppose $P(\rho)$ satisfies \eqref{P1}-\eqref{P2}, and $\omega\in
C^{1,\beta}[0,\infty)$ satisfies $\Upsilon(r)>0$. For any $\mu\in\lbrack
\mu_{0},\mu_{1}]\subset\left(  0,\tilde{\mu}\right)  $ and $\kappa$ small
enough, if $\frac{dM_{\mu,\kappa}}{d\mu}\geq0,$ then the rotating star
$(\rho_{\mu,\kappa},\kappa\omega r\mathbf{e}_{\theta})$ is spectrally stable.
\end{theorem}

For the proof of above Theorem, first we compute $n^{-}\left(  L_{\mu,\kappa
}|_{X_{\mu,\kappa}}\right)  $. Let $\dot{H}_{ax}^{1}$ and $\dot{H}_{ax}^{-1}$
be the axi-symmetric subspaces of $\dot{H}^{1}(\mathbb{R}^{3})$ and $\dot
{H}^{-1}(\mathbb{R}^{3})$ respectively. Define the reduced operator
$D_{\mu,\kappa}:\dot{H}_{ax}^{1}\rightarrow\dot{H}_{ax}^{-1}$ by
\[
D_{\mu,\kappa}:=-\Delta-\frac{4\pi}{\Phi^{\prime\prime}(\rho_{\mu,\kappa})}.
\]
Then
\[
\langle D_{\mu,\kappa}\psi,\psi\rangle=\int_{\mathbb{R}^{3}}|\nabla\psi
|^{2}dx-4\pi\int_{\mathbb{R}^{3}}\frac{|\psi|^{2}}{\Phi^{\prime\prime}%
(\rho_{\mu,\kappa})}dx,\ \psi\in\dot{H}_{ax}^{1},
\]
defines a bounded bilinear symmetric form on $\dot{H}_{ax}^{1}$. By the same
proof of Lemma 3.7 in \cite{LZ2019}, we have

\begin{lemma}
\label{lemma-equivalent-operators} It holds that $n^{-}\left(  L_{\mu,\kappa
}|_{X_{1}^{\mu,\kappa}}\right)  =n^{-}\left(  D_{\mu,\kappa}\right)  $ and
$\dim\ker L_{\mu,\kappa}=\dim\ker D_{\mu,\kappa}$.
\end{lemma}

Since the rotating star solution $(\rho_{\mu,\kappa},\kappa\omega
(r)r\mathbf{e}_{\theta})$ is even in $z$, we can compute $n^{-}\left(
L_{\mu,\kappa}|_{X_{1}^{\mu,\kappa}}\right)  $ and $n^{-}\left(  D_{\mu
,\kappa}\right)  $ on the even and odd (in $z$) subspaces respectively.
Define
\begin{align}
&  X_{od}^{\mu,\kappa}:=\{\rho\in X_{1}^{\mu,\kappa}|\ \rho(r,z)=-\rho
(r,-z)\},\ X_{ev}^{\mu,\kappa}:=\{\rho\in X_{1}^{\mu,\kappa}\ |\ \rho
(r,z)=\rho(r,-z)\},\label{defn-spaces-parity}\\
&  H^{od}:=\{\varphi\in\dot{H}_{ax}^{1}\ |\varphi(r,z)=-\varphi
(r,-z)\},\ H^{ev}:=\{\varphi\in\dot{H}_{ax}^{1}\ |\ \varphi(r,z)=\varphi
(r,-z)\}.\nonumber
\end{align}

\begin{lemma}
\label{N-Dmukappa} Assume $P(\rho)$ satisfies \eqref{P1}-\eqref{P2},
$\omega\in C^{1,\beta}[0,\infty)$ satisfies $\Upsilon(r)>0$. Then for any
$\mu\in\lbrack\mu_{0},\mu_{1}]\subset\left(  0,\tilde{\mu}\right)  $ and
$\kappa$ small enough, we have $n^{-}(L_{\mu,\kappa})=n^{-}(L_{\mu,0})=1$ and
$\ker L_{\mu,\kappa}=span\{\partial_{z}\rho_{\mu,\kappa}\}$. Moreover, we have
the following direct sum decompositions for $X_{ev}^{\mu,\kappa}$ and
$X_{ev}^{\mu,\kappa}:$
\[
X_{ev}^{\mu,\kappa}=X_{-,ev}^{\mu,\kappa}\oplus X_{+,ev}^{\mu,\kappa},\ \ \dim
X_{-,ev}^{\mu,\kappa}=1,
\]
and
\[
X_{od}^{\mu,\kappa}=span\{\partial_{z}\rho_{\mu,\kappa}\}\oplus X_{+,od}%
^{\mu,\kappa},
\]
satisfying: i) $L_{\mu,\kappa}|_{X_{-,ev}^{\mu,\kappa}}<0;$

ii) there exists $\delta>0$ such that
\[
\left\langle L_{\mu,\kappa}u,u\right\rangle \geq\delta\left\Vert u\right\Vert
_{L_{\Phi^{\prime\prime}(\rho_{\mu,\kappa})}^{2}}^{2}\ ,\text{ for any }u\in
X_{+,ev}^{\mu,\kappa}\oplus X_{+,od}^{\mu,\kappa},
\]
where $\delta$ is independent of $\mu$ and $\kappa$.

The same decompositions are also true for $K_{\mu,\kappa}\ $on $X_{ev}%
^{\mu,\kappa}$ and $X_{od}^{\mu,\kappa}$. In addition, for any $\mu\in
\lbrack\mu_{0},\mu_{1}]$, it holds that $\frac{dV_{\mu,\kappa}(0,Z_{\mu
,\kappa})}{d\mu}<0$ for $\kappa$ small enough.

\end{lemma}

\begin{proof}
It was showed in \cite{LZ2019} that: for any $\mu\in\left(  0,\tilde{\mu
}\right)  $, we have $n^{-}(D_{\mu,0})=1$ and $\ker D_{\mu,0}=span\{\partial
_{z}V_{\mu}\}\ $in the axi-symmetric function space. Here, $V_{\mu}=$
$-|x|^{-1}\ast\rho_{\mu}\ $is the gravitational potential of the non-rotating
star. Since $\partial_{z}V_{\mu}$ is odd in $z$, it follows that for any
$\mu\in\left(  0,\tilde{\mu}\right)  $: \ i) on $H^{ev}$, $n^{-}(D_{\mu,0}%
)=1$, $\ker D_{\mu,0}=\{0\}$; ii) on $H^{od}$, $\ker D_{\mu,0}=span\{\partial
_{z}V_{\mu}\}$ and $n^{-}(D_{\mu,0})=0$. Moreover, for $\mu\in\lbrack\mu
_{0},\mu_{1}]\subset\left(  0,\tilde{\mu}\right)  $, there exists $\delta
_{0}>0$ (independent of $\mu$) and decompositions $H^{ev}=H_{-,\mu}^{ev}\oplus
H_{+,\mu}^{ev}$ and $H^{od}=span\{\partial_{z}V_{\mu}\}\oplus H_{+,\mu}^{od}$
satisfying that: i) $\dim H_{-,\mu}^{ev}=1,\ D_{\mu,0}|_{H_{-,\mu}^{ev}%
}<-\delta_{0}$; ii) $D_{\mu,0}|_{H_{+,\mu}^{ev}\oplus H_{+,\mu}^{od}}%
\geq\delta_{0}$. Since $\partial_{z}V_{\mu,\kappa}\in H^{od}\cap\ker
D_{\mu,\kappa}$ and
\begin{align*}
\langle(D_{\mu,\kappa}-D_{\mu,0})\psi,\psi\rangle &  =\int\left(  \frac{4\pi
}{\Phi^{\prime\prime}(\rho_{\mu,\kappa})}-\frac{4\pi}{\Phi^{\prime\prime}%
(\rho_{\mu})}\right)  \psi^{2}dx\\
&  \lesssim\left(  \int\left(  \frac{4\pi}{\Phi^{\prime\prime}(\rho
_{\mu,\kappa})}-\frac{4\pi}{\Phi^{\prime\prime}(\rho_{\mu})}\right)
^{\frac{3}{2}}dx\right)  ^{\frac{2}{3}}\Vert\psi\Vert_{L^{6}}^{2}\\
&  \lesssim O(\kappa)\Vert\nabla\psi\Vert_{L^{2}}^{2}\rightarrow0,\text{ as
}\kappa\rightarrow0,
\end{align*}
by the perturbation arguments (e.g. Corollary 2.19 in \cite{LZ2019}) it
follows that for $\mu\in\lbrack\mu_{0},\mu_{1}]$ and $\kappa$ sufficiently
small, the decompositions $H^{ev}=H_{-,\mu}^{ev}\oplus H_{+,\mu}^{ev}$ and
$H^{od}=span\{\partial_{z}V_{\mu,\kappa}\}\oplus H_{+,\mu}^{ev}$ satisfy: i)
$\dim H_{-,\mu}^{ev}=1,\ D_{\mu,\kappa}|_{H_{-,\mu}^{ev}}<-\frac{1}{2}%
\delta_{0}$; ii) $D_{\mu,\kappa}|_{H_{+,\mu}^{ev}\oplus H_{+,\mu}^{ev}}%
\geq\frac{1}{2}\delta_{0}$.

By the proof of Lemma 3.4 in \cite{LZ2019}, for any $\rho\in X_{1}^{\mu
,\kappa}\ $we have
\begin{equation}
\left\langle L_{\mu,\kappa}\rho,\rho\right\rangle =\left\Vert \rho\right\Vert
_{L_{\Phi^{\prime\prime}(\rho_{\mu,\kappa})}^{2}}^{2}-\frac{1}{4\pi}%
\Vert\nabla\psi\Vert_{L^{2}}^{2}\geq\frac{1}{4\pi}\left\langle D_{\mu,\kappa
}\psi,\psi\right\rangle ,\ \ \label{quadratic lower bound}%
\end{equation}
where $\psi=\frac{1}{4\pi}\Delta^{-1}\rho$. We note that $\partial_{z}%
\rho_{\mu,\kappa}\in\ker L_{\mu,\kappa}\cap X_{od}^{\mu,\kappa}$ and
$\partial_{z}V_{\mu,\kappa}=\frac{1}{4\pi}\Delta^{-1}\rho_{\mu,\kappa}$. The
existence of decompositions for $X_{ev}^{\mu,\kappa}$ and $X_{od}^{\mu,\kappa
}$ as stated in the lemma follows readily from (\ref{quadratic lower bound})
and above decompositions for $H^{od}$ and $H^{ev}$.

Since
\[
\left\vert \left\langle \left(  L_{\mu,\kappa}-K_{\mu,\kappa}\right)
\rho,\rho\right\rangle \right\vert \lesssim o(\kappa^{2})\left\Vert
\rho\right\Vert _{L_{\Phi^{\prime\prime}(\rho_{\mu,\kappa})}^{2}}%
^{2},\ \ \forall\rho\in X_{1}^{\mu,\kappa},
\]
and $\partial_{z}\rho_{\mu,\kappa}\in\ker K_{\mu,\kappa}\cap X_{od}%
^{\mu,\kappa}$, we have the same decompositions for $K_{\mu,\kappa}$ on
$X_{ev}^{\mu,\kappa}$ and $X_{od}^{\mu,\kappa}$.

Since $\gamma_{0}\in(6/5,2)$, it is known that (see \cite{LZ2019})
\[
\frac{dV_{\mu}(0,R_{\mu})}{d\mu}=-\frac{d}{d\mu}\left(  \frac{M_{\mu}}{R_{\mu
}}\right)  <0
\]
for $\mu$ small. Recall that $\tilde{\mu}$ is the first critical point of
$\frac{M_{\mu}}{R_{\mu}}$. Therefore, when $\mu\in\lbrack\mu_{0},\mu
_{1}]\subset\left(  0,\tilde{\mu}\right)  $, we have $\frac{dV_{\mu}(0,R_{\mu
})}{d\mu}<-\epsilon_{0}$ for some constant $\epsilon_{0}>0$ independent of
$\mu$. Since $\left\vert \frac{dV_{\mu,\kappa}(0,Z_{\mu,\kappa})}{d\mu}%
-\frac{dV_{\mu}(0,R_{\mu})}{d\mu}\right\vert =O(\kappa)$, we have
$\frac{dV_{\mu,\kappa}(0,Z_{\mu,\kappa})}{d\mu}<0$ for any $\mu\in\lbrack
\mu_{0},\mu_{1}]$ and $\kappa$ small enough. This finishes the proof of the lemma.
\end{proof}

\begin{proof}
[Proof of Theorem \ref{Th: stability-fixed velocity}]The spectral stability of
$(\rho_{\mu,\kappa},\kappa\omega r\mathbf{e}_{\theta})$ is equivalent to show
$n^{-}\left(  K_{\mu,\kappa}|_{R(B_{1}^{\mu,\kappa})}\right)  =0$. By Lemma
\ref{N-Dmukappa} and the fact that $K_{\mu,\kappa}=L_{\mu,\kappa}\ $on
$X_{od}^{\mu,\kappa}$, we have
\[
n^{-}(K_{\mu,\kappa}|_{X_{od}^{\mu,\kappa}\cap R(B_{1}^{\mu,\kappa})}%
)=n^{-}(L_{\mu,\kappa}|_{X_{od}^{\mu,\kappa}\cap R(B_{1}^{\mu,\kappa})})\leq
n^{-}(L_{\mu,\kappa}|_{X_{od}^{\mu,\kappa}})=0.
\]
Since $K_{\mu,\kappa}\geq L_{\mu,\kappa}$ on $X_{ev}^{\mu,\kappa}\ $due to
$\Upsilon(r)>0$, for spectral stability it suffices to show $n^{-}\left(
L_{\mu,\kappa}|_{X_{ev}^{\mu,\kappa}\cap R(B_{1}^{\mu,\kappa})}\right)  =0$.

Applying $\frac{d}{d\mu}$ to \eqref{iftw1}, we obtain that
\[
L_{\mu,\kappa}\frac{d\rho_{\mu,\kappa}}{d\mu}=-\frac{dc_{\mu,\kappa}}{d\mu}.
\]
From \eqref{iftw1} we know that $c_{\mu,\kappa}=-V_{\mu,\kappa}(R_{\mu,\kappa
},0)$. By Lemma \ref{N-Dmukappa}, $\frac{dc_{\mu,\kappa}}{d\mu}>0$ for $\mu
\in\lbrack\mu_{0},\mu_{1}]$ and $\kappa$ small enough. Therefore,
\[
X_{ev}^{\mu,\kappa}\cap R(B_{1}^{\mu,\kappa})=\left\{  \delta\rho\in
X_{ev}^{\mu,\kappa}\ |\ \left\langle L_{\mu,\kappa}\frac{d\rho_{\mu,\kappa}%
}{d\mu},\delta\rho\right\rangle =0\right\}  ,
\]
i.e. $\delta\rho$ is orthogonal to $\frac{d\rho_{\mu,\kappa}}{d\mu}\ $in
$\left\langle L_{\mu,\kappa}\cdot,\cdot\right\rangle $.

When $\frac{dM_{\mu,\kappa}}{d\mu}>0$, we have
\[
\left\langle L_{\mu,\kappa}\frac{d\rho_{\mu,\kappa}}{d\mu},\frac{d\rho
_{\mu,\kappa}}{d\mu}\right\rangle =-\frac{dc_{\mu,\kappa}}{d\mu}\int
_{g_{\zeta_{\mu,\kappa}}(B_{\mu})}\frac{d\rho_{\mu,\kappa}}{d\mu}%
dx=\frac{dV_{\mu,\kappa}(0,Z_{\mu,\kappa})}{d\mu}\frac{dM_{\mu,\kappa}}{d\mu
}<0.
\]
Combining above with Lemma \ref{N-Dmukappa}, we get $n^{-}(L_{\mu,\kappa
}|_{X_{ev}^{\mu,\kappa}\cap R(B_{1}^{\mu,\kappa})})=0$. Hence we get the
spectrally stability.

When $\frac{dM_{\mu,\kappa}}{d\mu}=0$, since
\[
\frac{dM_{\mu,\kappa}}{d\mu}=\int\frac{d\rho_{\mu,\kappa}}{d\mu}dx=0,
\]
we have $\frac{d\rho_{\mu,\kappa}}{d\mu}\in X_{ev}^{\mu,\kappa}\cap
R(B_{1}^{\mu,\kappa})$. Meanwhile, since $\ker L_{\mu,\kappa}=\{0\}$ on
$X_{ev}^{\mu,\kappa}$, by the same argument as in the proof of Theorem 1.1 in
\cite{LZ2019}, we have $n^{-}(L_{\mu,\kappa}|_{X_{ev}^{\mu,\kappa}\cap
R(B_{1}^{\mu,\kappa})})=0$. The spectral stability is again true.
\end{proof}

It is natural to ask if extrema points of the total mass $M_{\mu,\kappa}\ $of
the rotating stars $(\rho_{\mu,\kappa},\kappa\omega r\mathbf{e}_{\theta}%
)\ $are the transition points for stability as in the case of nonrotating
stars. Below, we show that this is not true.

First, we give conditions to ensure that the first extrema point of total mass
$M_{\mu,\kappa}$ is obtained at a center density $\mu_{\ast}^{\kappa}$ before
$\tilde{\mu}$ (the first critical point of $M_{\mu}/R_{\mu}$). Assume
$P(\rho)\ $satisfies the following asymptotically polytropic conditions:

H1)
\begin{equation}
P(\rho)=c_{-}\rho^{\gamma_{0}}(1+O(\rho^{a_{0}}))\text{ when }\rho
\rightarrow0,\label{approxiP1}%
\end{equation}
for some $\gamma_{0}\in(\frac{4}{3},2)$ and $c_{-},\ a_{0}>0$;\newline

H2)
\begin{equation}
P(\rho)=c_{+}\rho^{\gamma_{\infty}}(1+O(\rho^{-a_{\infty}}))\text{ when }%
\rho\rightarrow+\infty,\label{approxiP2}%
\end{equation}
for some $\gamma_{\infty}\in(1,6/5)\cup(6/5,4/3)$ and $c_{+},\ a_{\infty}>0$.

Under assumptions H1)-H2), it was shown in \cite{HU2003} that the total mass
$M_{\mu}$ of the non-rotating stars has extrema points. Moreover, the first
extrema point of $M_{\mu}$, which is a maximum point denoted by $\mu_{\ast}$,
must be less than $\tilde{\mu}$ (see Lemma 3.14 in \cite{LZ2019}). For any
$\mu_{0}<\mu_{\ast}<\mu_{1}<$ $\tilde{\mu}$, we have $M_{\mu,\kappa
}\rightarrow M_{\mu}$ in $C^{1}\left[  \mu_{0},\mu_{1}\right]  $ when
$\kappa\rightarrow0$. Thus when $\kappa$ is small enough, the function
$M_{\mu,\kappa}$ has the first maximum $\mu_{\ast}^{\kappa}\in\left(  \mu
_{0},\mu_{1}\right)  $ and $\lim_{\kappa\rightarrow0}\mu_{\ast}^{\kappa}%
=\mu_{\ast}$. By Theorem \ref{Th: stability-fixed velocity}, the rotating
stars $(\rho_{\mu,\kappa},\kappa\omega(r)r\mathbf{e}_{\theta})\ $are stable
for $\mu\in\left[  \mu_{0},\mu_{\ast}^{\kappa}\right]  $. It is shown below
that the transition of stability occurs beyond $\mu_{\ast}^{\kappa}$.

\begin{theorem}
\label{Th: TPP-fixed velocity} Suppose $P(\rho)$ satisfies
\eqref{approxiP1}-\eqref{approxiP2}, $\omega\in C^{1,\beta}[0,\infty)$
satisfies $\Upsilon(r)>0$. Fixed $\kappa$ small, let $\hat{\mu}_{\kappa}\ $be
the first transition point of stability of the rotating stars $(\rho
_{\mu,\kappa},\kappa\omega(r)r\mathbf{e}_{\theta})$. Then for any $\kappa
\neq0$ small enough, we have $\hat{\mu}_{\kappa}>\mu_{\ast}^{\kappa}.$
\end{theorem}

\begin{proof}
As in the proof of Theorem \ref{Th: stability-fixed velocity}, the spectral
stability is equivalent to show $K_{\mu,\kappa}\geq0$ on $X_{ev}^{\mu,\kappa
}\cap R(B_{1}^{\mu,\kappa})$. Suppose the maxima point $\mu_{\ast}^{\kappa}$
of $M_{\mu,\kappa}$ is the first transition point for stability, then we have
\begin{equation}
\inf_{\rho\in X_{ev}^{\mu_{\ast}^{\kappa},\kappa}\cap R(B_{1}^{\mu_{\ast
}^{\kappa},\kappa})}\frac{\langle K_{\mu_{\ast}^{\kappa},\kappa}\rho
,\rho\rangle}{\Vert\rho\Vert_{L_{\Phi^{\prime\prime}(\rho_{\mu_{\ast}^{\kappa
},\kappa})}^{2}}}=0.\label{inf-even}%
\end{equation}
By Lemma \ref{N-Dmukappa}, when $\kappa$ is small enough, we have the
decomposition
\[
X_{ev}^{\mu_{\ast}^{\kappa},\kappa}=X_{-,ev}^{\mu_{\ast}^{\kappa},\kappa
}\oplus X_{+,ev}^{\mu_{\ast}^{\kappa},\kappa},\ \ \dim X_{-,ev}^{\mu_{\ast
}^{\kappa},\kappa}=1,
\]
satisfying: i) $K_{\mu_{\ast}^{\kappa},\kappa}|_{X_{-,ev}^{\mu_{\ast}^{\kappa
},\kappa}}<0;\ $ii) there exists $\delta>0$ such that
\[
\left\langle K_{\mu_{\ast}^{\kappa},\kappa}\rho,\rho\right\rangle \geq
\delta\left\Vert \rho\right\Vert _{L_{\Phi^{\prime\prime}(\rho_{\mu_{\ast
}^{\kappa},\kappa})}^{2}}^{2}\ ,\text{ for any }\rho\in X_{+,ev}^{\mu_{\ast
}^{\kappa},\kappa}.
\]
By using above decomposition, it is easy to show that the infimum in
(\ref{inf-even}) is obtained by some $\rho^{\ast}\in X_{ev}^{\mu_{\ast
}^{\kappa},\kappa}\cap R(B_{1}^{\mu_{\ast}^{\kappa},\kappa})$. Then
\[
\langle L_{\mu_{\ast}^{\kappa},\kappa}\rho^{\ast},\rho^{\ast}\rangle
\leq\langle K_{\mu_{\ast}^{\kappa},\kappa}\rho^{\ast},\rho^{\ast}\rangle=0.
\]
On the other hand, we have
\[
\left\langle L_{\mu_{\ast}^{\kappa},\kappa}\frac{d\rho_{\mu,\kappa}}{d\mu
}|_{\mu=\mu_{\ast}^{\kappa}},\frac{d\rho_{\mu,\kappa}}{d\mu}|_{\mu=\mu_{\ast
}^{\kappa}}\right\rangle =\frac{dV_{\mu,\kappa}(0,Z_{\mu,\kappa})}{d\mu}%
|_{\mu=\mu_{\ast}^{\kappa}}\frac{dM_{\mu,\kappa}}{d\mu}|_{\mu=\mu_{\ast
}^{\kappa}}=0,
\]
and
\[
\left\langle L_{\mu_{\ast}^{\kappa},\kappa}\frac{d\rho_{\mu,\kappa}}{d\mu
}|_{\mu=\mu_{\ast}^{\kappa}},\rho^{\ast}\right\rangle =\frac{dV_{\mu,\kappa
}(0,Z_{\mu,\kappa})}{d\mu}|_{\mu=\mu_{\ast}^{\kappa}}\int\rho^{\ast}dx=0.
\]
This implies that $\rho^{\ast}=c\frac{d\rho_{\mu,\kappa}}{d\mu}|_{\mu
=\mu_{\ast}^{\kappa}}$ for some constant $c\neq0$. Since otherwise,
\[
n^{\leq0}(L_{\mu_{\ast}^{\kappa},\kappa}|_{X_{ev}^{\mu_{\ast}^{\kappa},\kappa
}})\geq n^{\leq0}(L_{\mu_{\ast}^{\kappa},\kappa}|_{span\left\{  \frac
{d\rho_{\mu,\kappa}}{d\mu}|_{\mu=\mu_{\ast}^{\kappa}},\rho^{\ast}\right\}
})=2.
\]
which is in contradiction to $n^{\leq0}(L_{\mu_{\ast}^{\kappa},\kappa
}|_{X_{ev}^{\mu_{\ast}^{\kappa},\kappa}})=1$. Thus, we have
\begin{align*}
0 &  =\left\langle K_{\mu_{\ast}^{\kappa},\kappa}\frac{d\rho_{\mu,\kappa}%
}{d\mu}|_{\mu=\mu_{\ast}^{\kappa}},\frac{d\rho_{\mu,\kappa}}{d\mu}|_{\mu
=\mu_{\ast}^{\kappa}}\right\rangle \\
&  =2\pi\kappa^{2}\int_{0}^{+\infty}\Upsilon(r)\frac{\left(  \int_{0}^{r}%
s\int_{-\infty}^{+\infty}\frac{d\rho_{\mu,\kappa}}{d\mu}|_{\mu=\mu_{\ast
}^{\kappa}}(s,z)dzds\right)  ^{2}}{r\int_{-\infty}^{+\infty}\rho_{\mu_{\ast
}^{\kappa},\kappa}({r},z)dz}dr.
\end{align*}
and consequently
\begin{equation}
\int_{-\infty}^{+\infty}\frac{d\rho_{\mu,\kappa}}{d\mu}|_{\mu=\mu_{\ast
}^{\kappa}}(r,z)dz=0,\ \ \ \forall r\in\lbrack0,R_{\mu_{\ast}^{\kappa},\kappa
}].\label{iden zero}%
\end{equation}
Nevertheless, it is not true as shown below.

For non-rotating stars $(\rho_{\mu}(\mathbf{r}),0)$, we have
\[
\Delta V_{\mu}=\frac{1}{\mathbf{r}^{2}}\left(  \mathbf{r}^{2}\left(  V_{\mu
}(\mathbf{r})\right)  ^{\prime}\right)  ^{\prime}=4\pi\rho_{\mu},
\]
where $\mathbf{r}=\sqrt{r^{2}+z^{2}}$ and $V_{\mu}(\mathbf{r})$ is the
gravitational potential. Applying $\frac{d}{d\mu}$ to above equation, one has
\[
\frac{1}{\mathbf{r}^{2}}\left(  \mathbf{r}^{2}\left(  \frac{dV_{\mu
}(\mathbf{r})}{d\mu}\right)  ^{\prime}\right)  ^{\prime}=4\pi\frac{d\rho_{\mu
}}{d\mu}.
\]
When $\mathbf{r}\geq R_{\mu}$, since $\frac{d\rho_{\mu}}{d\mu}\left(
\mathbf{r}\right)  =0$ we have
\[
\mathbf{r}^{2}\left(  \frac{dV_{\mu}}{d\mu}\right)  ^{\prime}(\mathbf{r}%
)=R_{\mu}^{2}\left(  \frac{dV_{\mu}}{d\mu}\right)  ^{\prime}(R_{\mu})=4\pi
\int_{0}^{R_{\mu}}s^{2}\frac{d\rho_{\mu}}{d\mu}(s)ds=\frac{dM_{\mu}}{d\mu},
\]
and consequently%
\[
\frac{dV_{\mu}}{d\mu}(\mathbf{r})=-\frac{dM_{\mu}}{d\mu}\frac{1}{\mathbf{r}%
},\ \text{for }\mathbf{r}\geq R_{\mu}.\
\]
Since $\lim_{\kappa\rightarrow0}\mu_{\ast}^{\kappa}=\mu_{\ast}$, we have
$\lim_{\kappa\rightarrow0}\frac{dM_{\mu}}{d\mu}\left(  \mu_{\ast}^{\kappa
}\right)  =\frac{dM_{\mu}}{d\mu}\left(  \mu_{\ast}\right)  =0$. Thus
\[
\frac{dV_{\mu}}{d\mu}(R_{\mu})|_{\mu=\mu_{\ast}^{\kappa}}=-\frac{dM_{\mu}%
}{d\mu}\left(  \mu_{\ast}^{\kappa}\right)  \frac{1}{R_{\mu_{\ast}^{\kappa}}%
}\rightarrow0,\ \text{as }\kappa\rightarrow0\text{. }%
\]
Define $y_{\mu}(\mathbf{r})=V_{\mu}(R_{\mu})-V_{\mu}(\mathbf{r})=\Phi^{\prime
}(\rho_{\mu})$. Then by Lemma 3.13 in \cite{LZ2019}, we have
\begin{align}
\frac{dy_{\mu}}{d\mu}(R_{\mu})|_{\mu=\mu_{\ast}^{\kappa}} &  =-\frac{d}{d\mu
}\left(  \frac{M_{\mu}}{R_{\mu}}\right)  |_{\mu=\mu_{\ast}^{\kappa}}%
-\frac{dV_{\mu}}{d\mu}(R_{\mu})|_{\mu=\mu_{\ast}^{\kappa}}\label{dydmuneq0}\\
&  \rightarrow-\frac{d}{d\mu}\left(  \frac{M_{\mu}}{R_{\mu}}\right)
|_{\mu=\mu_{\ast}}\neq0\text{, as }\kappa\rightarrow0.\nonumber
\end{align}
Thus by \eqref{dydmuneq0}, we obtain
\[
\frac{d\rho_{\mu}}{d\mu}(\mathbf{r})=\frac{1}{\Phi^{\prime\prime}(\rho_{\mu}%
)}\frac{dy_{\mu}}{d\mu}(\mathbf{r})\approx\rho_{\mu}^{2-\gamma_{0}}%
\approx(R_{\mu}-\mathbf{r})^{\frac{2-\gamma_{0}}{\gamma_{0}-1}},
\]
for $\mathbf{r}\sim R_{\mu}$ and $\mu=\mu_{\ast}^{\kappa}$. By (3.36) and
(4.78) in \cite{SW2017}, we know
\[
\left\vert \frac{dg_{\zeta_{\mu,\kappa}}^{-1}}{d\mu}(y)\right\vert =\left\vert
\lim_{\mu_{1}\rightarrow\mu}\frac{g_{\zeta_{\mu_{1},\kappa}}^{-1}%
-g_{\zeta_{\mu,\kappa}}^{-1}}{\mu_{1}-\mu}(y)\right\vert \leq C\kappa,
\]
for some constant $C$ independent of $\mu$ and $\kappa$. Therefore,
\begin{align*}
\frac{d\rho_{\mu,\kappa}}{d\mu}(r,z) &  =\frac{d\rho_{\mu}(g_{\zeta
_{\mu,\kappa}}^{-1}(r,z))}{d\mu}=\frac{d\rho_{\mu}}{d\mu}(g_{\zeta_{\mu
,\kappa}}^{-1}(r,z))+\frac{d\rho_{\mu}(\mathbf{r})}{d\mathbf{r}}%
|_{\mathbf{r}=g_{\zeta_{\mu,\kappa}}^{-1}(r,z)}\frac{dg_{\zeta_{\mu,\kappa}%
}^{-1}}{d\mu}\\
&  \approx\rho_{\mu}(g_{\zeta_{\mu,\kappa}}^{-1}(r,z))^{2-\gamma_{0}}%
=\rho_{\mu,\kappa}(r,z)^{2-\gamma_{0}},
\end{align*}
for $g_{\zeta_{\mu,\kappa}}^{-1}(r,z)\sim R_{\mu}$ and $\mu=\mu_{\ast}%
^{\kappa}$. By Lemma \ref{lemma-rho-int}, we have
\[
\int_{-\infty}^{+\infty}\frac{d\rho_{\mu,\kappa}}{d\mu}|_{\mu=\mu_{\ast
}^{\kappa}}(r,z)dz\thickapprox\int_{-\infty}^{+\infty}\rho_{\mu,\kappa
}(r,z)^{2-\gamma_{0}}dz\thickapprox(R_{\mu_{\ast}^{\kappa},\kappa}%
-r)^{\frac{2-\gamma_{0}}{\mathbb{\gamma}_{0}-1}+\frac{1}{2}}\neq0,
\]
for $r\sim R_{\mu_{\ast}^{\kappa},\kappa}$. This is in contradiction to
(\ref{iden zero}) and finishes the proof of the theorem.
\end{proof}

\subsection{The case of fixed angular momentum distribution}

Let $j(p,q):\mathbb{R}^{2}\mapsto\mathbb{R}$ be a given function satisfying%

\begin{equation}
j(p,q)\in C^{1,\beta}(\mathbb{R}^{+}\times\mathbb{R}^{+})\text{ and
}j(0,q)=\partial_{p}j(0,q)=0. \label{j-smooth}%
\end{equation}
Define $J(p,q)=j^{2}(p,q)$. We construct a family of rotating stars of the
following form
\[%
\begin{cases}
\rho_{\mu,\varepsilon}(r,z)=\rho_{\mu}(g_{\zeta_{\mu,\varepsilon}}%
^{-1}((r,z))),\\
\vec{v}_{\mu,\varepsilon}=\varepsilon\frac{j(m_{\rho_{\mu,\varepsilon}%
}(r),M_{\mu,\varepsilon})}{r}\mathbf{e}_{\theta},
\end{cases}
\]
where
\[
m_{\rho_{\mu,\varepsilon}}(r)=\int_{0}^{r}s\int_{-\infty}^{\infty}\rho
_{\mu,\varepsilon}(s,z)dsdz,\ g_{\zeta_{\mu,\varepsilon}}=x\left(
1+\frac{\zeta_{\mu,\varepsilon}(x)}{|x|^{2}}\right)  ,
\]
and $\zeta_{\mu,\varepsilon}(x):B_{\mu}\rightarrow\mathbb{R}$ is axi-symmetric
and even in $z$.

The existence of rotating stars $\left(  \rho_{\mu,\varepsilon},\vec{v}%
_{\mu,\varepsilon}\right)  \ $is reduced to the following equations:
\begin{equation}
\Phi^{\prime}(\rho_{\mu,\varepsilon})+V_{\mu,\varepsilon}-\varepsilon^{2}%
\int_{0}^{r}J(m_{\rho_{\mu,\varepsilon}}(s),M_{\mu,\varepsilon})s^{-3}%
ds+c_{\mu,\varepsilon}=0,\quad\text{in }\Omega_{\mu,\varepsilon},\label{iftm1}%
\end{equation}%
\begin{equation}
V_{\mu,\varepsilon}=-|x|^{-1}\ast\rho_{\mu,\varepsilon}\text{ in }%
\mathbb{R}^{3},\label{iftm2}%
\end{equation}
where $\Omega_{\mu,\varepsilon}=g_{\zeta_{\mu,\varepsilon}}(B_{\mu})$ and
$c_{\mu,\varepsilon}$ is a constant.


Although \eqref{iftm1} is a little different from the steady state equations
in \cite{H1994} \cite{JJ2019}, the key linearized operator at the point
$\varepsilon=0$ is the same as \cite{H1994}. By similar arguments as
\cite{H1994,JJ2019,SW2017}, we can get the following existence theorem.

\begin{theorem}
\label{IFTlocaldistribution} Let $\mu\in\lbrack\mu_{0},\mu_{1}]\subset
(0,\mu_{\max})$, $P(\rho)$ satisfy \eqref{P1}-\eqref{P2} and $j(p,q)$ satisfy
(\ref{j-smooth}). Then there exist $\tilde{\varepsilon}>0$ and solutions
$\rho_{\mu,\varepsilon}$ of \eqref{iftm1} for all $|\varepsilon|<\tilde
{\varepsilon}$, with the following properties: \newline1) $\rho_{\mu
,\varepsilon}\in C_{c}^{1,\alpha}(\mathbb{R}^{3})$, where $\alpha=\min
(\frac{2-\mathbb{\gamma}_{0}}{\mathbb{\gamma}_{0}-1},1)$. \newline2)
$\rho_{\mu,\varepsilon}$ is axi-symmetric and even in $z$. \newline3)
$\rho_{\mu,\varepsilon}(0)=\mu$. \newline4) $\rho_{\mu,\varepsilon}\geq0$ has
compact support $g_{\zeta_{\mu,\varepsilon}}(B_{\mu})$. \newline5) For all
$\mu\in\lbrack\mu_{0},\mu_{1}]$, the mapping $\varepsilon\rightarrow\rho
_{\mu,\varepsilon}$ is continuous from $(-\tilde{\varepsilon},\tilde
{\varepsilon})$ into $C_{c}^{1}(\mathbb{R}^{3}).$

When $\varepsilon=0$, $\rho_{\mu,0}(x)=\rho_{\mu}(|x|)$ is the nonrotating
star solution with $\rho_{\mu}(0)=\mu$.
\end{theorem}

Now we use Theorem \ref{Th: rayleigh stable} to study the stability of
rotating star solutions $(\rho_{\mu,\varepsilon},\varepsilon j(m_{\rho
_{\mu,\varepsilon}}(r),M_{\mu,\varepsilon})/r\mathbf{e}_{\theta})$, where
$\varepsilon$ is small enough, $j(p,q)$ satisfies (\ref{j-smooth}) and the
Rayleigh stability condition $\partial_{p}J\left(  p,q\right)  >0~$(i.e.
$j\partial_{p}j>0$). As in Section 3.1, the assumptions in Theorem
\ref{Th: rayleigh stable} can be verified. That is, $\partial\Omega
_{\mu,\varepsilon}$ is $C^{2}$ and has positive curvature near $(R_{\mu
,\varepsilon},0)$ and \eqref{rhonearboundary} holds for any $\mu\in\lbrack
\mu_{0},\mu_{1}]$ and $\varepsilon$ small enough.

Below, for rotating stars $(\rho_{\mu,\varepsilon},\varepsilon j(m_{\rho
_{\mu,\varepsilon}}(r),M_{\mu,\varepsilon})/r\mathbf{e}_{\theta})$ we use
$X_{\mu,\varepsilon}$, $X_{1}^{\mu,\varepsilon}$, $Y_{\mu,\varepsilon}$,
$L_{\mu,\varepsilon}$, $A_{1}^{\mu,\varepsilon}$, $B_{1}^{\mu,\varepsilon}$,
$B_{2}^{\mu,\varepsilon}$, $K_{\mu,\varepsilon}$, etc., to denote the
corresponding spaces $X$, $X_{1}$, $Y$, and operators $L$, $A_{1}$, $B_{1}$,
$B_{2}$, $\mathcal{K}$ etc. defined in Section 2. Again, we denote $\tilde
{\mu}$ to be the first critical point of $M_{\mu}/R_{\mu}$ for non-rotating
stars. Define the spaces $X_{ev}^{\mu,\varepsilon}$ and $X_{ev}^{\mu
,\varepsilon}$ as in (\ref{defn-spaces-parity}). By the same proof of Lemma
\ref{N-Dmukappa}, we have the following.

\begin{lemma}
\label{Le-decom-3.2} Assume $P(\rho)$ satisfies\eqref{P1}-\eqref{P2} and
$j(p,q)$ satisfies (\ref{j-smooth}) and $\partial_{p}(j^{2}\left(  p,q\right)
)>0$. Then for any $\mu\in\lbrack\mu_{0},\mu_{1}]\subset(0,\tilde{\mu})$ and
$\varepsilon$ small enough, we have $n^{-}(K_{\mu,\varepsilon})=1$ and $\ker
K_{\mu,\varepsilon}=span\{\partial_{z}\rho_{\mu,\varepsilon}\}$. Moreover, we
have the following direct sum decompositions for $X_{ev}^{\mu,\varepsilon}$
and $X_{ev}^{\mu,\varepsilon}:$
\[
X_{ev}^{\mu,\varepsilon}=X_{-,ev}^{\mu,\varepsilon}\oplus X_{+,ev}%
^{\mu,\varepsilon},\ \ \dim X_{-,ev}^{\mu,\varepsilon}=1,
\]
and
\[
X_{od}^{\mu,\varepsilon}=span\{\partial_{z}\rho_{\mu,\varepsilon}\}\oplus
X_{+,od}^{\mu,\varepsilon},
\]
satisfying: i) $K_{\mu,\varepsilon}|_{X_{-,ev}^{\mu,\varepsilon}}<0;$

ii) there exists $\delta>0$ such that
\[
\left\langle K_{\mu,\varepsilon}u,u\right\rangle \geq\delta\left\Vert
u\right\Vert _{L_{\Phi^{\prime\prime}(\rho_{\mu,\varepsilon})}^{2}}%
^{2}\ ,\text{ }\forall\text{ }u\in X_{+,ev}^{\mu,\varepsilon}\oplus
X_{+,od}^{\mu,\varepsilon},
\]
where $\delta$ is independent of $\mu$ and $\varepsilon$.

In addition, for any $\mu\in\lbrack\mu_{0},\mu_{1}]$, it holds that
$\frac{dV_{\mu,\varepsilon}(R_{\mu,\varepsilon},0)}{d\mu}<0$ for $\varepsilon$ small.
\end{lemma}

By Theorem \ref{Th: rayleigh stable}, we get the following necessary and
sufficient condition for the stability of rotating stars $(\rho_{\mu
,\varepsilon},\varepsilon j(m_{\rho_{\mu,\varepsilon}}(r),M_{\mu,\varepsilon
})/r\mathbf{e}_{\theta}):$
\begin{align*}
&  \langle K_{\mu,\varepsilon}\delta\rho,\delta\rho\rangle=\langle
L_{\mu,\varepsilon}\delta\rho,\delta\rho\rangle\\
&  \quad+2\varepsilon^{2}\pi\int_{0}^{R_{\mu,\varepsilon}}\frac{\partial
_{p}J(m_{\rho_{\mu,\varepsilon}}(r),M_{\mu,\varepsilon})}{r^{3}}\left(
\int_{0}^{r}s\int_{-\infty}^{+\infty}\delta\rho(s,z)dzds\right)  ^{2}dr\geq0,
\end{align*}
for all $\delta\rho\in R(B_{1}^{\mu,\varepsilon})=\left\{  \delta\rho\in
X_{1}^{\mu,\varepsilon}|\int_{\mathbb{R}^{3}}\delta\rho dx=0\right\}  $.

The following Theorem shows that the stability of this family of rotating
stars can only change at the mass extrema.

\begin{theorem}
\label{TTPforfixdistribution} Assume $P(\rho)$ satisfies
\eqref{P1}-\eqref{P2}, and $j(p,q)$ satisfy (\ref{j-smooth}) and $\partial
_{p}(j^{2}\left(  p,q\right)  )>0$. Let $n^{u}(\mu)$ be the number of unstable
modes, namely the total algebraic multiplicities of unstable eigenvalues of
the linearized Euler-Poisson systems at $(\rho_{\mu,\varepsilon},\varepsilon
j(m_{\rho_{\mu,\varepsilon}}(r),M_{\mu,\varepsilon})/r\mathbf{e}_{\theta})$.
Then for any $\mu\in\lbrack\mu_{0},\mu_{1}]\subset(0,\tilde{\mu})$ and
$\varepsilon$ small enough, we have
\[
n^{u}(\mu)=%
\begin{cases}
1\text{, when }\frac{dM_{\mu,\varepsilon}}{d\mu}<0,\\
0\text{, when }\frac{dM_{\mu,\varepsilon}}{d\mu}\geq0.
\end{cases}
\]

\end{theorem}

\begin{proof}
By the same arguments in the proof of Theorem
\ref{Th: stability-fixed velocity}, we have
\[
n^{u}(\mu)=n^{-}\left(  K_{\mu,\varepsilon}|_{X_{ev}^{\mu,\varepsilon}\cap
R(B_{1}^{\mu,\varepsilon})}\right)  .
\]
Thus it is reduced to find the number of negative modes of the quadratic form
$\left\langle K_{\mu,\varepsilon}\cdot,\cdot\right\rangle $ restricted to the
even subspace of $R(B_{1}^{\mu,\varepsilon})$.

Applying $\frac{d}{d\mu}$ to (\ref{iftm1}), we obtain that
\begin{align}
&  L_{\mu,\varepsilon}\frac{d\rho_{\mu,\varepsilon}}{d\mu}=\varepsilon^{2}%
\int_{0}^{r}\partial_{p}J(m_{\rho_{\mu,\varepsilon}}(s),M_{\mu,\varepsilon
})\frac{dm_{\rho_{\mu,\varepsilon}}}{d\mu}s^{-3}ds\label{L-J}\\
&  \ \ \ \ \ \ \ \ \ \ \ \ \ \ \ \ \ \ \ \ \ +\varepsilon^{2}\int_{0}%
^{r}\partial_{q}J(m_{\rho_{\mu,\varepsilon}}(s),M_{\mu,\varepsilon}%
)\frac{dM_{\mu,\varepsilon}}{d\mu}s^{-3}ds-\frac{dc_{\mu,\varepsilon}}{d\mu
},\nonumber
\end{align}
where
\begin{align*}
\frac{dc_{\mu,\varepsilon}}{d\mu} &  =\frac{d}{d\mu}\left(  -V_{\mu
,\varepsilon}(R_{\mu,\varepsilon},0)+\varepsilon^{2}\int_{0}^{R_{\mu
,\varepsilon}}J(m_{\rho_{\mu,\varepsilon}}(s),M_{\mu,\varepsilon}%
)s^{-3}ds\right)  \\
&  =-\frac{dV_{\mu,\varepsilon}(R_{\mu,\varepsilon},0)}{d\mu}+\varepsilon
^{2}\int_{0}^{R_{\mu,\varepsilon}}\partial_{p}J(m_{\rho_{\mu,\varepsilon}%
}(s),M_{\mu,\varepsilon})\frac{dm_{\rho_{\mu,\varepsilon}}(s)}{d\mu}s^{-3}ds\\
&  \quad+\varepsilon^{2}\frac{dM_{\mu,\varepsilon}}{d\mu}h_{\mu,\varepsilon
}(R_{\mu,\varepsilon})+\varepsilon^{2}J(M_{\mu,\varepsilon},M_{\mu
,\varepsilon})R_{\mu,\varepsilon}^{-3}\frac{dR_{\mu,\varepsilon}}{d\mu}.
\end{align*}
By integration by parts and \eqref{L-J}, we obtain that
\begin{align*}
&  2\pi\int_{0}^{R_{\mu,\varepsilon}}\varepsilon^{2}\left[  \partial
_{p}J(m_{\rho_{\mu,\varepsilon}}(r),M_{\mu,\varepsilon})r^{-3}\right]  \left(
\int_{0}^{r}\int_{-\infty}^{\infty}s\frac{d\rho_{\mu,\varepsilon}}{d\mu
}dzds\right)  \left(  \int_{0}^{r}\int_{-\infty}^{\infty}s\varphi dzds\right)
dr\\
&  =\varepsilon^{2}\left[  \int_{0}^{R_{\mu,\varepsilon}}\partial_{p}%
J(m_{\rho_{\mu,\varepsilon}}(r),M_{\mu,\varepsilon})r^{-3}\frac{dm_{\rho
_{\mu,\varepsilon}}}{d\mu}dr\right]  \int_{\mathbb{R}^{3}}\varphi dx\\
&  \quad-2\pi\int_{0}^{R_{\mu,\varepsilon}}\int_{-\infty}^{\infty}%
\varepsilon^{2}\left[  \int_{0}^{r}\partial_{p}J(m_{\rho_{\mu,\varepsilon}%
}(s),M_{\mu,\varepsilon})\frac{dm_{\rho_{\mu,\varepsilon}(s)}}{d\mu}%
s^{-3}ds\right]  r\varphi dzdr\\
&  =\varepsilon^{2}\left[  \int_{0}^{R_{\mu,\varepsilon}}\partial_{p}%
J(m_{\rho_{\mu,\varepsilon}}(r),M_{\mu,\varepsilon})r^{-3}\frac{dm_{\rho
_{\mu,\varepsilon}}}{d\mu}dr\right]  \int_{\mathbb{R}^{3}}\varphi
dx-\left\langle \frac{dc_{0}}{d\mu},\varphi\right\rangle \\
&  \quad-\left\langle L_{\mu,\varepsilon}\frac{d\rho_{\mu,\varepsilon}}{d\mu
},\varphi\right\rangle -\varepsilon^{2}\frac{dM_{\mu,\varepsilon}}{d\mu
}\left\langle \int_{0}^{r}\partial_{q}J(m_{\rho_{\mu,\varepsilon}}%
(s),M_{\mu,\varepsilon})s^{-3}ds,\varphi\right\rangle \\
&  =\left(  \frac{dV_{\mu,\varepsilon}(R_{\mu,\varepsilon},0)}{d\mu
}-\varepsilon^{2}J(M_{\mu,\varepsilon},M_{\mu,\varepsilon})R_{\mu,\varepsilon
}^{-3}\frac{dR_{\mu,\varepsilon}}{d\mu}-\varepsilon^{2}\frac{dM_{\mu
,\varepsilon}}{d\mu}h_{\mu,\varepsilon}(R_{\mu,\varepsilon})\right)
\int_{\mathbb{R}^{3}}\varphi dx\\
&  \quad-\left\langle L_{\mu,\varepsilon}\frac{d\rho_{\mu,\varepsilon}}{d\mu
},\varphi\right\rangle -\varepsilon^{2}\frac{dM_{\mu,\varepsilon}}{d\mu
}\left\langle K_{\mu,\varepsilon}g_{\mu,\varepsilon},\varphi\right\rangle .
\end{align*}
Here, in the above we used
\[
h_{\mu,\varepsilon}(r)=\int_{0}^{r}\partial_{q}J\left(  m_{\rho_{\mu
,\varepsilon}}(s),M_{\mu,\varepsilon}\right)  s^{-3}ds,
\]
and $g_{\mu,\varepsilon}=K_{\mu,\varepsilon}^{-1}h_{\mu,\varepsilon}$. The
inverse operator
\[
K_{\mu,\varepsilon}^{-1}:\left(  X_{ev}^{\mu,\varepsilon}\right)  ^{\ast
}\subset L_{\frac{1}{\Phi^{\prime\prime}(\rho_{\mu,\varepsilon})}}%
^{2}\rightarrow X_{ev}^{\mu,\varepsilon}%
\]
exists and is bounded by Lemma \ref{Le-decom-3.2}. Since $\frac{1}%
{\Phi^{\prime\prime}(\rho_{\mu,\varepsilon})}$ has compact support and
$\Phi^{\prime\prime}\left(  s\right)  \thickapprox s^{\gamma_{0}-2}\ $for
$s\sim0^{+}$, we have
\[
\left\vert \int g_{\mu,\varepsilon}dx\right\vert \lesssim\Vert g_{\mu
,\varepsilon}\Vert_{L_{\Phi^{\prime\prime}(\rho_{\mu,\varepsilon})}^{2}%
}\lesssim\left\Vert K_{\mu,\varepsilon}^{-1}\right\Vert \Vert h_{\mu
,\varepsilon}\Vert_{L_{\frac{1}{\Phi^{\prime\prime}(\rho_{\mu,\varepsilon})}%
}^{2}}\lesssim\left(  \int\frac{h_{\mu,\varepsilon}^{2}}{\Phi^{\prime\prime
}(\rho_{\mu,\varepsilon})}dx\right)  ^{\frac{1}{2}}<+\infty.
\]

Therefore, we have
\begin{equation}
\left\langle K_{\mu,\varepsilon}\left[  \frac{d\rho_{\mu,\varepsilon}}{d\mu
}+\varepsilon^{2}\frac{dM_{\mu,\varepsilon}}{d\mu}g_{\mu,\varepsilon}\right]
,\varphi\right\rangle =\left(  \frac{dV_{\mu,\varepsilon}(R_{\mu,\varepsilon
},0)}{d\mu}+O(\varepsilon^{2})\right)  \int_{\mathbb{R}^{3}}\varphi
dx,\label{K-inner-product}%
\end{equation}
for any $\varphi\in X_{ev}^{\mu,\varepsilon}$.

By (\ref{K-inner-product}) and the fact that $\frac{dV_{\mu,\varepsilon
}(R_{\mu,\varepsilon},0)}{d\mu}+O(\varepsilon^{2})<0$ when $\mu\in\lbrack
\mu_{0},\mu_{1}]$ and $\varepsilon$ is small, we have
\[
X_{ev}^{\mu,\varepsilon}\cap R(B_{1}^{\mu,\varepsilon})=\left\{  \delta\rho\in
X_{ev}^{\mu,\varepsilon}|\left\langle K_{\mu,\varepsilon}\left(  \frac
{d\rho_{\mu,\varepsilon}}{d\mu}+\varepsilon^{2}\frac{dM_{\mu,\varepsilon}%
}{d\mu}g_{\mu,\varepsilon}\right)  ,\delta\rho\right\rangle =0\right\}  .
\]
On the other hand, we have
\begin{align*}
&  \left\langle K_{\mu,\varepsilon}\left(  \frac{d\rho_{\mu,\varepsilon}}%
{d\mu}+\varepsilon^{2}\frac{dM_{\mu,\varepsilon}}{d\mu}g_{\mu,\varepsilon
}\right)  ,\left(  \frac{d\rho_{\mu,\varepsilon}}{d\mu}+\varepsilon^{2}%
\frac{dM_{\mu,\varepsilon}}{d\mu}g_{\mu,\varepsilon}\right)  \right\rangle \\
&  =\left(  \frac{dV_{\mu,\varepsilon}(R_{\mu,\varepsilon},0)}{d\mu
}+O(\varepsilon^{2})\right)  \int\left(  \frac{d\rho_{\mu,\varepsilon}}{d\mu
}+\varepsilon^{2}\frac{dM_{\mu,\varepsilon}}{d\mu}g_{\mu,\varepsilon}\right)
dx\\
&  =\left(  \frac{dV_{\mu,\varepsilon}(R_{\mu,\varepsilon},0)}{d\mu
}+O(\varepsilon^{2})\right)  \frac{dM_{\mu,\varepsilon}}{d\mu}.
\end{align*}
By Lemma \ref{Le-decom-3.2}, $n^{-}(K_{\mu,\varepsilon}|_{X_{ev}%
^{\mu,\varepsilon}})=1\ $and $\ker$ $K_{\mu,\varepsilon}|_{X_{ev}%
^{\mu,\varepsilon}}=\left\{  0\right\}  $. We consider two cases:

1) $\frac{dM_{\mu,\varepsilon}}{d\mu}\neq0$. A combination of above properties
immediately yields
\[
n^{u}(\mu)=n^{-}\left(  K_{\mu,\varepsilon}|_{X_{ev}^{\mu,\varepsilon}\cap
R(B_{1}^{\mu,\varepsilon})}\right)  =%
\begin{cases}
1\text{ when }\frac{dM_{\mu,\varepsilon}}{d_{\mu}}<0,\\
0\text{ when }\frac{dM_{\mu,\varepsilon}}{d_{\mu}}>0.
\end{cases}
\]

2) When $\frac{dM_{\mu,\varepsilon}}{d\mu}=0$, as in the proof of Theorem
\ref{Th: stability-fixed velocity}, we have
\[
n^{u}(\mu)=n^{-}\left(  K_{\mu,\varepsilon}|_{X_{ev}^{\mu,\varepsilon}\cap
R(B_{1}^{\mu,\varepsilon})}\right)  =0.
\]
This finishes the proof of the theorem.
\end{proof}

\begin{remark}
The above theorem implies that for a family of rotating stars with fixed
angular momentum distribution $j(m,M)$, the transition of stability occurs at
the first extrema of the total mass. That is, the turning point principle
(TPP) is true for this family of rotating stars. This contrasts greatly to
rotating stars of fixed angular velocity, for which case TPP is shown to be
not true (see Theorem \ref{Th: TPP-fixed velocity}).

In the literature, there are three common choices of $j(m,M)$ in the study of
rotating stars.

i) (Fixed angular momentum distribution) The most common one is $j(m,M)=j(m)$.
See for example \cite{AGBR1971,JJMT2019,LS2004,LTSJ2008,LS2009,LTSJ2011};

ii) (Fixed angular momentum distribution per unit mass) $j(m,M)=j(m/M)$. See
for example \cite{OM1968};

iii) (Fixed angular momentum distribution with given total angular momentum)
$j(m,M)=\frac{1}{M}j(m/M)$. See for example \cite{BB1974}. We note that for
this case, the total angular momentum given by
\[
\int\frac{1}{M}j(\frac{m}{M})dm=\int_{0}^{1}j\left(  m^{\prime}\right)
dm^{\prime}\ \ (m^{\prime}=\frac{m}{M}),
\]
is a constant depending only on $j$.
\end{remark}

In the rest of this subsection, we use Theorem \ref{TTPforfixdistribution} to
study two examples of rotating stars with mass extrema points.

\textbf{Example 1. }Asymptotically polytropic rotating stars

Assume $P(\rho)$ satisfies assumptions (\ref{approxiP1})-(\ref{approxiP2}). By
the same arguments as in the case of fixed angular velocity, when
$\varepsilon$ is small enough and $\mu\in\lbrack\mu_{0},\mu_{1}]\subset
(0,\tilde{\mu})$, the mass $M_{\mu,\varepsilon}$ of the rotating stars
$(\rho_{\mu,\varepsilon},\varepsilon j(m_{\rho_{\mu,\varepsilon}}%
(r),M_{\mu,\varepsilon})/r\mathbf{e}_{\theta})$\ has the the first maximum
$\mu_{\ast}^{\varepsilon}\in\left(  \mu_{0},\mu_{1}\right)  $. Then by Theorem
\ref{TTPforfixdistribution}, the rotating stars are stable when $\mu\in
\lbrack\mu_{0},\mu_{\ast}^{\varepsilon}]$ and unstable when $\mu$ goes between
$\mu_{\ast}^{\varepsilon}$ and the next extrema point of $M_{\mu,\varepsilon}$
in $\left(  \mu_{\ast}^{\varepsilon},\mu_{1}\right)  $.

\textbf{Example 2. }Polytropic rotating stars

Consider the polytropic equation of state $P(\rho)=\rho^{\gamma}~\left(
\gamma\in\left(  \frac{6}{5},2\right)  \right)  $. The non-rotating stars
(i.e. Lane-Emden stars) with any center density $\mu$ are stable when
$\gamma\in(4/3,2)$ and are unstable when $\gamma\in(6/5,4/3)$. In particular,
$M_{\mu}=C_{\gamma}\mu^{\frac{1}{2}\left(  3\gamma-4\right)  }$ is a monotone
function when $\gamma\neq\frac{4}{3}\ $and there is no transition point of stability.

However, polytropic rotating stars with fixed angular momentum distribution
$j\left(  m,M\right)  \ $can have mass extrema points, which are also the
transition points of stability. One such example was given in \cite{BB1974}
for $\gamma=\frac{4.03}{3.03}<\frac{4}{3}$ and $j(m,M)=\frac{1}{M}%
[1-(1-\frac{m}{M})^{2/3}]$. With numerical help, it was found (see Figure 1
below taken from \cite{BB1974}) that there is a mass minimum point $\mu^{\ast
}\ $for the total mass $M\left(  \mu\right)  $. This is the first transition
point of stability. In particular, rotating stars with center density $\mu$
beyond $\mu^{\ast}$ become stable. \begin{figure}[h]
\centering\includegraphics[scale=0.3]{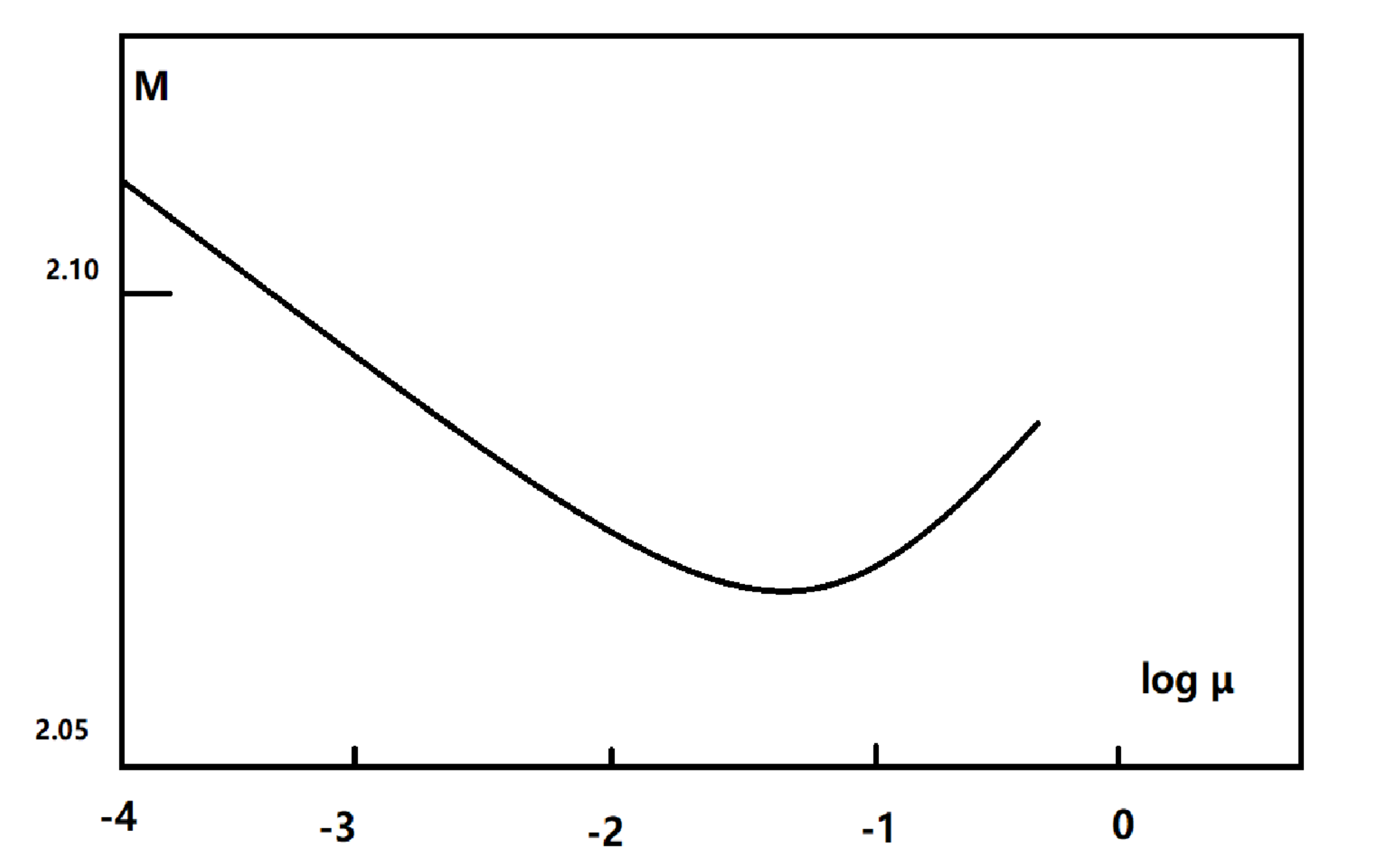}\caption{The dependence of
the mass $M\left(  \mu\right)  $ on the center density $\mu$ for $\gamma
=\frac{4.03}{3.03}$ and the angular momentum distribution $j(m,M)=\frac{1}%
{M}[1-(1-\frac{m}{M})^{2/3}]$. From Bisnovatyi-Kogan and Blinnikov
\cite{BB1974}.}%
\label{fig:3}%
\end{figure}

\begin{remark}
It can also be seen from above Example 2 that the critical index $\gamma
^{\ast}$ for the onset of instability of rotating polytropic stars is not
$\frac{4}{3}$. Ledoux \cite{Ledoux1945}, Chandrasekhar and Lebovitz
\cite{ChandraLebovitz68} indicated that the critical index $\gamma^{\ast}$ is
reduced from $\frac{4}{3}$ to $\gamma^{\ast}=\frac{4}{3}-\frac{2\omega^{2}%
I}{9|W|}$ for small uniform rotating stars, where $I>0$ is the moment of
inertia about the center of mass and $W$ is the gravitational potential
energy. For more discussion about the critical index $\gamma^{\ast}$ of
rotating stars, see \cite{Hazelhurst1994,Kahler1994,Sidorov1982,Sidorov1981}.
\end{remark}

\section{\label{Rayleigh-instability}Instability for Rayleigh Unstable case}

Consider an axi-symmetric rotating star $(\rho_{0},\vec{v}_{0})=\left(
\rho_{0}\left(  r,z\right)  ,\omega_{0}(r)r\mathbf{e}_{\theta}\right)  $,
where the angular velocity $\omega_{0}(r)$ satisfies the Rayleigh instability
condition, that is, there exists a point $r_{0}\in(0,R_{0})$ such that
\begin{equation}
\Upsilon(r_{0})=\frac{\partial_{r}(\omega_{0}^{2}r^{4})}{r^{3}}\bigg|_{r=r_{0}%
}<0. \label{condition-rayleigh}%
\end{equation}
For incompressible Euler equation, it is a classical result by Rayleigh in
1880 \cite{Rayleigh1880} that condition (\ref{condition-rayleigh}) implies
linear instability of the rotating flow $\vec{v}_{0}=\omega_{0}(r)r\mathbf{e}%
_{\theta}$ under axi-symmetric perturbations. In this section, we will show
the axi-symmetric instability of rotating stars with Rayleigh unstable angular velocity.

From the linearized Euler-Poisson system \eqref{linearized-EP}, we get the
following second order equation for $u_{2}=%
\begin{pmatrix}
v_{r}\\
v_{z}%
\end{pmatrix}
$,%
\begin{equation}
\partial_{tt}u_{2}=-\mathbb{\tilde{L}}u_{2}=-(\mathbb{L}_{1}+\mathbb{L}%
_{2})u_{2}, \label{eqn-2nd-order}%
\end{equation}
where $\mathbb{L}_{1},\ \mathbb{L}_{2}\,\ $are operators on $Y=\left(
L_{\rho_{0}}^{2}\right)  ^{2}$ defined by
\[
\mathbb{L}_{1}u_{2}=B_{1}^{\prime}LB_{1}A=\nabla\lbrack\Phi^{\prime\prime
}(\rho_{0})(\nabla\cdot(\rho_{0}u_{2}))-4\pi(-\Delta)^{-1}(\nabla\cdot
(\rho_{0}u_{2})],
\]
and
\[
\mathbb{L}_{2}u_{2}=%
\begin{pmatrix}
\Upsilon(r)v_{r}\\
0
\end{pmatrix}
.
\]

\begin{lemma}
$\mathbb{\tilde{L}}$ is a self-adjoint operator on $(Y,[\cdot,\cdot])$ with
the equivalent inner product $[\cdot,\cdot]=\langle A\cdot,\cdot\rangle$.
\end{lemma}

\begin{proof}
By Lemma 2.9 in \cite{LZ2019}, $\mathbb{L}_{1}$ is self-adjoint on
$(Y,[\cdot,\cdot])$ with the equivalent inner product $[\cdot,\cdot]:=\langle
A\cdot,\cdot\rangle$. Since $\mathbb{L}_{2}$ is a symmetric bounded operator
on $(Y,[\cdot,\cdot]),$ $\mathbb{\tilde{L}}=\mathbb{L}_{1}+\mathbb{L}_{2}$ is
self-adjoint by Kato-Rellich Theorem.
\end{proof}

The next lemma on the quadratic form of $\mathbb{\tilde{L}\ }$will be used later.

\begin{lemma}
\label{lemma-lower-bound-quadratic}There exists constants $m>0$ such that for
any $u_{2}\in Y$, we have
\[
\left[  \mathbb{\tilde{L}}u_{2},u_{2}\right]  +m\left\Vert u_{2}\right\Vert
_{Y}^{2}\geq\left\Vert \nabla\cdot(\rho_{0}u_{2})\right\Vert _{L_{\Phi
^{\prime\prime}(\rho_{0})}^{2}}^{2}.
\]

\end{lemma}

\begin{proof}
Since%
\[
\left[  \mathbb{\tilde{L}}u_{2},u_{2}\right]  =\left[  \mathbb{L}_{1}%
u_{2},u_{2}\right]  +\left[  \mathbb{L}_{2}u_{2},u_{2}\right]  ,
\]
and obviously $\left\vert \left[  \mathbb{L}_{2}u_{2},u_{2}\right]
\right\vert \lesssim\left\Vert u_{2}\right\Vert _{L_{\rho_{0}}^{2}}^{2}$, it
suffices to estimate
\[
\left[  \mathbb{L}_{1}u_{2},u_{2}\right]  =\left\langle LB_{1}Au_{2}%
,B_{1}Au_{2}\right\rangle =\left\Vert \nabla\cdot(\rho_{0}u_{2})\right\Vert
_{L_{\Phi^{\prime\prime}(\rho_{0})}^{2}}^{2}-4\pi\int_{\mathbb{R}^{3}%
}\left\vert \nabla V\right\vert ^{2}dx,
\]
where $-\Delta V=\nabla\cdot(\rho_{0}u_{2})$. By integration by parts,
\[
\int_{\mathbb{R}^{3}}\left\vert \nabla V\right\vert ^{2}dx=-\int
_{\mathbb{R}^{3}}\rho_{0}u_{2}\cdot\nabla Vdx\lesssim\left(  \left\Vert
u_{2}\right\Vert _{Y}^{2}\right)  ^{\frac{1}{2}}\left(  \int\left\vert \nabla
V\right\vert ^{2}dx\right)  ^{\frac{1}{2}},
\]
which implies that $\int\left\vert \nabla V\right\vert ^{2}dx\lesssim
\left\Vert u_{2}\right\Vert _{Y}^{2}$. This finishes the proof of the lemma.
\end{proof}

The study of equation (\ref{eqn-2nd-order}) is reduced to understand the
spectra of the self-adjoint operator $\mathbb{\tilde{L}}$. First, we give a
Helmholtz type decomposition of vector fields in $Y$.

\begin{lemma}
\label{lemma-helom-decomp}There is a direct sum decomposition $Y=Y_{1}\oplus
Y_{2}$, where $Y_{1}$ is the closure of
\[
\left\{  u\in Y\ |\ u=\nabla p,\text{ for some }p\in C^{1}\left(
\Omega\right)  \right\}  ,
\]
in $Y$ and $Y_{2}$ is the closure of
\[
\left\{  u\in\left(  C^{1}\left(  \Omega\right)  \right)  ^{2}\cap
Y\ |\ \nabla\cdot\left(  \rho_{0}u\right)  =0\text{ }\right\}  ,
\]
in $Y$.
\end{lemma}

The proof of above lemma is similar to that of Lemma 3.15 in \cite{LZ2019} and
we skip. Denote $\mathbb{P}_{1}:Y\mapsto Y_{1}$ and $\mathbb{P}_{2}:Y\mapsto
Y_{2}$ to be the projection operators. Then $\left\Vert \mathbb{P}%
_{1}\right\Vert ,\left\Vert \mathbb{P}_{2}\right\Vert \leq1$.

For any $u_{2}\in Y$, let $u_{2}=v_{1}+v_{2}$ where $v_{1}=\mathbb{P}_{1}%
u_{2}\in Y_{1}$ and $v_{2}=\mathbb{P}_{2}u_{2}\in Y_{2}$. Since
\[
\mathbb{\tilde{L}}u_{2}=\mathbb{L}_{1}v_{1}+\mathbb{P}_{1}\mathbb{L}_{2}%
v_{1}+\mathbb{P}_{1}\mathbb{L}_{2}v_{2}+\mathbb{P}_{2}\mathbb{L}_{2}%
v_{1}+\mathbb{P}_{2}\mathbb{L}_{2}v_{2},
\]
the operator $\mathbb{\tilde{L}}:Y\rightarrow Y$ is equivalent to the
following matrix operator on $Y_{1}\times Y_{2}$
\begin{align*}
&
\begin{pmatrix}
\tilde{\mathbb{L}}_{1}, & \mathbb{C}\\
\mathbb{C}^{\ast}, & \tilde{\mathbb{L}}_{2}%
\end{pmatrix}%
\begin{pmatrix}
v_{1}\\
v_{2}%
\end{pmatrix}
\\
&  =\left[
\begin{pmatrix}
\tilde{\mathbb{L}}_{1}, & \mathbb{C}\\
0, & \tilde{\mathbb{L}}_{2}%
\end{pmatrix}
+%
\begin{pmatrix}
0, & 0\\
\mathbb{C}^{\ast}, & 0
\end{pmatrix}
\right]
\begin{pmatrix}
v_{1}\\
v_{2}%
\end{pmatrix}
\\
&  =(T+\mathbb{A})v,
\end{align*}
where
\[
\tilde{\mathbb{L}}_{1}=\mathbb{L}_{1}+\mathbb{P}_{1}\mathbb{L}_{2}%
\mathbb{P}_{1}:Y_{1}\rightarrow Y_{1},\tilde{\ \mathbb{L}}_{2}=\mathbb{P}%
_{2}\mathbb{L}_{2}\mathbb{P}_{2}:Y_{2}\rightarrow Y_{2},
\]%
\[
\mathbb{C}=\mathbb{P}_{1}\mathbb{L}_{2}\mathbb{P}_{2}:Y_{2}\rightarrow
Y_{1},\ \mathbb{C}^{\ast}=\mathbb{P}_{2}\mathbb{L}_{2}\mathbb{P}_{1}%
:Y_{1}\rightarrow Y_{2},
\]
and
\[
T=%
\begin{pmatrix}
\tilde{\mathbb{L}}_{1}, & \mathbb{C}\\
0, & \tilde{\mathbb{L}}_{2}%
\end{pmatrix}
,\ \mathbb{A}=%
\begin{pmatrix}
0, & 0\\
\mathbb{C}^{\ast}, & 0
\end{pmatrix}
:\ Y_{1}\times Y_{2}\rightarrow Y_{1}\times Y_{2}.\
\]

\begin{lemma}
\label{Tcompact} The operator $\mathbb{A}$ is $T$-compact.
\end{lemma}

\begin{proof}
For any $v=\left(  v_{1},v_{2}\right)  \in D\left(  T\right)  $, the graph
norm $\left\Vert v\right\Vert _{T}\ $is defined by
\begin{align*}
\left\Vert v\right\Vert _{T}  &  =\Vert v\Vert_{Y}+\Vert Tv\Vert_{Y}\\
&  \approx\Vert v\Vert_{Y}+\Vert\tilde{\mathbb{L}}_{1}v_{1}\Vert_{Y}%
\approx\Vert v\Vert_{Y}+\Vert\mathbb{L}_{1}v_{1}\Vert_{Y}.
\end{align*}
It is obvious that $D(\mathbb{A})\supset D(T)$. To prove $\mathbb{A}$ is
$T$-compact, we need to prove $\mathbb{A}:(D(\mathbb{A}),\left\Vert
\cdot\right\Vert _{T})\mapsto(Y,\Vert\cdot\Vert_{Y})$ is compact. By the
definition of $\mathbb{A}$, we notice that $\mathbb{A}v=\mathbb{P}%
_{2}\mathbb{L}_{2}v_{1}:Y_{1}\times Y_{2}\mapsto\left\{  0\right\}  \times
Y_{2}$. For $v_{1}=\nabla\xi\in Y_{1}$,
\[
\left\Vert v_{1}\right\Vert _{Z}=\Vert\nabla\cdot(\rho_{0}v_{1})\Vert
_{L_{\Phi^{\prime\prime}(\rho_{0})}^{2}}+\Vert v_{1}\Vert_{Y}=\Vert\nabla
\cdot(\rho_{0}\nabla\xi)\Vert_{L_{\Phi^{\prime\prime}(\rho_{0})}^{2}}%
+\Vert\nabla\xi\Vert_{Y},
\]
as defined in (\ref{norm-Z}). By the proof of Lemma
\ref{lemma-lower-bound-quadratic} we have%
\begin{align*}
\Vert\nabla\cdot(\rho_{0}v_{1})\Vert_{L_{\Phi^{\prime\prime}(\rho_{0})}^{2}%
}^{2}+\Vert v_{1}\Vert_{Y}^{2}  &  \lesssim\langle\mathbb{L}_{1}v_{1}%
,v_{1}\rangle+2m\Vert v_{1}\Vert_{Y}^{2}\\
&  \lesssim\Vert\mathbb{L}_{1}v_{1}\Vert_{Y}^{2}+\Vert v_{1}\Vert_{Y}%
^{2}\approx\left\Vert v\right\Vert _{T}^{2}.
\end{align*}
Thus $\left\Vert v_{1}\right\Vert _{Z}\lesssim\left\Vert v\right\Vert _{T}$.
Since the embedding $(Y_{1},\left\Vert \cdot\right\Vert _{Z})\hookrightarrow
(Y_{1},\Vert\cdot\Vert_{Y})$ is compact by Proposition 12 in \cite{JM2020} and
$\mathbb{P}_{2},\ \mathbb{L}_{2}$ are bounded operators, it follows that
$\mathbb{A}:(D(\mathbb{A}),\left\Vert \cdot\right\Vert _{T})\mapsto
(Y,\Vert\cdot\Vert_{Y})$ is compact.
\end{proof}

The above lemma implies that the essential spectra of $\mathbb{\tilde{L}}$ is
the same as $\tilde{\mathbb{L}}_{2}$.

\begin{lemma}
\label{lemma-equ-essential}$\sigma_{ess}(\mathbb{\tilde{L}})=\sigma
_{ess}(\tilde{\mathbb{L}}_{2})$.
\end{lemma}

\begin{proof}
We have $\sigma_{ess}(\mathbb{\tilde{L}})=\sigma_{ess}(T+\mathbb{A})$ by the
definition of the operator $T+\mathbb{A}$. By Lemma \ref{Tcompact} and Weyl's
Theorem, we have $\sigma_{ess}(T+\mathbb{A})=\sigma_{ess}(T).$ By Theorem 2.3
v) in \cite{LZ2019} and the compact embedding of $(Y_{1},\left\Vert
\cdot\right\Vert _{Z})\hookrightarrow(Y_{1},\Vert\cdot\Vert_{Y})$, the spectra
of $\mathbb{L}_{1}$ on $Y_{1}$ are purely discrete and $\sigma_{ess}\left(
\mathbb{L}_{1}\right)  =\left\{  \emptyset\right\}  $. By the same arguments
as in the proof of Lemma \ref{Tcompact}, $\tilde{\mathbb{L}}_{1}$ is relative
compact to $\mathbb{L}_{1}$ and as a result $\sigma_{ess}\left(
\tilde{\mathbb{L}}_{1}\right)  =\sigma_{ess}\left(  \mathbb{L}_{1}\right)
=\left\{  \emptyset\right\}  .$ Since the matrix operator $T$ is upper
triangular, it follows that
\[
\sigma_{ess}(T)=\sigma_{ess}\left(  \tilde{\mathbb{L}}_{1}\right)  \cup
\sigma_{ess}\left(  \tilde{\mathbb{L}}_{2}\right)  =\sigma_{ess}\left(
\tilde{\mathbb{L}}_{2}\right)  .
\]

\end{proof}

We study the essential spectra of $\tilde{\mathbb{L}}_{2}$ in the next two
lemmas. By the Rayleigh instability condition (\ref{condition-rayleigh}) and
the fact that $\Upsilon(0)=4\omega_{0}(0)^{2}\geq0$, we know that
$\text{range}\left(  \Upsilon(r)\right)  =[-a,b]$ for some $a>0$, $b\geq0$.

\begin{lemma}
\label{rangofconsp} $\sigma_{ess}(\tilde{\mathbb{L}}_{2})\supset
\text{range}(\Upsilon(r))=[-a,b]$.
\end{lemma}

\begin{proof}
For any $\lambda\in\left(  -a,b\right)  $, let $r_{0}\in(0,R_{0})$ be such
that $\lambda=\Upsilon(r_{0})$. Choose $(r_{0},z_{0})\in\Omega$ and
$\varepsilon_{0}$ small enough, such that $\left(  r,z\right)  \in\Omega$ when
$\left\vert r-r_{0}\right\vert \leq\varepsilon_{0}$ and $\left\vert
z-z_{0}\right\vert \leq\varepsilon_{0}^{2}$. Choose a sequence $\left\{
\varepsilon_{n}\right\}  _{n=1}^{\infty}\subset\left(  0,\varepsilon
_{0}\right)  $ with $\lim_{n\rightarrow\infty}\varepsilon_{n}=0$. Let
$\varphi(r)$, $\psi(z)\in C_{0}^{\infty}(-1,1)$ be two smooth cutoff functions
such that $\varphi(0)=\psi(0)=1$. Define $\delta v^{\varepsilon_{n}}=(\delta
v_{r}^{\varepsilon_{n}},\delta v_{z}^{\varepsilon_{n}})$ with
\[
\delta v_{z}^{\varepsilon_{n}}=-\frac{\varepsilon_{n}}{A_{\varepsilon_{n}}%
\rho_{0}r}\varphi^{\prime}(\frac{r-r_{0}}{\varepsilon_{n}})\psi(\frac{z-z_{0}%
}{\varepsilon_{n}^{2}}),
\]
and
\[
\delta v_{r}^{\varepsilon_{n}}=\frac{1}{A_{\varepsilon_{n}}\rho_{0}r}%
\varphi(\frac{r-r_{0}}{\varepsilon_{n}})\psi^{\prime}(\frac{z-z_{0}%
}{\varepsilon_{n}^{2}}),
\]
where
\begin{align*}
A_{\varepsilon_{n}}^{2}  &  =\int_{\mathbb{R}^{3}}\rho_{0}\left(  \left\vert
\frac{\varepsilon_{n}}{\rho_{0}r}\varphi^{\prime}(\frac{r-r_{0}}%
{\varepsilon_{n}})\psi(\frac{z-z_{0}}{\varepsilon_{n}^{2}})\right\vert
^{2}+\left\vert \frac{1}{\rho_{0}r}\varphi(\frac{r-r_{0}}{\varepsilon_{n}%
})\psi^{\prime}(\frac{z-z_{0}}{\varepsilon_{n}^{2}})\right\vert ^{2}\right)
dx\\
&  =2\pi\varepsilon_{n}^{3}\int_{-1}^{1}\int_{-1}^{1}\frac{\left(
\varepsilon_{n}^{2}\left\vert \varphi^{\prime}(t)\psi(s)\right\vert
^{2}+\left\vert \varphi(t)\psi^{\prime}(s)\right\vert ^{2}\right)  }{\rho
_{0}r|_{(r,z)=(\varepsilon t+r_{0},\varepsilon_{n}^{2}s+z_{0})}}dtds=O\left(
\varepsilon_{n}^{3}\right)  .
\end{align*}
Then $\Vert\delta v^{\varepsilon_{n}}\Vert_{Y}=1$ and $\delta v^{\varepsilon
_{n}}\in Y_{2}$ owing to
\[
\delta\rho^{\varepsilon_{n}}=B_{1}A\delta v^{\varepsilon_{n}}=\frac{1}%
{r}\partial_{r}(r\rho_{0}\delta v_{r}^{\varepsilon_{n}})+\partial_{z}(\rho
_{0}\delta v_{z}^{\varepsilon_{n}})=0.
\]
We will show that $\left\{  \delta v^{\varepsilon_{n}}\right\}  $ is a Weyl's
sequence for the operator $\tilde{\mathbb{L}}_{2}\ $and therefore $\lambda
\in\sigma_{ess}(\tilde{\mathbb{L}}_{2})$.

First, we check that $\delta v^{\varepsilon_{n}}$ converge to $0$ weakly in
$Y_{2}$. For any $\xi\in Y_{2}$, since $\delta v^{\varepsilon_{n}}$ is
supported in $\Omega_{\varepsilon_{n}}=\left\{  \left\vert r-r_{0}\right\vert
\leq\varepsilon_{n},\left\vert z-z_{0}\right\vert \leq\varepsilon_{n}%
^{2}\right\}  $, we have
\[
\left\vert \langle\delta v^{\varepsilon_{n}},\xi\rangle\right\vert
\leq\left\Vert \delta v^{\varepsilon_{n}}\right\Vert _{Y}\left(  2\pi
\int_{r_{0}-\varepsilon_{n}}^{r_{0}+\varepsilon_{n}}\int_{z_{0}-\varepsilon
_{n}^{2}}^{z_{0}+\varepsilon_{n}^{2}}\rho_{0}|\xi|^{2}rdrdz\right)  ^{\frac
{1}{2}}\rightarrow0,
\]
when $\varepsilon_{n}\rightarrow0$.

Next, we prove that $(\tilde{\mathbb{L}}_{2}-\lambda)\delta v^{\varepsilon
_{n}}$ converge to $0$ strongly in $Y_{2}$. We write
\[
(\tilde{\mathbb{L}}_{2}-\lambda)\delta v^{\varepsilon_{n}}=\mathbb{P}_{2}%
\begin{pmatrix}
\Upsilon(r)\delta v_{r}^{\varepsilon_{n}}\\
0
\end{pmatrix}
-\lambda\delta v^{\varepsilon_{n}}=\mathbb{P}_{2}%
\begin{pmatrix}
\left(  \Upsilon(r)-\Upsilon(r_{0})\right)  \delta v_{r}^{\varepsilon_{n}}\\
-\Upsilon(r_{0})\delta v_{z}^{\varepsilon_{n}}%
\end{pmatrix}
.
\]
Noticing that $\left\Vert \mathbb{P}_{2}\right\Vert \leq1$, and
\[
\left\Vert \delta v_{z}^{\varepsilon_{n}}\right\Vert _{Y}^{2}=\frac{O\left(
\varepsilon_{n}^{5}\right)  }{A_{\varepsilon_{n}}^{2}}=O\left(  \varepsilon
_{n}^{2}\right)  ,
\]
we have
\begin{align*}
&  \ \ \ \ \ \Vert(\tilde{\mathbb{L}}_{2}-\lambda)\delta v^{\varepsilon_{n}%
}\Vert_{Y}^{2}\\
&  \leq\max_{\left(  r,z\right)  \in\Omega_{\varepsilon_{n}}}\left(
\Upsilon(r)-\Upsilon(r_{0})\right)  ^{2}\left\Vert \delta v_{r}^{\varepsilon
_{n}}\right\Vert _{Y}^{2}+\Upsilon(r_{0})^{2}\left\Vert \delta v_{z}%
^{\varepsilon_{n}}\right\Vert _{Y}^{2}\\
&  \leq\max_{\left(  r,z\right)  \in\Omega_{\varepsilon_{n}}}\left(
\Upsilon(r)-\Upsilon(r_{0})\right)  ^{2}+O\left(  \varepsilon_{n}^{2}\right)
\rightarrow0,\text{ }%
\end{align*}
when $\varepsilon_{n}\rightarrow0$. This shows that $\delta v^{\varepsilon
_{n}}$ is a Weyl's sequence for $\tilde{\mathbb{L}}_{2}$ and $\lambda\in
\sigma_{ess}(\tilde{\mathbb{L}}_{2})$. Thus $\left(  -a,b\right)
\subset\sigma_{ess}(\tilde{\mathbb{L}}_{2})$ which implies $\left[
-a,b\right]  \subset\sigma_{ess}(\tilde{\mathbb{L}}_{2})$ since $\sigma
_{ess}(\tilde{\mathbb{L}}_{2})$ is closed.


\end{proof}

\begin{lemma}
\label{lemma-spectra-L2} $\sigma(\tilde{\mathbb{L}}_{2})=\sigma_{ess}%
(\tilde{\mathbb{L}}_{2})=\text{range}\left(  \Upsilon(r)\right)  =\left[
-a,b\right]  $.
\end{lemma}

\begin{proof}
Fix $\lambda\notin\left[  -a,b\right]  $. For any $u=\left(  u_{r}%
,u_{z}\right)  \in Y_{2}$, we have
\begin{align*}
\lbrack(\tilde{\mathbb{L}}_{2}-\lambda)u,u]  &  =[(\mathbb{L}_{2}%
-\lambda)u,u]\\
&  =[(\Upsilon(r)-\lambda)u_{r},u_{r}]-[\lambda u_{z},u_{z}]\\
&  =\int_{\mathbb{R}^{3}}\rho_{0}(\Upsilon(r)-\lambda)u_{r}^{2}dx+\int
_{\mathbb{R}^{3}}(-\lambda)\rho_{0}u_{z}^{2}dx.
\end{align*}
Since $a>0,b\geq0$, we have
\[
|[(\tilde{\mathbb{L}}_{2}-\lambda)u,u]|\geq c_{1}\left\Vert u\right\Vert
_{Y}^{2},
\]
where $c_{1}=\min\left\{  \left\vert \lambda-b\right\vert ,|a+\lambda
|\right\}  >0$. Thus $\left\Vert \left(  \tilde{\mathbb{L}}_{2}-\lambda
\right)  u\right\Vert \geq c_{1}\left\Vert u\right\Vert _{Y}$, which implies
that $(\tilde{\mathbb{L}}_{2}-\lambda)^{-1}$ is bounded and $\lambda\in
\rho(\tilde{\mathbb{L}}_{2})$. Therefore, $\sigma(\tilde{\mathbb{L}}%
_{2})\subset\left[  -a,b\right]  $. This prove the lemma by combining with
Lemma \ref{rangofconsp}.
\end{proof}

The following proposition gives a complete characterization of the spectra of
$\tilde{\mathbb{L}}$.

\begin{proposition}
\label{prop-spectra-l-tilde}Under the Rayleigh instability condition
(\ref{condition-rayleigh}), it holds:\newline i) $\sigma_{ess}(\tilde
{\mathbb{L}})=\text{range}(\Upsilon(r))=\left[  -a,b\right]  $.\newline ii)
$\sigma(\tilde{\mathbb{L}})\cap(-\infty,-a)$ consists of at most finitely many
negative eigenvalues of finite multiplicity.\newline iii) $\sigma
(\tilde{\mathbb{L}})\cap(b,+\infty)$ consists of a sequence of positive
eigenvalues tending to infinity.
\end{proposition}

\begin{proof}
The conclusion in i) follows from Lemmas \ref{lemma-equ-essential} and
\ref{lemma-spectra-L2}. This implies that any $\lambda\in\sigma(\tilde
{\mathbb{L}})$ in $(-\infty,-a)$ or $(b,+\infty)$ must be a discrete
eigenvalue of finite multiplicity.

Proof of ii): Suppose otherwise. Then there exists an infinite dimensional
eigenspace for negative eigenvalues in $(-\infty,-a)$. We notice that
\[
\tilde{\mathbb{L}}+aI=\mathbb{L}_{1}+\mathbb{L}_{2}+aI\geq\mathbb{L}_{1},
\]
since $\mathbb{L}_{2}+aI$ is nonnegative. It follows that $n^{-}\left(
\mathbb{L}_{1}\right)  =\infty$ since $n^{-}\left(  \tilde{\mathbb{L}%
}+aI\right)  =\infty$. This is in contradiction to that $n^{-}\left(
\mathbb{L}_{1}\right)  \leq n^{-}\left(  L\right)  <\infty$.

Proof of iii): Suppose otherwise. Then there exists an upper bound of
$\sigma(\tilde{\mathbb{L}})$, denoted by $\lambda_{max}\geq b$. Thus
$\tilde{\mathbb{L}}\leq\lambda_{max}I$ which implies that
\[
\mathbb{L}_{1}\leq-\mathbb{L}_{2}+\lambda_{max}I\leq\left(  a+\lambda
_{max}\right)  I.
\]
Consequently the eigenvalues of $\mathbb{L}_{1}$ cannot exceed $a+\lambda
_{max}$. This is in contradiction to the fact that $\mathbb{L}_{1}$ has a
sequence of positive eigenvalues tending to infinity.
\end{proof}

Now we can prove Theorem \ref{th: rayleigh unstable}.

\begin{proof}
[Proof of Theorem \ref{th: rayleigh unstable}]Denote $\mathbb{\pi}_{\lambda}$
$\in L\left(  X\right)  \ \left(  \lambda\in\mathbf{R}\right)  $ to be the
spectral family of the self-adjoint operator $\tilde{\mathbb{L}}$. Let
$\left\{  \mu_{i}\right\}  _{i=1}^{\infty}$ be the eigenvalues of
$\tilde{\mathbb{L}}\mathbb{\ }$in $\left(  b,\infty\right)  $. If
$\sigma(\tilde{\mathbb{L}})\cap(-\infty,-a)\neq\varnothing$, we denote the
eigenvalues in $\left(  -\infty,-a\right)  $ by $\nu_{1}<\cdots<\nu_{K}$ where
$K=\dim\left(  R\left(  \pi_{-a}\right)  \right)  $. For $1\leq i<\infty
,\ 1\leq j\leq K$, let $P_{i}^{+}$ $=\mathbb{\pi}_{\mu_{i}+}-\mathbb{\pi}%
_{\mu_{i}-}\ $and $P_{j}^{-}$ $=\mathbb{\pi}_{\nu_{j}+}-\mathbb{\pi}_{\nu
_{j}-}\ $be the projections to $\ker\left(  \tilde{\mathbb{L}}-\mu
_{i}I\right)  $ and $\ker\left(  \tilde{\mathbb{L}}-\nu_{j}I\right)
\ \ $respectively, and $P_{0}=\mathbb{\pi}_{0+}-\mathbb{\pi}_{0-}$ be the
projection to $\ker\tilde{\mathbb{L}}$. By Proposition
\ref{prop-spectra-l-tilde}, we have
\[
\tilde{\mathbb{L}}=\int\lambda d\pi_{\lambda}=\sum_{i=1}^{\infty}\mu_{i}%
P_{i}^{+}+\sum_{j=1}^{K}\nu_{j}P_{j}^{-}+\int_{-a}^{b}\lambda d\pi_{\lambda}.
\]
For any initial data $\left(  u_{2}\left(  0\right)  ,u_{2t}\left(  0\right)
\right)  \in Z\times Y$, the solution to the second order equation
(\ref{eqn-2nd-order}) can be written as%

\begin{align}
u_{2}\left(  t\right)   &  =\sum_{i=1}^{\infty}\left[  \cos(\sqrt{\mu_{i}%
}t)P_{i}^{+}u_{2}\left(  0\right)  +\frac{1}{\sqrt{\mu_{i}}}\sin(\sqrt{\mu
_{i}}t)P_{i}^{+}u_{2t}\left(  0\right)  \right] \label{formula-u-2-t}\\
&  +\sum_{j=1}^{K}\left[  \cosh\left(  \sqrt{-\nu_{j}}t\right)  P_{j}^{-}%
u_{2}\left(  0\right)  +\frac{1}{\sqrt{-\nu_{j}}}\sinh\left(  \sqrt{-\nu_{j}%
}t\right)  P_{j}^{-}u_{2t}\left(  0\right)  \right] \nonumber\\
&  +\int_{0}^{b}\cos(\sqrt{\lambda}t)d\pi_{\lambda}u_{2}(0)+\int_{0}^{b}%
\frac{1}{\sqrt{\lambda}}\sin(\sqrt{\lambda}t)d\pi_{\lambda}u_{2t}%
(0)\nonumber\\
&  +\int_{-a}^{0}\cosh(\sqrt{-\lambda}t)d\pi_{\lambda}u_{2}(0)+\int_{-a}%
^{0}\frac{1}{\sqrt{-\lambda}}\sinh(\sqrt{-\lambda}t)d\pi_{\lambda}%
u_{2t}(0)\nonumber\\
&  +P_{0}u_{2}(0)+tP_{0}u_{2t}(0).\nonumber
\end{align}
If $\sigma(\tilde{\mathbb{L}})\cap(-\infty,-a)=\varnothing$, the solution
$u_{2}\left(  t\right)  $ is obtained by removing the second term above.

Denote the minimum of $\lambda\in\sigma(\tilde{\mathbb{L}})$ by $\eta_{0}$,
that is,
\begin{align*}
\eta_{0}  &  =\min_{\Vert\psi\Vert_{Y}=1}[\tilde{\mathbb{L}}\psi
,\psi]\label{eta}\\
&  =%
\begin{cases}
-a,\quad\text{ if }\sigma(\tilde{\mathbb{L}})\cap(-\infty,-a)=\varnothing,\\
\nu_{1},\text{ \ if }\sigma(\tilde{\mathbb{L}})\cap(-\infty,-a)=\{\nu
_{1}<\cdots<\nu_{K}\}.
\end{cases}
\end{align*}
By the formula (\ref{formula-u-2-t}), it is easy to see that $\Vert
u_{2}(t)\Vert_{Y}\lesssim e^{\sqrt{-\eta_{0}}t}$ for $t>0$. To estimate $\Vert
u_{2}(t)\Vert_{Z}$, we note that by Lemma \ref{lemma-lower-bound-quadratic}
\begin{equation}
\Vert u_{2}\Vert_{Z}^{2}\approx\left[  \mathbb{\tilde{L}}u_{2},u_{2}\right]
+2m\Vert u_{2}\Vert_{Y}^{2}. \label{equivalence-norm-Z}%
\end{equation}
By using (\ref{formula-u-2-t}), we have
\begin{align*}
\left[  \mathbb{\tilde{L}}u_{2}\left(  t\right)  ,u_{2}\left(  t\right)
\right]   &  \lesssim\sum_{i=1}^{\infty}\left[  \mu_{j}\left\Vert P_{i}%
^{+}u_{2}\left(  0\right)  \right\Vert _{Y}^{2}+P_{i}^{+}\left\Vert
u_{2t}\left(  0\right)  \right\Vert _{Y}^{2}\right] \\
&  +e^{-\eta_{0}t}\sum_{j=1}^{K}\left[  \left\Vert P_{j}^{-}u_{2}\left(
0\right)  \right\Vert _{Y}^{2}+\left\Vert P_{j}^{-}u_{2t}\left(  0\right)
\right\Vert _{Y}^{2}\right] \\
&  +\int_{0}^{b}d\left(  \pi_{\lambda}u_{2}(0),u_{2}(0)\right)  +\int_{0}%
^{b}d\left(  \pi_{\lambda}u_{2t}(0),u_{2t}(0)\right) \\
&  +e^{-\eta_{0}t}\left[  \int_{-a}^{0}d\left(  \pi_{\lambda}u_{2}%
(0),u_{2}(0)\right)  +\int_{-a}^{0}d\left(  \pi_{\lambda}u_{2t}(0),u_{2t}%
(0)\right)  \right] \\
&  \lesssim e^{-\eta_{0}t}\left(  \left(  \mathbb{\tilde{L}}u_{2}\left(
0\right)  ,u_{2}\left(  0\right)  \right)  +m\Vert u_{2}\left(  0\right)
\Vert_{Y}^{2}+\Vert u_{2t}\left(  0\right)  \Vert_{Y}^{2}\right) \\
&  \lesssim e^{-\eta_{0}t}\left(  \Vert u_{2}\left(  0\right)  \Vert_{Z}%
^{2}+\Vert u_{2t}\left(  0\right)  \Vert_{Y}^{2}\right)  .
\end{align*}
This implies
\[
\Vert u_{2}\left(  t\right)  \Vert_{Z}\lesssim e^{\sqrt{-\eta_{0}}t}\left(
\Vert u_{2}\left(  0\right)  \Vert_{Z}+\Vert u_{2t}\left(  0\right)  \Vert
_{Y}\right)  ,
\]
by using (\ref{equivalence-norm-Z}) and the estimate for $\Vert u_{2}%
(t)\Vert_{Y}$. Since
\begin{align*}
u_{2t}\left(  t\right)   &  =\sum_{i=1}^{\infty}\left[  -\sqrt{\mu_{i}}%
\sin(\sqrt{\mu_{i}}t)P_{i}^{+}u_{2}\left(  0\right)  +\cos(\sqrt{\mu_{i}%
}t)P_{i}^{+}u_{2t}\left(  0\right)  \right] \\
&  +\sum_{j=1}^{K}\left[  \sqrt{-\nu_{j}}\sinh\left(  \sqrt{-\nu_{j}}t\right)
P_{j}^{-}u_{2}\left(  0\right)  +\cosh\left(  \sqrt{-\nu_{j}}t\right)
P_{j}^{-}u_{2t}\left(  0\right)  \right] \\
&  +\int_{0}^{b}-\sqrt{\lambda}\sin(\sqrt{\lambda}t)d\pi_{\lambda}%
u_{2}(0)+\int_{0}^{b}\cos(\sqrt{\lambda}t)d\pi_{\lambda}u_{2t}(0)\\
&  +\int_{-a}^{0}\sqrt{-\lambda}\sinh(\sqrt{-\lambda}t)d\pi_{\lambda}%
u_{2}(0)+\int_{-a}^{0}\cosh(\sqrt{-\lambda}t)d\pi_{\lambda}u_{2t}%
(0)+P_{0}u_{2t}(0),
\end{align*}
by similar estimates as above for $\Vert u_{2}\left(  t\right)  \Vert_{Z}%
,\ $we obtain
\[
\left\Vert u_{2t}\left(  t\right)  \right\Vert _{Y}\lesssim e^{\sqrt{-\eta
_{0}}t}\left(  \Vert u_{2}\left(  0\right)  \Vert_{Z}+\Vert u_{2t}\left(
0\right)  \Vert_{Y}\right)  .
\]
This finishes the proof of the upper bound estimate
(\ref{estimate-upper-growth}). It is straightforward to show that the energy
$E(u_{2},u_{2t})$ defined in (\ref{energy}) is conserved for solutions of
(\ref{eqn-2nd-order}).

Next, we prove the lower bound estimate (\ref{estimate-lower-growth}) in two cases.

Case 1: $\sigma(\tilde{\mathbb{L}})\cap(-\infty,-a)\neq\varnothing$. We choose
$u_{2}(0)=\psi_{1}$ and $u_{2t}(0)=\sqrt{-\nu_{1}}\psi_{1}$ where $\psi_{1}\in
Z$ is the eigenfunction of $\tilde{\mathbb{L}}$ corresponding to the smallest
eigenvalue $\nu_{1}$ in $(-\infty,-a)$. Then
\[
\left(  u_{2}(t),u_{2t}(t)\right)  =\left(  e^{\sqrt{-\nu_{1}}t}\psi_{1}%
,\sqrt{-\nu_{1}}e^{\sqrt{-\nu_{1}}t}\psi_{1}\right)  ,
\]
which clearly implies $\left\Vert u_{2}(t)\right\Vert _{Y}\gtrsim
e^{\sqrt{-\eta_{0}}t}\left\Vert u_{2}\left(  0\right)  \right\Vert _{Z}$.

Case 2: $\sigma(\tilde{\mathbb{L}})\cap(-\infty,-a)=\varnothing$. Since
$\sigma_{ess}(\tilde{\mathbb{L}})=[-a,b]$, for any $\varepsilon>0$ small there
exists a nonzero function $\phi\in R(\pi_{-a+\varepsilon}-\pi_{-a})\subset Z$.
Choose the initial data $u_{2}(0)=\phi$ and $u_{2t}(0)=0$. Then the solution
$u_{2}\left(  t\right)  $ for the equation (\ref{eqn-2nd-order}) is given by
\[
u_{2}(t)=\int_{-a}^{-a+\varepsilon}\cosh(\sqrt{-\lambda}t)d\pi_{\lambda}\phi.
\]
Thus
\begin{align*}
\Vert u_{2}(t)\Vert_{Y}^{2}  &  =\int_{-a}^{-a+\varepsilon}\cosh^{2}%
(\sqrt{-\lambda}t)d\left(  \pi_{\lambda}\phi,\phi\right)  \gtrsim
e^{\sqrt{-\eta_{0}+\varepsilon}t}\int_{-a}^{-a+\varepsilon}d\left(
\pi_{\lambda}\phi,\phi\right) \\
&  \gtrsim e^{\sqrt{-\eta_{0}+\varepsilon}t}\left\Vert \phi\right\Vert _{Z}.
\end{align*}
This finishes the proof of the theorem.
\end{proof}

\begin{remark}
By Theorem \ref{th: rayleigh unstable}, the maximal growth rate of unstable
rotating stars can be due to either discrete or continuous spectrum. Consider
a family of slowly rotating stars $\left(  \rho_{\varepsilon},\vec
{v_{\varepsilon}}=\varepsilon r\omega_{0}\left(  r\right)  \mathbf{e}_{\theta
}\right)  \ $near a non-rotating star $\left(  \rho_{0}\left(  \left\vert
x\right\vert \right)  ,\vec{v_{0}}=\vec{0}\right)  $ with $\omega_{0}\left(
r\right)  $ satisfying the Rayleigh instability condition
(\ref{condition-rayleigh}). If the non-rotating star is linearly stable, then
for sufficiently small $\varepsilon$, the linear instability of $\left(
\rho_{\varepsilon},\vec{v_{\varepsilon}}\right)  $ is due to the continuous
spectrum. On the other hand, if the the non-rotating star is linearly
unstable, then for sufficiently small $\varepsilon$, $\left(  \rho
_{\varepsilon},\vec{v_{\varepsilon}}\right)  $ remains unstable and the
maximal growth rate is due to the discrete eigenvalue perturbed from the
unstable eigenvalue of the non-rotating star.
\end{remark}

\begin{remark}
In \cite{LNR1970}, Lebovitz indicated that for slowly rotating stars with any
angular velocity profile $\omega_{0}(r)$, discrete unstable modes cannot be
perturbed from neutral modes of non-rotating stars. More precisely, Lebovitz
showed the stabilizing influence of rotation on the fundamental mode
(corresponding to the first eigenvalue of the operator $\mathbb{\tilde{L}}$ in
(\ref{eqn-2nd-order})) even when $\omega_{0}(r)$ does not satisfy the Rayleigh
stability condition. However, this does not imply the stability of the
rotating stars since the unstable continuous spectrum was not considered in
\cite{LNR1970}.
\end{remark}

\textbf{Acknowledgments:} This work is supported partly by the NSF grants
DMS-1715201 and DMS-2007457 (Lin) and the China Scholarship Council
No.201806310066(Wang). 



\bigskip

\begin{thebibliography}{99}                                                                                               %


\bibitem {AG1991}G.~Auchmuty.
\newblock {The global branching of rotating stars}.
\newblock {\em Arch. Rat. Mech. Anal.} 114:179--194, 1991.

\bibitem {AGBR1971}G.~Auchmuty and R.~Beals.
\newblock {Variational solutions of some nonlinear free boundary problems}.
\newblock {\em Arch. Rat. Mech. Anal.} 43:255--271, 1971.

\bibitem {BEL2015}A.~Balinsky, W.~Evans and R.~Lewis.
\newblock {The Analysis and Geometry of Hardy's Inequality}.
\newblock {\em Physical Review D Particles Fields.} 81(8):1014--1025, 2015.

\bibitem {BB1974}S.~I. Bisnovaty-Kogan, G. S.~Blinnikov.
\newblock {Static criteria for stability of arbitrarily rotating stars}.
\newblock {\em Astronomy $\&$ Astrophysics.} 31(4), 1974.

\bibitem {CF1980}A.~Caffarelli, L.~Friedman.
\newblock {The shape of axi-symmetric rotating fluid}.
\newblock {\em Journal of Functional Analysis.} 694:109--142, 1980.

\bibitem {CS1939}S.~Chandrasekhar.
\newblock {\em {Introduction to the Stellar Structure.}} \newblock University
of Chicago Press, 1939.

\bibitem {Chandra69}S.~Chandrasekhar.
\newblock {\em {Ellipsoidal Figures of Equilibrium}}. \newblock Yale
University Press, 1969.

\bibitem {ChandraLebovitz68}S.~Chandrasekhar and N. R. ~Lebovitz.
\newblock {\em {The Pulsations and the Dynamical Stability of Gaseous Masses in Uniform Rotation}}.
\newblock {\em Astrophysical Journal.} 152(1):267--291, 1968.

\bibitem {CL1994}S.~Chanillo and Y.~Li.
\newblock {On diameters of uniformly rotating stars}.
\newblock {\em Communications in Mathematical Physics.} 166(2):417--430, 1994.

\bibitem {FT1980}A.~Friedman and B.~Turkington.
\newblock {Asymptotic estimates for an axisymmetric rotating fluid}.
\newblock {\em Journal of Functional Analysis.} 37(2):136--163, 1980.

\bibitem {FT1981}A.~Friedman and B.~Turkington.
\newblock {Existence and dimensions of a rotating white dwarf}.
\newblock {\em Journal of Differential Equations.} 42(3):414--437, 1981.

\bibitem {FI1988}J.~Friedman, J.~Ipser, and R.~Sorkin.
\newblock {Turning-Point method for axisymmetric stability of rotating
relativistic stars}. \newblock {\em Astrophysical Journal.} 325:722--724, 1988.

\bibitem {Hazelhurst1994}J.~Hazelhurst. \newblock The stabilizing effect of
rotation. \newblock {\em Astronomy and Astrophysics.} 219:181--184, 1994.

\bibitem {H1994}U.~Heilig.
\newblock {On Lichtenstein's analysis of rotating Newtonian stars}.
\newblock {\em In Annales de l'IHP Physique th\'{e}orique.} 60:457--487, 1994.

\bibitem {HU2003}J. M.~Heinzle. and C. ~Uggla.
\newblock {Newtonian stellar models}. \newblock {\em Ann. Physics.}
308:18--61, 2003.

\bibitem {JJ2014}J.~Jang.
\newblock {Nonlinear Instability Theory of Lane-Emden stars}.
\newblock {\em Comm. Pure Appl. Math.} 67:1418--1465, 2014.

\bibitem {JJ2017}J.~Jang and T.~Makino.
\newblock {On slowly rotating axisymmetric solutions of the Euler-Poisson
equations}. \newblock {\em Arch. Ration. Mech. Anal.} 225:873--900, 2017.

\bibitem {JJMT2019}J.~Jang and T.~Makino.
\newblock {On rotating axisymmetric solutions of the Euler-Poisson equations}.
\newblock {\em Journal of Differential equations.} 266(7):3942--3972, 2019.

\bibitem {JJ2019}J.~Jang and T.~Makino.
\newblock {On rotating axisymmetric solutions of the Euler-Poisson equations}.
\newblock {\em Journal of Differential equations.} 266:3942--3972, 2019.

\bibitem {JM2020}J.~Jang and T.~Makino.
\newblock {Linearized Analysis of Barotropic Perturbations around Spherically
Symmetric Gaseous Stars Governed by the Euler-Poisson Equations}.
\newblock {\em Journal of Mathematical Physics.} 61(5):051508,42, 2020.

\bibitem {JWS2013}W.~S. Jardetzky.
\newblock {\em {Theories of figures of celestial bodies}}.
\newblock Interscience Publishers, 1958.

\bibitem {Kahler1994}H.~K\"{a}hler. \newblock Rotational effects on stellar
structure and stability. \newblock {\em Astronomy and Astrophysics.}
288:191--203, 1994.

\bibitem {Katobook1995}T.~Kato.
\newblock {\em {Perturbation theory for linear operators. Reprint of the 1980
edition.}} \newblock Springer-Verlag, Berlin, 1995.

\bibitem {LNR1970}N.~R. Lebovitz.
\newblock {The Effect of an Arbitrary Law of Slow Rotation on the Oscillations
and the Stability of Gaseous Masses}.
\newblock {\em The Astrophysical Journal.} 160:701, 1970.

\bibitem {Ledoux1945}P.~Ledoux.
\newblock {On the Radial Pulsation of Gaseous Stars}.
\newblock {\em The Astrophysical Journal.} 102(2):143, 1945.

\bibitem {LYY1991}Y.~Li. \newblock {On uniformly rotating stars}.
\newblock {\em Arch. Rat. Mech. Anal.} 115:367--393, 1991.

\bibitem {Lich1933}L.~Lichtenstein.
\newblock {Untersuchungen \"{u}ber die Gleichgewichtsfiguren rotierender
Fl\"{u}ssigkeiten, deren Teilchen einander nach dem Newtonschen Gesetze
anziehen}. \newblock {\em Math. Z.} 36:481--562, 1933.

\bibitem {LS1997}S.S. Lin.
\newblock {Stability of gaseous stars in spherically symmetric motions}.
\newblock {\em SIAM J. Math. Anal.} 28(3):539--569, 1997.

\bibitem {lin-zeng-hamiltonian}Z.~Lin and C.~Zeng.
\newblock {Instability, index theorem, and exponential trichotomy for Linear
Hamiltonian PDEs}. \newblock {\em Mem. Amer. Math. Soc.} 275 (1347), 2022.

\bibitem {LZ2019}Z.~Lin and C.~Zeng.
\newblock {Separable Hamiltonian PDEs and Turning point principle for stability
of gaseous stars}. \newblock {\em Comm. Pure. Appl. Math.}, accepted 2022.

\bibitem {LS2004}T.~Luo and J.~Smoller.
\newblock {Rotating fluids with self-gravitation in bounded domains}.
\newblock {\em Arch. Rat. Mech. Anal.} 173(3):345--377, 2004.

\bibitem {LTSJ2008}T.~Luo and J.~Smoller.
\newblock {Nonlinear Dynamical Stability of Newtonian Rotating and Non-rotating
White Dwarfs and Rotating Supermassive Stars}.
\newblock {\em Communications in Mathematical Physics.} 284(2):425--457, 2008.

\bibitem {LS2009}T.~Luo and J.~Smoller.
\newblock {Existence and non-linear stability of rotating star solutions of the
compressible Euler-Poisson equations}.
\newblock {\em Arch. Ration. Mech. Anal.} 191(3):447--496, 2009.

\bibitem {LTSJ2011}T.~Luo and J.~Smoller.
\newblock {\em {On the Euler-Poisson equations of self-gravitating compressible
fluids.Nonlinear conservation laws and applications}}, volume 153.
\newblock Springer, New York, 2011.

\bibitem {OM1968}J. ~P. Ostriker and W. ~K. Mark.
\newblock {\em {Rapidly rotating stars. I. The self-consistent-field method}.}
\newblock {\em The Astrophysical Journal.} 1968, 151:1075--1088.

\bibitem {PJKA2003}J.~E. Pringle and A.~R. King.
\newblock {\em {Astrophysical Flows}}. \newblock Cambridge university press, 2003.

\bibitem {Rayleigh1880}L. Rayleigh.
\newblock {On the stability, or instability, of certain fluid motions}.
\newblock {\em Proc. London Math. Soc.} 11:57--70, 1880.

\bibitem {Sidorov1982}K.~A. Sidorov.
\newblock {Influence of rotation and a binary companion on the frequency of the
radial pulsations of a homogeneous star}. \newblock {\em Astrophysics.}
18(1):90--96, 1982.

\bibitem {Sidorov1981}K.~A. Sidorov.
\newblock {Structure and oscillations of rotating polytropes}.
\newblock {\em Astrophysics.} 17(4):427--436, 1982.

\bibitem {Stahler83}S.~W. Stahler.
\newblock {The Equilibria of Rotating Isothermal Clouds - Part Two - Structure
and Dynamical Stability}. \newblock {\em Astrophysical Journal.} 268:165--184, 1983.

\bibitem {SW2017}W.~A. Strauss and Y.~Wu.
\newblock {Steady States of Rotating Stars and Galaxies}.
\newblock {\em SIAM Journal on Mathematical Analysis.} 49(6):4865--4914, 2017.

\bibitem {SW2019}W.~A Strauss and Y.~Wu. \newblock {Rapidly Rotating Stars}.
\newblock {\em Communications in Mathematical Physics.} 368(2):701--721, 2019.

\bibitem {SW2019R}W.~A Strauss and Y.~Wu.
\newblock {Rapidly rotating white dwarfs}. \newblock {\em Nonlinearity.}
33(9):4783--4798, 2020.

\bibitem {TRY2011}K.~Takami, L.~Rezzolla, and S.~Yoshida.
\newblock {A quasi-radial stability criterion for rotating relativistic stars}.
\newblock {\em Monthly Notices of the Royal Astronomical Society Letters.}
416(1):L1--L5, 2011.

\bibitem {TJ1978}J.~L. Tassoul. \newblock {\em {Theory of Rotating Stars}}.
\newblock Princeton University Press, 1978.

\bibitem {TJL2000}J.~L. Tassoul. \newblock {\em {Stellar rotation}}.
\newblock Cambridge university press, 2000.
\end{thebibliography}
\end{document}